\pdfoutput=1

\documentclass[10pt,a4paper]{article}

\usepackage[utf8]{inputenc}
\usepackage[british]{babel}

\usepackage[left=2cm,right=2cm,top=2cm,bottom=2cm]{geometry}

\usepackage{cancel}
\usepackage{amstext}
\usepackage{amssymb}
\usepackage{amsmath}
\usepackage{amsfonts}
\usepackage{mathrsfs}
\usepackage{bm}
\usepackage{nicefrac, xfrac}
\usepackage{upgreek}
\usepackage{booktabs}
\usepackage{array}
\usepackage{xfrac}

\usepackage{xcolor}
\usepackage{color}

\usepackage{graphicx}
\usepackage{pgf,tikz,pgfplots}
\pgfplotsset{compat=1.15}
\usetikzlibrary{arrows, arrows.meta, calc, math, shapes, positioning, fit, backgrounds}
\definecolor{sweep1}{HTML}{4899C3}
\definecolor{sweep2}{HTML}{3391C6}
\definecolor{sweep3}{HTML}{247BB8}
\definecolor{sweep4}{HTML}{1A66A3}
\definecolor{sweep5}{HTML}{125089}
\definecolor{sweep6}{HTML}{0C3B6C}
\definecolor{sweep7}{HTML}{062A52}

\definecolor{relC}{HTML}{1A66A3}
\definecolor{effC}{HTML}{C25E1B}

\pgfplotsset{
    sweep axis/.style={
        width=\linewidth, height=5.6cm,
        ymin=0, ymax=1,
        tick align=outside, tick pos=left,
        axis line style={gray!55, line width=0.4pt},
        tick style={gray!55},
        grid=major,
        grid style={gray!22, line width=0.3pt},
        tick label style={font=\footnotesize},
        label style={font=\footnotesize},
        ylabel near ticks, xlabel near ticks,
    },
    profile axis/.style={
        sweep axis,
        xlabel={$\tau$, performance ratio in iterations},
        ylabel={$\pi(\tau)$},
        every axis plot/.append style={const plot, no marks, line width=0.9pt},
        legend cell align=left,
        legend columns=3,
        legend style={draw=gray!45, fill=white, font=\footnotesize,
                      column sep=6pt, inner sep=3pt,
                      at={(0.5,-0.33)}, anchor=north},
        reverse legend,
    },
    sweep cycle 7/.style={cycle list={
        {sweep1, dash pattern=on 0.6pt off 1.6pt},
        {sweep2, dash pattern=on 1.2pt off 1.6pt},
        {sweep3, dash pattern=on 2pt off 1.6pt},
        {sweep4, dash pattern=on 3pt off 1.6pt},
        {sweep5, dash pattern=on 4.5pt off 1.6pt},
        {sweep6, dash pattern=on 7pt off 1.6pt},
        {sweep7, solid}}},
    sweep cycle 6/.style={cycle list={
        {sweep1, dash pattern=on 0.6pt off 1.6pt},
        {sweep2, dash pattern=on 1.2pt off 1.6pt},
        {sweep3, dash pattern=on 2pt off 1.6pt},
        {sweep5, dash pattern=on 3.5pt off 1.6pt},
        {sweep6, dash pattern=on 6pt off 1.6pt},
        {sweep7, solid}}},
    rate axis/.style={
        sweep axis,
        ylabel={fraction of $196$ problems},
        every axis plot/.append style={line width=0.9pt, mark size=2pt},
        legend cell align=left,
        legend columns=1,
        legend style={draw=gray!45, fill=white, font=\footnotesize,
                      row sep=0.5pt, inner sep=3pt,
                      at={(0.5,-0.33)}, anchor=north},
    },
    rel line/.style={relC, mark=*,      mark options={fill=relC, draw=white, line width=0.3pt}},
    eff line/.style={effC, mark=square*, dashed,
                     mark options={solid, fill=effC, draw=white, line width=0.3pt}},
    ref line/.style={gray!65, line width=0.6pt, dash pattern=on 2pt off 2pt},
}

\usepgfplotslibrary{statistics}

\definecolor{anchorDelta}{HTML}{0072B2}   
\definecolor{anchorStep}{HTML}{D55E00}    
\definecolor{anchorCrit}{HTML}{009E73}    
\definecolor{anchorRetro}{HTML}{8C4799}   

\pgfplotsset{
    comparison cycle/.style={cycle list={
        {anchorDelta, solid},
        {anchorStep,  solid},
        {anchorCrit,  solid},
        {anchorCrit,  dash pattern=on 4pt off 2pt},
        {anchorCrit,  dash pattern=on 4pt off 2pt on 1pt off 2pt},
        {anchorRetro, solid},
        {anchorStep,  dash pattern=on 4pt off 2pt},
        {anchorCrit,  dash pattern=on 1pt off 1.6pt}}},
    comparison legend/.style={
        legend cell align=left,
        legend columns=2,
        legend style={draw=gray!45, fill=white, font=\footnotesize,
                      column sep=6pt, inner sep=3pt,
                      at={(0.5,-0.33)}, anchor=north}},
    scatter axis/.style={
        width=\linewidth, height=5.6cm,
        tick align=outside, tick pos=left,
        axis line style={gray!55, line width=0.4pt},
        tick style={gray!55},
        grid=major, grid style={gray!22, line width=0.3pt},
        tick label style={font=\footnotesize},
        label style={font=\footnotesize},
        ylabel near ticks, xlabel near ticks,
        every axis plot/.append style={only marks, mark size=2.4pt},
        legend cell align=left,
        legend columns=2,
        legend style={draw=gray!45, fill=white, font=\footnotesize,
                      column sep=6pt, inner sep=3pt,
                      at={(0.5,-0.33)}, anchor=north}},
    anchor mark/.style n args={2}{color=#1, mark=#2,
        mark options={fill=#1, draw=white, line width=0.4pt}},
}

\definecolor{boxInk}{HTML}{1A66A3}

\pgfplotsset{
    box axis/.style={
        width=\linewidth, height=6.0cm,
        boxplot/draw direction=y,
        tick align=outside, tick pos=left,
        axis line style={gray!55, line width=0.4pt},
        tick style={gray!55},
        ymajorgrids=true,
        grid style={gray!22, line width=0.3pt},
        xtick={1,2,3,4,5,6,7,8},
        x tick label style={rotate=45, anchor=north east, font=\scriptsize,
                            inner sep=1pt},
        tick label style={font=\footnotesize},
        label style={font=\footnotesize},
        ylabel near ticks,
        every axis plot/.append style={
            boxInk, line width=0.6pt, solid,
            boxplot/every median/.style={boxInk, line width=1.1pt},
            boxplot/every box/.style={draw=boxInk, fill=boxInk!12},
            boxplot/every whisker/.style={boxInk},
            mark=o, mark size=1.4pt, mark options={boxInk, line width=0.4pt}}},
}

\tikzset{
    ref rule/.style={gray!70, line width=0.6pt, dash pattern=on 2pt off 2pt},
    pred rule/.style={anchorStep, line width=0.8pt, dash pattern=on 4pt off 2pt},
}

\usepackage{wrapfig}

\usepackage[colorlinks=true]{hyperref}

\usepackage{ifpdf}

\usepackage[babel=true]{csquotes}
\MakeOuterQuote{"}
\usepackage{textcomp}

\usepackage{graphicx}
\usepackage{longtable}
\usepackage{caption}
\usepackage{subcaption}

\usepackage{svg}

\usepackage{url}
\usepackage{float}
\usepackage{multicol}

\usepackage[framed, amsthm, amsmath, thmmarks]{ntheorem}

\definecolor{mygray}{gray}{0.95}

\newcommand{\theoremcolorbox}[2]{\def\theoremframecommand{\fcolorbox{#1}{#2}}}
\theoremcolorbox{black}{mygray}
\newtheorem{theorem}{Theorem}[section]

\newtheorem{corollary}[theorem]{Corollary}
\newtheorem{lemma}[theorem]{Lemma}
\newtheorem{proposition}[theorem]{Prop.}

\newcommand{\claimno}[1]{\refstepcounter{claimP}\theclaimP\label{#1}}

\theoremstyle{remark}
\newtheorem{remark}{Remark}

\makeatletter
\NewDocumentEnvironment{assum}{m O{}}{%
  \par\addvspace{\topsep}%
  \noindent
  \def\@currentlabel{#1}%
  {\normalfont\bfseries Assumption #1%
    \if\relax\detokenize{#2}\relax\else
      \ \normalfont(#2)%
    \fi
    \normalfont\bfseries.}%
  \nobreak\hspace{.5em}\ignorespaces
}{%
  \par\addvspace{\topsep}%
}
\NewDocumentEnvironment{cond}{m O{}}{%
  \par\addvspace{\topsep}%
  \noindent
  \def\@currentlabel{#1}%
  {\normalfont\bfseries Condition #1%
    \if\relax\detokenize{#2}\relax\else
      \ \normalfont(#2)%
    \fi
    \normalfont\bfseries.}%
  \nobreak\hspace{.5em}\ignorespaces
}{%
  \par\addvspace{\topsep}%
}
\makeatother

\makeatletter
\NewDocumentEnvironment{update}{m O{}}{%
  \par\addvspace{\topsep}%
  \noindent
  \def\@currentlabel{#1}%
  {\normalfont\bfseries Update #1%
    \if\relax\detokenize{#2}\relax\else
      \ \normalfont(#2)%
    \fi
    \normalfont\bfseries.}%
  \nobreak\hspace{.5em}\ignorespaces
}{%
  \par\addvspace{\topsep}%
}
\makeatother

\usepackage{enumerate}
\usepackage[shortlabels]{enumitem}

\numberwithin{equation}{section}

\usepackage[ruled,vlined]{algorithm2e}
\allowdisplaybreaks
\usepackage{algorithmicx, algpseudocode}
\usepackage[capitalize]{cleveref}

\newcommand{\rme}{{\mathrm{e}}}

\newcommand{\R}{{\mathbb{R}}} 
 
\newcommand{\reals}{\mathbb{R}}

\DeclareSymbolFont{bbold}{U}{bbold}{m}{n}
\DeclareSymbolFontAlphabet{\mathbbold}{bbold}

\newcommand{\calB}{{\mathcal{B}}}

\newcommand{\calK}{{\mathcal{K}}}

\newcommand{\calM}{{\mathcal{M}}}

\newcommand{\calO}{{\mathcal{O}}}
\newcommand{\calP}{{\mathcal{P}}}

\newcommand{\calR}{{\mathcal{R}}}
\newcommand{\calS}{{\mathcal{S}}}

\newcommand{\defeq}{:=}

\newcommand{\abs}[1]{\left \vert #1 \right \vert}

\newcommand{\norm}[1]{\left \lVert #1 \right \rVert}

\newcommand*\Rp{\reals^p}

\newcommand*\mykappa[1]{ \kappa_{\scriptscriptstyle \text{#1}} }

\providecommand{\defeq}{\mathrel{:=}}
\providecommand{\kbar}{\overline{\kappa}}

\providecommand{\calK}{\mathcal{K}}

\newcommand{\Rdelta}{\textsf{R-delta}}
\newcommand{\Rstep}{\textsf{R-step}}
\newcommand{\Rdfo}{\textsf{R-DFO}}
\newcommand{\Rgrad}{\textsf{R-grad}}
\newcommand{\Rrtr}{\textsf{R-RTR}}
\newcommand{\Rgrtr}{\textsf{R-gRTR}}
\newcommand{\RAdaptiveGrad}{\textsf{R-adaptive-grad}}
\newcommand{\RAdaptstep}{\textsf{R-adaptive-step}}
\newcommand{\Rdeltastep}{\textsf{R-delta-step}}

\newcommand{\gone}{\gamma_1}
\newcommand{\gtwo}{\gamma_2}
\newcommand{\gthree}{\gamma_3}

\newif\ifproposals \proposalstrue
\DeclareRobustCommand{\prop}[1]{\ifproposals\textcolor{blue}{#1}\else#1\fi}

\newcommand{\thisdoc}{paper}

\newcommand{\PartI}{\citet{Jeremy2026-I}}
\newcommand{\PartII}{\citet{Jeremy2026-II}}

\newcommand{\PartIref}{Part~I~\citep{Jeremy2026-I}}
\newcommand{\PartIIref}{Part~II~\citep{Jeremy2026-II}}

\usepackage[square,numbers]{natbib}
\setcitestyle{semicolon, aysep={,}}

\title{{\bf A numerical study of radius update rules in trust-region algorithms} }

\author{{\bf J\'{e}r\'{e}my Rieussec} \quad {\bf Fabian Bastin}\\
	University of Montreal, Dept. Computing Science and Operational Research\\
	Email: jeremy.rieussec@umontreal.ca, bastin@iro.umontreal.ca \\
}

\date{\today}

\proposalsfalse  
\begin{document}

\maketitle

\begin{abstract}
  \prop{\PartIref{} and \PartIIref{} establish} which trust-region radius
  update mechanisms \emph{can} fail and under which hypotheses they \emph{must} succeed. Neither says where the
  thresholds sit once a rule runs with real constants, or how the rule interacts with the model
  Hessian and the subproblem solver. This \thisdoc{} turns them into nine falsifiable claims
  and settles what it can, on four designed problems where the theory's quantities are computable
  and on CUTEst, where only aggregate cost is visible. Every mechanism runs in one Julia package
  written for this study, so two runs differing in one component differ in nothing else.

  The rule is worth less than the model and more than nothing. Crossing the mechanisms with four
  model Hessians puts $96.5$ per cent of the variation on the model and $1.6$ on the rule, and
  the rules separate most under the most accurate model rather than least. Uncapped
  $\Rgrad$ is the robust default among the criticality-anchored rules, though the retrospective
  rule is ahead of it on this collection. No constant escapes the side conditions. For every cap there is a minimiser
  at which the constraint stays active, and for every anchoring constant a flatness on which
  escape takes arbitrarily long. The second of these is proved here as well as measured, and the
proof uses no property of the model or of the subproblem solution beyond
$\norm{s_k} \leq \Delta_k$.

  The radius rule governs the \emph{length} of the step, and so the cost of a run. The model
  decides the search direction, and with it whether the run converges at all.
\end{abstract}

\section{Introduction}
\label{sec: introduction}

The trust-region radius update rule is the least standardised component of a trust-region
method. Textbook treatments~\citep{ConnGoulToin00, NoceWrig06} multiply the radius by a constant
chosen from the agreement ratio $\rho_k$, and the literature offers a range of alternatives.
Some rules reset the radius to a multiple of the accepted step~\citep{Hei03, FanPan16}. Others
keep it proportional to a criticality measure~\citep{ConnScheVice09_introduction}, scale it by
the gradient norm~\citep{FanPan16, WangYuan2022, CurtSche20}, or re-judge the previous step with
the model built at the new iterate~\citep{BastMalmMoufToinToma10}. Authors present these as
implementation details, and a reader looking for a reason to prefer one of them finds little.

The two analyses this study rests on~\citep{Jeremy2026-I, Jeremy2026-II} argue instead that they
are structurally different algorithms whose asymptotic behaviour differs. The first isolates three structural conditions on the sequence
$\{\Delta_k\}$ and shows that they suffice for first-order global convergence and
$\calO(\varepsilon^{-2})$ complexity. The second asks what the radius does rather than what it
permits, characterising its asymptotic behaviour, the local dynamics near a nondegenerate
minimiser, a second-order extension through $\tau_k = \max\{\norm{g_k},\, -\lambda_k^-\}$, and
the conditions under which the constraint eventually becomes inactive and the method converges
superlinearly. One axis organises both parts. What decides a rule's behaviour is not whether its
radius vanishes but what the radius is anchored to, the previous radius, the step actually
taken, or a criticality measure.

Together the two parts identify which mechanisms can fail and under which hypotheses they must
succeed. Three questions lie outside their reach. A threshold stated as a function of the
curvature at the solution says nothing about where that threshold sits once a rule runs with
real constants on a problem whose solution nobody knows. A hypothesis stated about the limit
point cannot be checked on a collection assembled precisely because its problems are hard. A
theorem about the radius holds the model Hessian and the subproblem solver fixed, while a user
choosing a solver changes all three at once.

This study answers those questions by experiment, and the first step is to make the theory
answerable. Section~\ref{sec: claims} states nine claims an experiment
can falsify, and records at each one where the restatement changed what was asserted. Two
regimes carry the evidence and neither substitutes for the other. On the designed problems the
curvature at the solution is known in closed form and the limit point can be classified, so one
run is evidence about the mechanism. On the CUTEst collection~\citep{GouldOrbanToint15} none of
those quantities is available and only aggregate cost is visible, so that regime answers the
question a user asks, whether a distinction established on a designed problem is visible in what
the method costs.

Three things are contributed besides the verdicts. Every mechanism is implemented in one Julia
package, \textsf{TrustRegion\-Radius.jl}, written for this study and released with it, in which
the radius rule, the model Hessian, the subproblem solver and the oracle are independent
components of a single iteration, so two runs differing in one of them differ in nothing else.
Four problems are designed so that the quantities the theory is stated about are computable and
the limit point can be classified. And the campaign returns a recommendation, uncapped
$\Rgrad$, worth $59$ problems of $196$ in reliability across the cap sweep.

Section~\ref{sec: package} describes the package and Section~\ref{sec: problems} the problems.
Section~\ref{sec: cost} measures what the mechanisms cost on CUTEst, in sweeps on $\zeta$ and on
$\overline{\mu}$ and in a like-for-like comparison at matched constants.
Section~\ref{sec: inactivity} measures the activity of the constraint and the local order after
it releases. Section~\ref{sec: model vs radius rule} crosses the mechanisms with the model
Hessians and shows how a small constant silences the model.
Sections~\ref{sec: threshold} and~\ref{subsec: flat} return to the designed problems, locating
the activity threshold on P2 and bounding the escape time from a flat region on P4.

\section{From Parts~I and~II to testable claims}
\label{sec: claims}

The two companion analyses~\citep{Jeremy2026-I, Jeremy2026-II} treat the trust-region framework
analytically, proving theorems about the sequences a method generates. We extract claims from those results and confront them with
experiment, for evidence where a claim can be settled and for insight where it cannot.

\label{subsec: two regimes}%
Two regimes carry the evidence and neither substitutes for the other. On a designed problem the
curvature at the solution is known in closed form, the limit point can be classified and the
radius is visible at every iteration, so one run either exhibits the behaviour a claim describes
or it does not, and carries nothing about how often it arises. On CUTEst none of those
quantities is available, so that regime cannot settle a claim whose statement mentions the
solution. It answers the question a user asks, whether a distinction established on a designed
problem is visible in what the method costs.

We present the claims in Table~\ref{tab: claims}, indicating the regime in which each claim is relevant and the section where it is discussed.

\begin{table}[htbp]
    \centering
    \small
    \renewcommand{\arraystretch}{1.15}
    \begin{tabular}{@{}r p{0.56\textwidth} l l@{}}
        \toprule
        & Claim & Regime & Section \\
        \midrule
        \addlinespace[3pt]
        \claimno{claim: asymptotics} &
            \emph{ $\Delta_k \to 0$ for criticality-anchored rules and
            $\liminf_k \Delta_k > 0$ for fixed-factor rules.}
            & designed & \ref{subsec: P1} \\
        \addlinespace[7pt]
        \claimno{claim: rate of decay} &
            \emph{Where the multiplier of a criticality-anchored rule is bounded and
            eventually constant, the decay is harmonic, $\Delta_k = \Theta(1/k)$, so
            $\sum_k \Delta_k$ diverges and $\sum_k \Delta_k^2$ converges. }
            & designed & \ref{subsec: P1} \\
        \addlinespace[7pt]
        \claimno{claim: zeta} &
            \emph{Along convergent runs, there is a problem-dependent threshold in $\zeta$ below which $\Rdfo$ keeps the
            constraint active indefinitely and above which it eventually becomes inactive. 
            Same for $\overline{\mu}$ in $\Rgrad$, and the uncapped
            variant crosses it unconditionally.} & designed, CUTEst & \ref{sec: threshold}, \ref{subsec: zeta sweep}, \ref{subsec: mu sweep} \\
        \addlinespace[7pt]
        \claimno{claim: rate} &
            \emph{Conditional on inactivity, the observed local rate is that of the underlying
            (quasi)-Newton iteration and is independent of the mechanism. Mechanisms differ in when
            inactivity is reached, not in the rate thereafter.}
            & CUTEst & \ref{subsec: order rules} \\
        \addlinespace[7pt]
        \claimno{claim: cauchy} &
            \emph{With truncated CG and $\overline{\mu}$ small enough that the first CG iterate
            leaves the trust region, the step is the Cauchy point and the method inherits
            gradient-descent behaviour.}
            & CUTEst & \ref{subsec: mu sweep} \\
        \addlinespace[7pt]
        \claimno{claim: bad limits} &
            \emph{A configuration can satisfy every hypothesis of \PartI{}, with every $\rho_k$
            successful and a well-behaved radius, and still converge to a saddle point or a
            degenerate critical point that is not a minimiser.}
            & designed & \ref{subsec: spd} \\
        \addlinespace[7pt]
        \claimno{claim: escape} &
            \emph{For every cap there is a flatness $\varepsilon$ on which $\Rgrad$ escapes the
            flat region only after arbitrarily many iterations, while $\Rdelta$ escapes from the
            same starting point in a number bounded independently of $\varepsilon$.}
            & designed & \ref{subsec: flat} \\
        \addlinespace[7pt]
        \claimno{claim: order} &
            \emph{A first-order anchor admits the failures of Claims~\ref{claim: cauchy}~and~\ref{claim: bad limits}. The second-order anchor removes them on the same problems and
            from the same starting points.}
            & designed & \ref{subsec: flat} \\
        \addlinespace[7pt]
        \claimno{claim: no uniform cap} &
            \emph{There exists a function, $f_2$ of~\eqref{eqn: p2}, such that for every cap
            $\overline{\mu}$, however large, one of its minimisers has a basin in which capped
            $\Rgrad$ converges with the trust-region constraint active at every sufficiently
            large iteration.}
            & designed & \ref{sec: threshold} \\
        \addlinespace[7pt]
        \bottomrule
    \end{tabular}
    \caption{The nine claims, the regime in which each can be settled, and the section that
    settles it. Section~\ref{sec: verdict} carries the verdicts.}
    \label{tab: claims}
\end{table}

\section{\textsf{TrustRegionRadius.jl}}
\label{sec: package}

Every number in Sections~\ref{sec: cost} to~\ref{subsec: flat} comes from one Julia package
running one iteration. This section states what that package is, why it had to be written, what
the iteration does, which mechanisms it carries, and what a run records.

\subsection{The package}
\label{subsec: package}
\textsf{TrustRegion\-Radius.jl} was written for this study and is released with it at
\url{https://github.com/JeremyRieussec/TrustRegionRadius.jl} under the MIT licence, with
documentation at \url{https://JeremyRieussec.github.io/TrustRegionRadius.jl/dev/}. It is built on
the JuliaSmoothOptimizers interfaces~\citep{JSO} in Julia~\citep{Julia17} and depends on
\textsf{NLPModels}, \textsf{SolverCore}, \textsf{Krylov} and \textsf{LinearOperators}.

\subsection{What a comparison across published codes cannot measure}
\label{subsec: why a package}

The comparison cannot be assembled from published codes. Julia has trust-region solvers and none
of them exposes the radius rule as a component that can be exchanged.
\textsf{JSOSolvers.jl}~\citep{JSOSolvers} supplies \texttt{trunk}, a trust-region Newton method
with truncated conjugate gradients, and \texttt{tron}, the bound-constrained method
of~\citet{LinMore99}, and \textsf{Optim.jl}~\citep{MogensenRiseth18} supplies
\texttt{NewtonTrustRegion}. Each fixes a single ratio-driven rule and binds it to its own
subproblem solver and its own stopping test, so a gap between two such codes is a gap between
their model constructions, their subproblem solvers, their stopping tests and their linear
algebra as much as between their radius rules. Figure~\ref{fig: package architecture} shows the
arrangement that removes that confound. Four components are independent and swappable, each
supplying exactly one step of the iteration.

\begin{figure}[t]
\centering
\begin{tikzpicture}[
    font=\small,
    axisbox/.style   = {draw, rounded corners=2pt, align=center, inner sep=4pt,
                        minimum height=10mm, fill=black!4},
    stepbox/.style   = {draw, thick, align=center, inner sep=4pt,
                        minimum height=9mm, minimum width=26mm, fill=white},
    outbox/.style    = {draw, dashed, align=center, inner sep=4pt,
                        minimum height=10mm},
    plug/.style      = {-{Latex[length=2mm]}, thin, black!55},
    flow/.style      = {-{Latex[length=2.4mm]}, thick},
    lbl/.style       = {font=\scriptsize\itshape, black!60}
]

\node[axisbox] (ax1) at (0,2.5)
    {\textbf{Model Hessian}\\[1pt]\scriptsize exact, L-BFGS, SR1,\\[-1pt]
     \scriptsize scaled identity, SPD target,\\[-1pt]
     \scriptsize BHHH, BHHH-2, Gauss--Newton};
\node[axisbox, right=6mm of ax1] (ax2)
    {\textbf{Subproblem solver}\\[1pt]\scriptsize Steihaug--Toint CG,\\[-1pt]
     \scriptsize Mor\'e--Sorensen, Krylov CG/CR/\\[-1pt]
     \scriptsize CG-Lanczos, eigenpoint wrapper};
\node[axisbox, right=6mm of ax2] (ax4)
    {\textbf{Radius rule}\\[1pt]\scriptsize eleven mechanisms, the six\\[-1pt]
     \scriptsize criticality-anchored ones with\\[-1pt]
     \scriptsize a variant $\omega_k=\tau_k$};

\node[stepbox] (s1) at (0,0) {Build $m_k$, $\omega_k$\\\scriptsize Step 1};
\node[stepbox, right=7mm of s1] (s2) {Solve for $s_k$\\\scriptsize Step 2};
\node[stepbox, right=7mm of s2] (s3) {$\rho_k$, accept if $\rho_k\!\geq\!\eta$\\\scriptsize Step 3};
\node[stepbox, right=7mm of s3] (s4) {Update $\Delta_{k+1}$\\\scriptsize Step 4};

\draw[flow] (s1) -- (s2);
\draw[flow] (s2) -- (s3);
\draw[flow] (s3) -- (s4);
\draw[flow] (s4.south) -- ++(0,-1.0) -| (s1.south)
    node[lbl, pos=0.82, below] {$k \leftarrow k+1$};

\begin{scope}[on background layer]
    \node[draw, black!35, rounded corners=3pt, fit=(s1)(s2)(s3)(s4),
          inner sep=6pt, fill=black!2] (loop) {};
\end{scope}
\node[lbl, anchor=east] at ($(loop.north east)+(-1mm,3.2mm)$)
    {one iteration, shared by the three problem classes};

\draw[plug] (ax1.south) -- (s1.north);
\draw[plug] (ax2.south) -- (s2.north);
\draw[plug] (ax4.south) -- (s4.north);

\node[axisbox, left=9mm of s1, text width=25mm] (ax3)
    {\textbf{Oracle}\\[1pt]\scriptsize deterministic,\\[-1pt]
     \scriptsize expectation,\\[-1pt] \scriptsize finite sum\\[2pt]
     \scriptsize ten sampling rules};
\draw[plug] (ax3.east) -- (s1.west) node[lbl, midway, above] {$f, g$};

\node[stepbox, minimum width=110mm, minimum height=8mm, fill=black!6]
     (tr) at ($(loop.south)+(0,-2.0)$)
    {Per-iteration record\\[-1pt]
     \scriptsize $\Delta_k$, $f_k$, $\norm{g_k}$, $\norm{s_k}$, $\rho_k$,
     $\widetilde\rho_{k+1}$, accepted, active, branch, CG iterations,
     $\cos(s_k,-g_k)$, $\gamma_k$, $\xi_k$, $\tau_k$, $\lambda_{\min}$};
\draw[plug] ($(loop.south)+(0,-0.55)$) -- (tr.north);

\node[outbox, text width=46mm] (diag) at ($(tr.south)+(-3.1,-1.5)$)
    {\textbf{Diagnostics}\\[1pt]
     \scriptsize $\overline{\kappa}$, $\theta_k$, activity fraction,\\[-1pt]
     \scriptsize inactivity index, branch counts,\\[-1pt]
     \scriptsize observed order, radius sums};
\node[outbox, text width=46mm] (bench) at ($(tr.south)+(3.1,-1.5)$)
    {\textbf{Benchmark}\\[1pt]
     \scriptsize run matrix, performance and\\[-1pt]
     \scriptsize data profiles, parameter sweeps,\\[-1pt]
     \scriptsize \textsc{pgfplots} export};
\draw[plug] (tr.south) -- (diag.north);
\draw[plug] (tr.south) -- (bench.north);

\end{tikzpicture}
\caption{Architecture of \textsf{TrustRegionRadius.jl}. The radius rule, the model Hessian,
the subproblem solver and the oracle are independent axes, each supplying one step of a single
iteration that every configuration shares. Step~1 also evaluates the criticality measure
$\omega_k$, which is $\norm{g_k}$ by default and $\tau_k$ under the second-order wrapper, and
Step~3 accepts on $\eta$ alone while Step~4 scales on $\eta_1$ and $\eta_2$. The per-iteration
record is written by the driver and is the only input to the diagnostics and to the benchmark
layer.}
\label{fig: package architecture}
\end{figure}
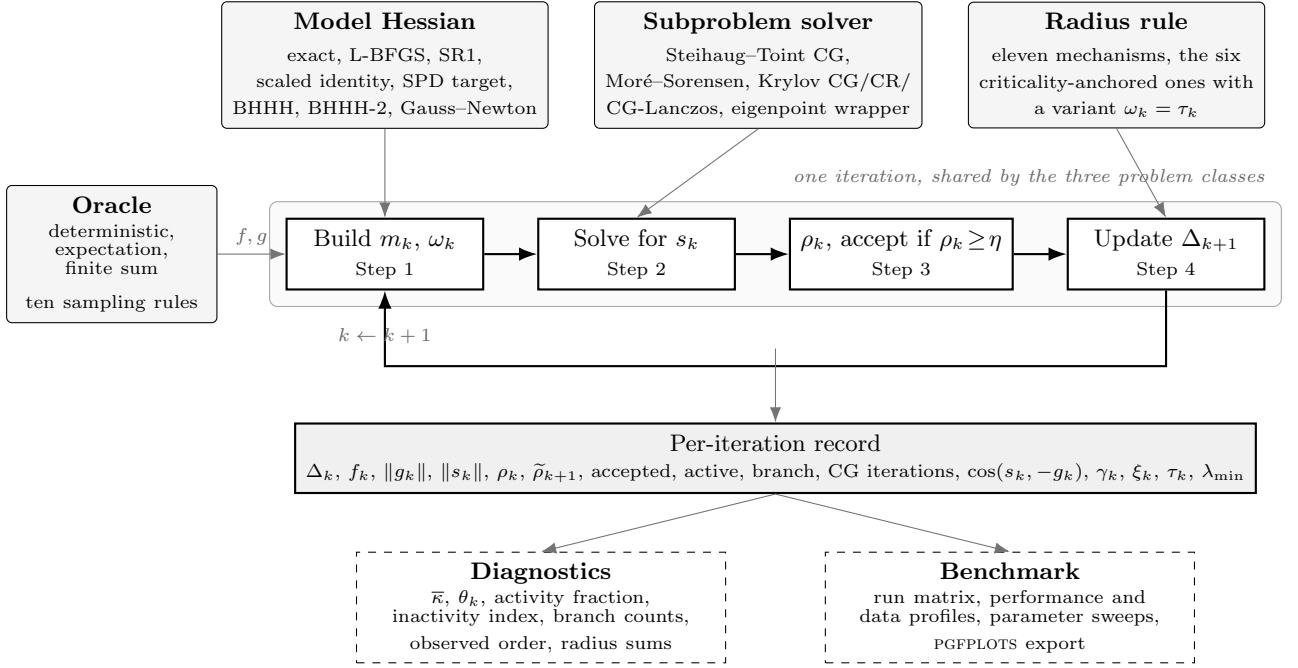

A rule is a struct and two methods, a model a struct and two methods, so a mechanism proposed
after this \thisdoc{} can be compared against the roster below without touching the driver.
\subsection{The driver}
\label{subsec: driver}

Algorithm~\ref{alg: tr solve} is the loop every configuration shares.

\begin{algorithm}[!tbp]
\small
\caption{The trust-region driver, \texttt{tr\_solve}}
\label{alg: tr solve}
\DontPrintSemicolon
\KwIn{$x_0$; rule $\calR$, model $\calM$, subsolver $\calS$; thresholds
      $0 \leq \eta \leq \eta_1 \leq \eta_2 < 1$; $\Delta_0$, $\Delta_{\min}$, $\Delta_{\max}$;
      tolerances $\epsilon$ and $\epsilon_H$; budgets $k_{\max}$, $t_{\max}$}
\KwOut{$x_k$ and a status}
\BlankLine
$f \leftarrow f(x_0)$; $g \leftarrow \nabla f(x_0)$;
$\omega \leftarrow \texttt{criticality}(\calR, \norm{g}, \lambda_{\min})$
  \tcp*{$\norm{g}$, or $\tau$ under \texttt{SecondOrder}}
$\Delta \leftarrow \operatorname{clamp}\{\texttt{initial\_radius}(\calR, \Delta_0, \omega),\
  \Delta_{\min},\ \Delta_{\max}\}$; $k \leftarrow 0$\;
\While{true}{
    \If{$\norm{g} \leq \epsilon$ \textbf{and}
        \textup{(}$\epsilon_H \leq 0$ \textbf{or} $\lambda_{\min} \geq -\epsilon_H$\textup{)}}{
        \KwRet \texttt{first\_order} if $\epsilon_H \leq 0$, else \texttt{second\_order}\;
    }
    \lIf{$k \geq k_{\max}$}{\KwRet \texttt{max\_iter}}
    \lIf{elapsed time $> t_{\max}$}{\KwRet \texttt{max\_time}}
    $k \leftarrow k+1$\;
    \BlankLine
    $(s, \texttt{active}) \leftarrow \texttt{solve\_subproblem}(\calS, \calM, x, g, \Delta)$
      \tcp*{Algorithm~\ref{alg: SteihaugCG}}
    \lIf{the model does not exist at $x$}{\KwRet \texttt{exception}}
    \BlankLine
    $\text{pred} \leftarrow -g^\top s - \tfrac12 s^\top H_k s$;
    $\text{ared} \leftarrow f(x) - f(x+s)$\;
    $\rho \leftarrow \text{ared}/\text{pred}$ if $\text{pred} > 0$, else $-\infty$\;
    \lIf{$\abs{\text{ared}} \leq \varepsilon_{\mathrm{mach}}\max(1,\abs{f})$
         \textbf{and} $\norm{s} < \varepsilon_{\mathrm{mach}}^{3/4}$}{\KwRet \texttt{stalled}}
    \BlankLine
    $\omega^{-} \leftarrow \omega$; $\texttt{accepted} \leftarrow (\rho \geq \eta)$
      \tcp*{acceptance uses $\eta$ alone, never $\eta_1$}
    \If{\texttt{accepted}}{
        $y \leftarrow \nabla f(x+s) - g$;
        $x \leftarrow x + s$; $f \leftarrow f(x)$; $g \leftarrow \nabla f(x)$;
        $\texttt{update\_model}(\calM, s, y)$\;
        $\omega \leftarrow \texttt{criticality}(\calR, \norm{g}, \lambda_{\min})$\;
    }
    \BlankLine
    \eIf{$\texttt{needs\_retrospective}(\calR)$ \textbf{and} \texttt{accepted}}{
        $\widetilde\rho \leftarrow \texttt{retrospective\_ratio}(\text{ared}, s, g, H_{k+1}s)$\;
    }{
        $\widetilde\rho \leftarrow \rho$\;
    }
    \BlankLine
    $\Delta \leftarrow \texttt{update\_radius}(\calR, \Delta, \widetilde\rho,
      \texttt{accepted}, \eta_1, \eta_2, \norm{s}, \omega^{-}, \omega)$
      \tcp*{the only line the rule touches}
    $\Delta \leftarrow \operatorname{clamp}\{\Delta,\ \Delta_{\min},\ \Delta_{\max}\}$\;
}
\end{algorithm}

Acceptance and scaling are two decisions. The step is accepted when $\rho_k \geq \eta$, and
$\eta_1$ and $\eta_2$ only scale the radius. Taking $\eta < \eta_1$ opens the band
$\rho_k \in [\eta, \eta_1)$, on which the step is accepted and the radius contracts anyway.
Setting $\eta = \eta_1$ recovers the coupled behaviour exactly, and an implementation that
accepted on $\eta_1$ could not express the band at all. 

The rule receives the criticality measure twice, before and after the acceptance decision.
$\Rdfo$ gates on the value at the current iterate, $\Rgrad$ and its relatives set the next radius
from the value at the next one, and passing both keeps the rules comparable without a special
case in the loop. The measure itself is supplied by \texttt{criticality}, which returns
$\norm{g_k}$ by default and $\tau_k = \max\{\norm{g_k}, -\lambda^-_k\}$ when the rule is wrapped
in \texttt{SecondOrder}. The update subroutine is identical in both cases. Only the quantity
placed in those two slots changes, and that is the entire content of the second-order variant.
The retrospective ratio compares the achieved reduction with the reduction the \emph{new} model
predicts for the step just taken,
\begin{equation}
    \widetilde\rho_{k+1}
      = \frac{f(x_k) - f(x_{k+1})}{-g_{k+1}^\top s_k + \tfrac12 s_k^\top H_{k+1} s_k} ,
    \label{eqn: retrospective ratio}
\end{equation}
set to $-\infty$ when the denominator is not positive.
The acceptance flag is passed to the rule. $\Rrtr$ and $\Rgrtr$ branch on it before anything
else, because $\widetilde\rho$ judges a step that was not taken when the step was rejected.

\subsection{The radius mechanisms under test}
\label{subsec: roster}

Eleven mechanisms are implemented and Appendix~\ref{app: update mechanisms} states them all.
They differ along three axes, and Table~\ref{tab: roster axes} places each one on them. The
eleven were chosen so that every comparison in this \thisdoc{} moves one axis and holds the other
two. Six are anchored to a criticality measure, and each of those runs at either order of
criticality with no change to the rule, giving the $\omega_k = \tau_k$ column of
Section~\ref{sec: cost}.

The updates of Appendix~\ref{app: update mechanisms} are set-valued on their contraction and
expansion branches. An implementation has to select a point from each interval, and the package
selects the endpoint nearest the boundary the branch is named for, so
$\Delta_{k+1} = \gone\Delta_k$ on contraction, $\gtwo\Delta_k$ on the intermediate branch and
$\gthree\Delta_k$ on expansion. Every rule takes $\eta$, $\eta_1$ and $\eta_2$ from the driver
rather than carrying its own, so a comparison at fixed thresholds is a comparison at genuinely
fixed thresholds.
The package departs from \PartII{} on $\Rrtr$ in one respect: it holds the safeguard factor
at $\gone$, so that a step with a negative ratio gives
$\Delta_{k+1} = \min\{\gtwo\norm{s_k}, \gone\Delta_k\}$, the smallest radius the rule allows.
Both retrospective rules take acceptance from $\eta$ alone, which equals $\eta_1$ in every
campaign.

\begin{table}[htbp]
    \centering
    \small
    \begin{tabular}{@{}lllllll@{}}
        \toprule
        Update & Anchor & Ratio & Factor & Cap & Stated in & Type in the package \\
        \midrule
        $\Rdelta$                                & previous radius     & $\rho_k$                 & fixed    & ---              & I, II  & \texttt{RDelta} \\
        $\Rstep$                                 & step                & $\rho_k$                 & fixed    & ---              & I, II  & \texttt{RStep} \\
        $\Rdeltastep$                            & prev.\ radius, step & $\rho_k$                 & fixed    & ---              & I      & \texttt{RDeltaStep} \\
        $\Rdfo(\omega)$                          & criticality measure & $\rho_k$                 & fixed    & ---              & I, II  & \texttt{RDFO} \\
        $\Rgrad(\omega)$                         & criticality measure & $\rho_k$                 & fixed    & none             & I, II  & \texttt{RGrad} \\
        $\Rgrad(\omega, \overline{\mu})$         & criticality measure & $\rho_k$                 & fixed    & $\overline{\mu}$ & II     & \texttt{RGradCapped} \\
        $\RAdaptstep$                            & step                & $\rho_k$                 & adaptive & ---              & I      & \texttt{RAdaptiveStep} \\
        $\RAdaptiveGrad(\omega)$                 & criticality measure & $\rho_k$                 & adaptive & none             & III    & \texttt{RAdaptiveGrad} \\
        $\RAdaptiveGrad(\omega, \overline{\mu})$ & criticality measure & $\rho_k$                 & adaptive & $\overline{\mu}$ & III    & \texttt{RAdaptiveGradCapped} \\
        $\Rrtr$                                  & step, prev.\ radius & $\widetilde{\rho}_{k+1}$ & fixed    & ---              & II     & \texttt{RRTR} \\
        $\Rgrtr(\omega, \overline{\mu})$         & criticality measure & $\widetilde{\rho}_{k+1}$ & fixed    & $\overline{\mu}$ & II     & \texttt{RRTRGrad} \\
        \bottomrule
    \end{tabular}
    \caption{The mechanisms under scrutiny, and the type that implements each one. The
             correspondence is one to one. A capped rule is a separate type rather than a
             keyword, so a sweep over $\overline{\mu}$ and a run without a cap are different
             configurations and never one object in two states. Where this \thisdoc{} takes the
             capped form as primary and the uncapped one as its variant, the package does the
             reverse, \texttt{RGrad} being uncapped and \texttt{RGradCapped} the variant; the
             rules are the same and only the emphasis is inverted. The second-order readings
             $\omega_k = \tau_k$ are not further types, \texttt{SecondOrder} wrapping any of the
             criticality-anchored six. The last column but one records where each update is
             first stated, and $\Rdelta$, $\Rdfo$ and $\Rgrad$ are stated differently in
             Parts~I and~II, Section~\ref{sec: introduction} recording which statement is
             implemented here.}
    \label{tab: roster axes}
\end{table}

$\RAdaptstep$ is the self-adaptive rule of \citet{Hei03}, for which we define 
\begin{equation}
    R_{\eta_1}(t) =
    \begin{cases}
        \gone + (\gtwo - \gone)\,\rme^{\lambda_1 (t - \eta_1)},
            & t < \eta_1, \\[2pt]
        \gthree + (M - \gthree)
            \bigl(1 - \rme^{-\lambda_2 (t - \eta_1)}\bigr),
            & t \geq \eta_1,
    \end{cases}
    \label{eqn: hei factor}
\end{equation}
which satisfies the four conditions of
Update~\ref{update: adaptive step-size dependent update}. \PartI{} writes the same factor in a narrower normalisation, with $1 + \gtwo$
where~\eqref{eqn: hei factor} has $\gthree$ and $\gthree$ where it has $M$.
Applying it to $\norm{s_k}$ gives
$\RAdaptstep$ and applying it to the multiplier $\mu$ that then multiplies $\omega_{k+1}$ gives
$\RAdaptiveGrad$. The Fan--Yuan multiplier is the one carried by $\Rgrad$ and $\Rgrtr$ and is not
a separate rule.
\subsection{What a run records}
\label{subsec: instruments}

With tracing on, the driver writes one record per iteration, and that record is the only input to
every number reported below. It holds the radius, the objective, the gradient norm, the step
norm, both ratios, whether the step was accepted, whether the trust-region constraint was
binding, the branch the rule took, the number of inner iterations, the cosine between the step
and the steepest-descent direction, and, where a rule or a stopping test requires it, the
criticality measure and the leftmost model Hessian eigenvalue. The radius, the objective and the
gradient norm have a value before the first iteration and so are one entry longer than the rest,
which matters to anything plotting them together.

The acceptance flag is kept even though it looks redundant. When $\eta < \eta_1$ the rule's
thresholds no longer coincide with the one that decided the step, so the accepted set cannot be
reconstructed from the ratios.

Six derived quantities are computed from that record and nothing else.

\begin{itemize}[leftmargin=1.6em, itemsep=2pt]
    \item The ratio $\theta_k \defeq \Delta_k/\omega_k$, which is the multiplier $\mu_k$ itself
          under $\Rgrad$ and is the quantity every activity threshold is stated in.
    \item The \emph{activity fraction} over a tail of the run, the proportion of iterations on
          which $\norm{s_k} = \Delta_k$ within the subproblem solver's own tolerance. It is
          computed only for runs of at least $30$ iterations, a tail of one iteration taking only
          the values $0$ and $1$ and measuring rounding rather than asymptotics.
    \item The \emph{inactivity index} $k^\star$, the first iteration after which the constraint
          never binds again, undefined when the final iteration is active.
    \item The \emph{observed order}, the ordinary least-squares slope of $\log\norm{g_{k+1}}$ on
          $\log\norm{g_k}$ over the accepted iterations of a window beginning at $k^\star$ and
          ending at termination. It is reported with the number of points it used and is
          undefined below six.
\end{itemize}

The activity fraction is the one instrument both regimes of Section~\ref{subsec: two regimes}
share, and it is the only one that needs no knowledge of the solution. It is also the quantity
that separates mechanisms a first-order convergence test cannot tell apart, both of them driving
$\norm{g_k}$ to zero.

\subsection{Availability}
\label{subsec: availability}

Each campaign writes a timestamped archive holding the configuration in TOML, a summary in
Markdown, the figures, the tables and the raw data, so that an archive describes the run that
produced it. Every number in this \thisdoc{} is traceable to one such archive, named where it is first
used. Every campaign reported here was run at one release of the package,
TrustRegionRadius.jl v0.1.0, though at five distinct commits within it, each recorded in the
archive it produced. The archives are deposited with that release at
10.5281/zenodo.22655838, so that a number, the code that produced it and the configuration
it was produced under are recovered together.

\section{The test problems}
\label{sec: problems}
Two families carry the evidence, and Section~\ref{subsec: two regimes} says why neither
substitutes for the other. The designed problems of Sections~\ref{subsec: P1}
to~\ref{subsec: P4} are stated here in closed form, with their critical points, their
curvatures and their starting points, so that nothing in Sections~\ref{sec: cost}
to~\ref{subsec: flat} rests on a quantity the reader has to take on trust. The CUTEst
collection of Section~\ref{subsec: cutest} supplies the aggregate regime.

\subsection{What a designed problem has to supply}
\label{subsec: designed criteria}

Three requirements decide the family. Each problem states its critical points in closed form, so
a run is classified by where it stops and not merely by whether it stopped. Each carries a
computable curvature at the point the method reaches, so the exact activity threshold
$1/\lambda^*_{\min}$ is a number rather than an estimate. Each isolates one mechanism, so a
single run is evidence about that mechanism.

The four are ordered by what they do to the curvature at the limit. P1 has no critical point at
all, so it separates the rules on the radius alone. P2 has a sequence of minimisers whose
curvature decays to zero along one fixed function, so it supplies a sequence of thresholds
without changing the problem. P3 has a degenerate critical point at which the second-order
measure vanishes, a nonsingular saddle at which it does not, and two nondegenerate minimisers,
all on one polynomial. P4 has a nonsingular saddle whose negative curvature is the parameter, so
it measures how long escape takes as that curvature goes to zero.
Table~\ref{tab: designed summary} states the four together.

\begin{table}[htbp]
    \centering
    \small
    \caption{The four designed problems. Column five is the smallest eigenvalue of the Hessian
    at the minimiser, which is the quantity every threshold in this \thisdoc{} is stated in. P1 has
    no minimiser and so no threshold, and it carries the radius trajectory alone. P2 supplies a sequence of thresholds along one function, P3 a
    single threshold in two coupled dimensions, and P4 a threshold tunable to zero.}
    \label{tab: designed summary}
    \begin{tabular}{@{}llllll@{}}
        \toprule
        & $n$ & objective & limit the method reaches & $\lambda^*_{\min}$ & isolates \\
        \midrule
        P1 & $1$ & $\rme^{-x}$
           & none, $x_k \to +\infty$ & n/a
           & radius asymptotics \\
        P2 & $1$ & $\sin(x)/x$
           & minimiser $y_j$ & $(1+y_j^2)^{-1/2}$
           & the activity threshold \\
        P3 & $2$ & \eqref{eqn: p3}
           & minimiser, saddle or origin & $0.4946892$
           & which limits the measures see \\
        P4 & $2$ & \eqref{eqn: p4}
           & minimiser or saddle & $2\varepsilon$
           & escape time as $\lambda^- \to 0$ \\
        \bottomrule
    \end{tabular}
\end{table}

The four are not interchangeable and each answers a question the others cannot. P1 has no
minimiser, so no statement about inactivity applies to it and every difference between the rules
is a difference in the radius. P2 has a diagonal Hessian by construction, being
one-dimensional. P3 has a coupled Hessian at the minimiser, three kinds of critical point and
no parameter. P4 has one kind of failure and a parameter that drives it, which is what a rate
measurement needs and what a fixed problem cannot give.

\subsection{P1, the exponential}
\label{subsec: P1}

The first problem is
\begin{equation}
    f_1(x) = \rme^{-x} , \qquad
    f_1'(x) = -\rme^{-x} , \qquad
    f_1''(x) = \rme^{-x} > 0 ,
    \qquad x_0 = 0 .
    \label{eqn: p1}
\end{equation}
The objective is strictly convex and bounded below by zero, it has no critical point, and its
infimum is not attained. The iterates therefore have no accumulation point and travel to
$+\infty$, and the gradient decays at the known rate $\norm{g_k} = \rme^{-x_k}$. None of that
puts P1 outside \PartI{}. It is bounded below, its gradient is Lipschitz on any set bounded to
the left, and its Hessian is uniformly bounded there. Attainment is what the assumptions do not
require.

Two properties make P1 the right instrument. The Newton step is $1$ everywhere, since
$-f_1'/f_1'' = 1$, so under the exact Hessian $s_k = \min\{1, \Delta_k\}$ and
\begin{equation}
    \rho_k = \frac{1 - \rme^{-s_k}}{s_k - \tfrac12 s_k^2} > 1
    \qquad \text{for every } s_k \in (0,1] ,
    \label{eqn: p1 rho}
\end{equation}
since $1 - \rme^{-s} - s + \tfrac12 s^2$ vanishes at $s = 0$ with derivative
$\rme^{-s} - 1 + s \geq 0$. Every iteration is therefore very successful and nothing but the rule
ever moves the radius. And the gradient decays at a known rate, so a criticality-anchored radius
decays in closed form.

\PartII{} does that calculation on this problem: once the multiplier sits at its cap and the
radius has fallen below the Newton step, $x_{k+1} = x_k + \Delta_k$ and
$\Delta_{k+1} = \Delta_k \rme^{-\Delta_k}$, whence $k\Delta_k \to 1$. The decay is geometric in
$x$ and harmonic in $k$ only because the iterates reach $x_k = \ln k + \calO(\ln\ln k)$ and no
further. The cap carries that argument, which is why Claim~\ref{claim: rate of decay} is stated
for a multiplier bounded and eventually constant. What is left here is the measurement.

    Under capped $\Rgrad$ with $\mu_0 = \overline{\mu} = 1$, the exact Hessian and truncated
    conjugate gradients, from $x_0 = 0$ over $2 \times 10^{5}$ iterations, $k\Delta_k$ reads
    $0.8698$, $0.9746$, $0.99627$, $0.99951$ and $0.99994$ at
    $k = 10, 10^{2}, 10^{3}, 10^{4}, 10^{5}$. The two series separate as the claim asserts:
    $\sum_{k \leq m}\Delta_k$ gives $9.2109$ against $\ln m = 9.2103$ at $m = 10^{4}$ and
    $11.5130$ against $11.5129$ at $m = 10^{5}$, which is $x_k = \ln k + \calO(1)$ since
    $x_k = x_0 + \sum_{j<k}\Delta_j$, while $\sum_k \Delta_k^2$ has reached $1.4083$ with a tail
    of order $10^{-5}$. One call, recorded in \texttt{TrustRegionRadius.jl/\allowbreak benchmark/\allowbreak experiments/\allowbreak p1\_harmonic\_decay.jl}.
    Appendix~\ref{app: p1 asymptotics} draws the same problem with the other four mechanisms
    beside this one, where the contrast between the anchors is what the picture carries.

\paragraph{Verdict on Claims~\ref{claim: asymptotics} and~\ref{claim: rate of decay}. Both settled
here.}
Claim~\ref{claim: asymptotics} asks for $\Delta_k \to 0$ under a criticality-anchored rule and
$\liminf_k \Delta_k > 0$ under a fixed factor, and one problem answers both because the contrast
is the whole of it. Capped $\Rgrad$ drives the radius from $1$ down to $5.00 \times 10^{-5}$ over
$2 \times 10^{4}$ iterations, monotonically. $\Rdelta$ from the same $\Delta_0$ never returns
below its initial radius, which is therefore its minimum over the run, and reaches
$2.68 \times 10^{8}$ before the gradient test stops it at iteration $29$. Every iteration of both
is very successful by~\eqref{eqn: p1 rho}, so nothing but the rule moves the radius.
Claim~\ref{claim: rate of decay} is the rate at which the first of the two falls, and the reading
above supplies it. What P1 cannot say is whether the harmonic rate survives near a minimiser,
since it has none; that question belongs to P2 and P3 and is not asked here.

\subsection{P2, the cardinal sine}
\label{subsec: P2}

The second problem is
\begin{equation}
    f_2(x) = \frac{\sin x}{x} , \qquad f_2(0) \defeq 1 ,
    \label{eqn: p2}
\end{equation}
which is $C^\infty$ on $\reals$ with $\abs{f_2''} \leq \tfrac13$ everywhere, so the model Hessian
is uniformly bounded whatever the iterates do. Writing $f_2'(x) = (x\cos x - \sin x)/x^2$, the
critical points are $x = 0$ and the roots of $\tan x = x$, one in each interval
$(j\pi, (j+\tfrac12)\pi)$, for which
\begin{equation}
    f_2''(y_j) = \frac{(-1)^{j+1}}{\sqrt{1 + y_j^2}},
    \label{eqn: sinc curvature}
\end{equation}
each of them the curvature is exactly given by this expression.
$y_j$ is a minimiser for $j$ odd and a local maximiser for $j$ even.
Table~\ref{tab: sinc values} lists the first values and the constants the thresholds are
stated in.

\begin{table}[tbp]
    \centering
    \caption{Critical points of $f_2(x) = \sin(x)/x$, and the constants attached to each
        minimiser. The root solve carries about $70$ decimal digits and has residual
        $\max_j \abs{f_2'(y_j)} = 1.28 \times 10^{-65}$, and~\eqref{eqn: sinc curvature} holds at
        the returned roots to $2.78 \times 10^{-17}$ in double precision.
        The last three columns are the exact threshold for capped $\Rgrad$
        of~\eqref{eqn: rgrad threshold}, the limiting value of the provable threshold for
        $\Rdfo$ of~\eqref{eqn: sinc zeta threshold}, and the operative constant
        $\kbar = 8/\lambda^*_j$ of \PartII{}. A maximiser carries no threshold.}
    \label{tab: sinc values}
    \small
    \begin{tabular}{rrrlrrr}
        \toprule
        $j$ & $y_j$ & $\lambda^*_j = f_2''(y_j)$ & type
          & $1/\lambda^*_j$ & $1/(5\lambda^*_j)$ & $\kbar = 8/\lambda^*_j$ \\
        \midrule
        $1$ & $4.493409458$ & $0.21723363$ & minimiser & $4.60$ & $0.92$ & $36.83$ \\
        $2$ & $7.725251837$ & $-0.12837455$ & maximiser & & & \\
        $3$ & $10.904121659$ & $0.091325203$ & minimiser & $10.95$ & $2.19$ & $87.60$ \\
        $4$ & $14.066193913$ & $-0.070913459$ & maximiser & & & \\
        $5$ & $17.220755272$ & $0.057971802$ & minimiser & $17.25$ & $3.45$ & $138.00$ \\
        $6$ & $20.371302959$ & $-0.049029624$ & maximiser & & & \\
        $7$ & $23.519452499$ & $0.042479617$ & minimiser & $23.54$ & $4.71$ & $188.33$ \\
        $8$ & $26.666054259$ & $-0.037474520$ & maximiser & & & \\
        $\vdots$ & $\vdots$ & $\vdots$ & & $\vdots$ & $\vdots$ & $\vdots$ \\
        $21$ & $67.529434777$ & $0.014806734$ & minimiser & $67.54$ & $13.51$ & $540.29$ \\
        $51$ & $161.785840728$ & $0.0061808926$ & minimiser & $161.79$ & $32.36$ & $1294.31$ \\
        $101$ & $318.868518261$ & $0.0031360734$ & minimiser & $318.87$ & $63.77$ & $2550.96$ \\
        \bottomrule
    \end{tabular}
\end{table}

In one dimension the subproblem is solved exactly. For a quadratic model of curvature $a > 0$
the solution is $s_k = -g_k/a$ when $\abs{g_k}/a \leq \Delta_k$ and
$s_k = -\Delta_k\operatorname{sign}(g_k)$ otherwise, so, writing
$\theta_k \defeq \Delta_k/\abs{g_k}$,
\begin{equation}
    \text{the trust region is active at } k
    \iff \theta_k < 1/a_k ,
    \qquad a_k \defeq f_2''(x_k) .
    \label{eqn: sinc activity}
\end{equation}
The exact threshold is therefore $1/\lambda^*_{\min}$.
The curvature at the minimisers decays like $1/y_j$, so one fixed function supplies a sequence
of thresholds that runs to zero.Sweeping $j$ moves $1/\lambda^*_{\min}$ over an order of magnitude while the objective, the smoothness and the Hessian bound stay fixed. That is what P2 is for, and it is why a threshold
located here cannot be an artefact of a change of problem.

\subsection{P3, the quartic with a degenerate critical point}
\label{subsec: P3}

The third problem is
\begin{equation}
    f_3(x,y) = x^4 - x^3 + \Bigl(\tfrac14 - \tfrac{x}{2}\Bigr) y^2 + \tfrac14 y^4 ,
    \label{eqn: p3}
\end{equation}
with
\begin{equation}
    \nabla f_3(x,y) =
    \begin{pmatrix} 4x^3 - 3x^2 - \tfrac12 y^2 \\[2pt] \bigl(\tfrac12 - x\bigr) y + y^3 \end{pmatrix} ,
    \qquad
    \nabla^2 f_3(x,y) =
    \begin{pmatrix} 12x^2 - 6x & -y \\[2pt] -y & \tfrac12 - x + 3y^2 \end{pmatrix} .
    \label{eqn: p3 derivatives}
\end{equation}
The objective is coercive and bounded below, so the level sets are compact and the Hessian is bounded
on each of them. There are four critical points and no others, and Table~\ref{tab: p3 critical points} lists them.

\begin{table}[htbp]
    \centering
    \caption{The four critical points of $f_3$. The measure $\tau$ is the second-order
        criticality measure $\max\{\norm{g}, -\lambda^-\}$ of \PartII{}, evaluated with the
        exact Hessian. It is positive at the nonsingular saddle and zero at the degenerate
        point, which is the distinction P3 exists to draw. Writing $s = (\sqrt5-1)/4$, the
        minimisers have $\nabla^2 f_3 = \bigl(\begin{smallmatrix} 3 & \mp\sqrt{s} \\ \mp\sqrt{s}
        & 2s \end{smallmatrix}\bigr)$, trace $3+2s$, determinant $5s$ and discriminant
        $(25-5\sqrt5)/2$.}
    \label{tab: p3 critical points}
    \small
    \begin{tabular}{@{}llrrrl@{}}
        \toprule
        Point & $\nabla^2 f_3$ & $\lambda_{\min}$ & $\lambda_{\max}$ & $\tau$ & type \\
        \midrule
        $(0,0)$
          & $\operatorname{diag}(0,\ \tfrac12)$
          & $0$ & $0.5$ & $0$
          & degenerate, not a minimiser \\
        $(\tfrac34,\, 0)$
          & $\operatorname{diag}(\tfrac94,\ -\tfrac14)$
          & $-0.25$ & $2.25$ & $0.25$
          & nonsingular saddle \\
        $\bigl(\tfrac{1+\sqrt5}{4},\, +\sqrt{s}\bigr)$
          & see caption
          & $0.4946892$ & $3.1233448$ & $0$
          & minimiser, $f_3 = -\tfrac18$ \\
        $\bigl(\tfrac{1+\sqrt5}{4},\, -\sqrt{s}\bigr)$
          & see caption
          & $0.4946892$ & $3.1233448$ & $0$
          & minimiser, $f_3 = -\tfrac18$ \\
        \bottomrule
    \end{tabular}
\end{table}

The two minimisers are nondegenerate and  the minimum value is exactly $-\tfrac18$. The threshold follows,
\begin{equation}
    \begin{gathered}
        \lambda^*_{\min}
            = \frac{1}{2}\left(3 + 2s - \sqrt{\tfrac{25-5\sqrt5}{2}}\,\right)
            = 0.4946892 , \\[4pt]
        \frac{1}{\lambda^*_{\min}} = 2.021471 , \qquad
        \kbar = \frac{8}{\lambda^*_{\min}} = 16.171770 .
    \end{gathered}
    \label{eqn: p3 threshold}
\end{equation}
The other two critical points are what P3 is really for, because they fail in different ways and
one criticality measure separates them while the other does not. At $(\tfrac34, 0)$ the Hessian
is $\operatorname{diag}(\tfrac94, -\tfrac14)$, indefinite and nonsingular, so the second-order
measure reports $\tau = \tfrac14$ and a $\tau$-anchored rule keeps a radius bounded away from
zero there. At the origin the Hessian is $\operatorname{diag}(0, \tfrac12)$, positive
semidefinite and singular, so both measures vanish, $\norm{g} = 0$ and $\tau = 0$. The origin
satisfies the second-order necessary conditions but is not a local minimiser. The repair of
Claim~\ref{claim: order} therefore has a boundary, and P3 exhibits it on the same problem that
exhibits the repair working.

\subsection{P4, the flat well}
\label{subsec: P4}

The fourth problem is a one-parameter family. For $\varepsilon \in (0, \tfrac12)$ let
\begin{equation}
    f_4(x,y;\varepsilon) = \tfrac12 x^2
        + \varepsilon \left( \tfrac14 y^4 - \tfrac12 y^2 \right),
    \label{eqn: p4}
\end{equation}
with $\nabla f_4 = \bigl(x,\ \varepsilon y(y^2-1)\bigr)^\top $, so the critical points are the origin and $(0,\pm1)$, and
\begin{equation}
    \nabla^2 f_4(0,0) = \operatorname{diag}(1, -\varepsilon) ,
    \qquad
    \nabla^2 f_4(0,\pm1) = \operatorname{diag}(1, 2\varepsilon) .
    \label{eqn: p4 hessians}
\end{equation}
The origin is a saddle with $\det \nabla^2 f_4(0,0) = -\varepsilon \neq 0$, so it is nonsingular
for every $\varepsilon > 0$ and the second-order measure correctly reports
$\tau(0,0) = \varepsilon$. At the minimisers $\lambda^*_{\min} = 2\varepsilon$, so
$1/\lambda^*_{\min} = 1/(2\varepsilon)$ and $\kbar = 4/\varepsilon$. P4 is therefore designed 
to illustrate slow escape from a flat well as $\varepsilon \to 0$.

The minimisers sit at distance $1$ from the saddle whatever $\varepsilon$ is, so the parameter
tunes the escape rate without moving the geometry. That separation is what P4 contributes and
what P3 cannot. Along the unstable direction, and on the invariant line $x = 0$, a
curvature-blind model $B_k = cI$ gives the step $-g_y/c$ and the recursion
$y_{k+1} = y_k\bigl(1 + (\varepsilon/c)(1 - y_k^2)\bigr)$, an Euler discretisation of
$\dot y = (\varepsilon/c)\,y(1-y^2)$. Separating and integrating from $y_0$ to the escape
threshold $\abs{y} = \tfrac12$ gives the antiderivative
$\ln\bigl(y/\sqrt{1-y^2}\bigr)$ and hence
\begin{equation}
    \frac{\varepsilon}{c}\, k_{\mathrm{esc}}
    \;\longrightarrow\; \ln\frac{1}{\sqrt3\, y_0}
    \qquad \text{as } \varepsilon \to 0 .
    \label{eqn: escape time}
\end{equation}

\begin{remark}[Where the escape constant comes from]
\label{rem: escape constant}
The threshold $\abs{y} = \tfrac12$ and the factor $1 - y^2$ in the gradient together fix the
constant at $\ln(1/(\sqrt3\, y_0))$, which is $6.3584$ at $y_0 = 10^{-3}$ against a measured
$\varepsilon k_{\mathrm{esc}}$ of $6.3588$ at $\varepsilon = 10^{-4}$. Two cruder constants
circulate and neither fits. Bounding the gradient by $\abs{g_y} \leq \varepsilon\abs{y}$ drops
the factor $1-y^2$ and gives $\ln(1/y_0) = 6.9078$. Taking the escape threshold alone gives
$\ln(1/(2y_0)) = 6.2146$, whose residual does not vanish with $\varepsilon$ but settles at
$\ln 2 - \tfrac12 \ln 3 = 0.1438$.
\end{remark}

\subsection{CUTEst}
\label{subsec: cutest}

The aggregate study uses the CUTEst collection~\citep{GouldOrbanToint15}, the current form of a
test set that began as CUTE~\citep{BongartzConnGouldToint95} and passed through
CUTEr~\citep{GouldOrbanToint03}. It gathers problems contributed over four decades from
applications and from the earlier collections of Hock and
Schittkowski~\citep{HockSchittkowski81} and others, in a common SIF description with exact
derivatives supplied by the decoder, so that no two configurations differ in the accuracy of the
gradients they receive.

The selection is every unconstrained problem with $2 \leq n \leq 1000$, which gives $196$
problems. The lower bound removes the scalar problems, on which a trust region carries no
directional information. The upper bound keeps the campaigns inside the time cap of
Section~\ref{subsec: settings} and keeps every configuration on the same problems, the cost of a
truncated-conjugate-gradient run growing with $n$ through the Hessian-vector products. Every
result in Sections~\ref{sec: inactivity} and~\ref{sec: cost} is computed on this set.

The collection is heterogeneous by construction and that is the point of using it. It contains
badly scaled problems, problems with many local minimisers, and problems on which every
configuration fails, so a solve rate measured on it is a statement about coverage rather than
about difficulty. Two configurations that solve the same number of problems need not solve the
same problems, and Section~\ref{subsec: like for like} reports the pairs that differ in one
coordinate rather than the ranking alone.

What the collection cannot supply is the reason the designed problems exist. The solutions are
unknown, so $\lambda^*_{\min}$ is not computable, the limit point cannot be classified, and no
threshold can be located. The activity of the trust-region constraint is the one instrument that
survives the passage from the designed problems to CUTEst, because it is read from the run
record and needs no knowledge of the solution.

\section{Experimental setup}
\label{sec: setup}

Section~\ref{sec: package} fixes the iteration. What remains is everything held constant around
it, the subproblem solvers and model Hessians the campaign was run with, when a run stops, at
which constants, and how two runs are compared.
\subsection{The subproblem solvers}
\label{subsec: subsolvers}

The subproblem
\begin{equation}
    \min_{\norm{s} \leq \Delta} \; q(s) \defeq g^\top s + \tfrac12 s^\top H s
    \label{eqn: subproblem}
\end{equation}
is nonconvex whenever $H$ is indefinite, and it nonetheless admits necessary \emph{and}
sufficient conditions for a global solution.

\begin{theorem}[Global optimality, \citealp{Gay81, Sorensen82, MoreSorensen1983}]
\label{thm: ms conditions}
The vector $s^\star$ solves~\eqref{eqn: subproblem} globally if and only if there is
$\lambda \geq 0$ with
\begin{equation}
    (H + \lambda I)\, s^\star = -g,
    \qquad H + \lambda I \succeq 0,
    \qquad \lambda \bigl(\Delta - \norm{s^\star}\bigr) = 0 .
    \label{eqn: ms conditions}
\end{equation}
\end{theorem}

The last equality is complementarity. Either $\lambda = 0$ and the constraint is inactive, or
$\norm{s^\star} = \Delta$. Sufficiency is what makes~\eqref{eqn: ms conditions} usable, a
nonconvex problem rarely admitting a checkable certificate of global optimality.

\texttt{SteihaugCG}, the truncated conjugate-gradient method of \citet{Stei83} and \citet{Toint81}, is stated in
Appendix~\ref{app: frame} (Algorithm~\ref{alg: SteihaugCG}) and will be the default subproblem solver throughout.
It needs only products $Bv$, so it costs nothing to run at $n = 1000$ and it
is the only solver the quasi-Newton models can be driven with as operators. The norm
$\norm{s_j}$ increases strictly along the inner iterations while the curvature stays
positive~\citep[Thm.~7.3.4]{ConnGoulToin00}, so an iterate that leaves the region cannot re-enter
it and stopping at the first exit loses nothing. Truncation delivers the Cauchy decrease
\begin{equation}
    m_k(0) - m_k(s_k) \;\geq\; \tfrac12 \norm{g_k}
        \min\left\{ \frac{\norm{g_k}}{\norm{H_k}},\ \Delta_k \right\} ,
    \label{eqn: cauchy decrease}
\end{equation}
because the first inner iterate is the Cauchy point, and~\eqref{eqn: cauchy decrease} is what
first-order global convergence rests on. The \texttt{EigenPoint} wrapper wraps it and guarantees at
least the decrease available along the leftmost eigenvector,
$\tfrac12\abs{\lambda_{\min}}\Delta^2$, which truncated conjugate gradients alone does not.

Truncation delivers a second guarantee, and this one \PartII{} assumes rather than proves. Its
quadratic model decrease condition
\begin{equation}
    m_k(0) - m_k(s_k) \;\geq\; \mykappa{mqd}\norm{s_k}^2 \;>\; 0
    \label{eqn: mqd}
\end{equation}
is stated for an exact minimiser of the model, and \PartII{} leaves the inexact case to this
\thisdoc{}. Appendix~\ref{app: mqd} settles it. Subspace optimality survives truncation as an
inequality of the right sign, $\langle H_ks_k + g_k,\, s_k\rangle \leq 0$, and that inequality
gives $m_k(0) - m_k(s_k) \geq \tfrac12 s_k^\top H_ks_k$ for any symmetric $H_k$ and whichever of
the three exit tests of Algorithm~\ref{alg: SteihaugCG} fires. What does not survive is the
constant. Reaching $\mykappa{mqd}\norm{s_k}^2$ from $\tfrac12 s_k^\top H_ks_k$ needs a uniform
$\lambda_{\min}(H_k) \geq \bar\delta > 0$, so~\eqref{eqn: mqd} is a property of the model and not
of the solver, and it fails on the models whose indefiniteness is the information being used.

\begin{remark}[Degeneration to the Cauchy point]
\label{rem: cauchy degeneration}
The first inner direction is $-g_k$. If the iterate leaves the region on the \emph{first} inner
iteration the returned step is the Cauchy point, and the model Hessian has had no influence on
the direction, though it still sets the length through $g_k^\top B_k g_k$. When $\Delta_k$ is
small, which is what a small cap in $\Rgrad$ or a small $\zeta$ in $\Rdfo$ produces, this happens
at every iteration and the method is gradient descent in disguise. No first-order diagnostic
reveals it. The driver therefore records the number of inner iterations and $\cos(s_k, -g_k)$,
whose signature here is $\texttt{cg\_iters} = 1$ together with $\cos = 1$ and
$\texttt{active} = \textnormal{true}$.
\end{remark}

The package offers the possibility to use several solvers for the inner problem:
\texttt{ExactMS}  \citep{MoreSorensen1983}, \texttt{KrylovCG}, \texttt{KrylovCR} and
\texttt{KrylovCGLanczos} delegate to \textsf{Krylov.jl} and need only the product $Bv$;
\texttt{KrylovCG} assumes $H \succ 0$ and so cannot be used where indefiniteness is the
information being sought.

\subsection{The model Hessians}
\label{subsec: models}

The driver never forms a matrix. It uses only the action $H v$, which is what lets the
quasi-Newton models be supplied as operators and what keeps $n = 1000$ affordable.
Table~\ref{tab: models} lists the eight models.

\begin{table}[htbp]
    \centering
    \small
    \begin{tabular}{@{}llp{0.50\textwidth}@{}}
        \toprule
        Model & Indefinite & Role \\
        \midrule
        \texttt{ExactHessian}  & yes & $H_k = \nabla^2 f(x_k)$. Does not satisfy the secant
            equation, which closes the route of \citet{FanPan16} to
            $\widetilde\rho \to 1$. \\
        \texttt{LBFGSModel}    & no  & Positive definite by construction, so it never reports
            negative curvature~\citep{NoceWrig06}. \\
        \texttt{SR1Model}      & yes & Can represent negative curvature, which is the reason to
            carry it~\citep{ConnGouldToint1991}. \\
        \texttt{ScaledIdentity} & no & $H_k = cI$. No curvature at all, so the step is
            $-\min\{\norm{g_k}/c,\, \Delta_k\}\,g_k/\norm{g_k}$ and any difference between rules
            is the rule. \\
        \texttt{SPDTarget}     & no  & Positive definite with its unconstrained minimiser pinned
            to a chosen point. Section~\ref{subsec: spd}. \\
        \bottomrule
    \end{tabular}
    \caption{The eight model Hessians. Limited-memory models run at memory $5$ throughout. The
    grid of Section~\ref{subsec: interaction} uses the first four. \texttt{SPDTarget} is a
    construction for P3 and appears in no aggregate campaign, and the last three apply only to
    the structured problem classes.}
    \label{tab: models}
\end{table}

    \subsubsection{A positive definite model with a chosen minimiser}
\label{subsec: spd}

Claim~\ref{claim: bad limits} asserts that a configuration can satisfy every hypothesis of
\PartI{}, keep every $\rho_k$ successful and a well-behaved radius, and still converge to a point
that is not a minimiser. Exhibiting that on P3 by the invariant line of
Section~\ref{subsec: P3} leaves one objection open, since a subproblem solver that inspects the
curvature off the line would escape. \texttt{SPDTarget} removes the objection by removing the
negative curvature from the model altogether.

The construction is two-dimensional, which is the dimension of P3. Write $d = x^\star - x$ for
the displacement to the chosen target, $u = d/\norm{d}$ and $v = u^\perp$, and set
\begin{equation}
    H = a\,uu^\top + b\,\bigl(uv^\top + vu^\top\bigr) + c\,vv^\top ,
    \qquad
    a = -\frac{g^\top u}{\norm{d}} ,
    \qquad
    b = -\frac{g^\top v}{\norm{d}} ,
    \qquad
    c = \frac{b^2}{a} + \lambda_\perp ,
    \label{eqn: spd target}
\end{equation}
with $\lambda_\perp > 0$. Two facts make~\eqref{eqn: spd target} the right instrument.

The unconstrained minimiser of the model is $x^\star$ for \emph{every} $c$. Since $u$ and $v$ are
orthonormal and $d = \norm{d}u$, the matrix satisfies $Hd = \norm{d}(a u + b v) = -g$, which is
the Newton equation for the step $d$.

The matrix is positive definite exactly when $a > 0$, the choice of $c$ then giving
$ac - b^2 = a\lambda_\perp > 0$. In terms of the point,
\begin{equation}
    a > 0
    \iff
    \varphi(x) \defeq g(x)^\top\bigl(x^\star - x\bigr) < 0 ,
    \label{eqn: phi target}
\end{equation}
so the target has to lie downhill. That condition is not an artefact of this particular family.
Any positive definite $H$ with $Hd = -g$ has $d^\top H d = -g^\top d > 0$, so no positive
definite model whatever has its minimiser at $x^\star$ where $\varphi(x) \geq 0$. A failure of
the construction is a statement about the point, and the package reports it as such.

For P3 the target is the nonsingular saddle $x^\star = (\tfrac34, 0)$, and the condition is
available in closed form. On the invariant line,
\begin{equation}
    \varphi(x,0) = \bigl(4x^3 - 3x^2\bigr)\Bigl(\tfrac34 - x\Bigr)
                 = -\tfrac14\,x^2 (4x-3)^2 \;\leq\; 0 ,
    \label{eqn: p3 phi}
\end{equation}
with equality only at the two critical points $x = 0$ and $x = \tfrac34$ themselves. The model
therefore exists at every other point of the line, and it exists on a neighbourhood of it. Runs
start at $x_0 = (-0.5,\, 0.6)$, which is off the line and has $\varphi(x_0) = -2.277 < 0$, and
converge to the saddle under every rule and every subproblem solver. Every $\rho_k$ is
successful, $\norm{g_k} \to 0$, the radius is well behaved, and the limit is not a minimiser. The
failure belongs to the model, and no radius rule can repair it.

\subsection{Stopping tests and run outcomes}
\label{subsec: stopping}

A run ends in exactly one of seven states. There are two tolerances, and the second one changes
what the first one means.

\begin{itemize}[leftmargin=1.6em, itemsep=2pt]
    \item \texttt{first\_order}, $\norm{g_k} \leq \epsilon$ with the curvature test disabled,
          $\epsilon_H \leq 0$.
    \item \texttt{second\_order}, $\norm{g_k} \leq \epsilon$ and
          $\lambda_{\min}(H_k) \geq -\epsilon_H$ with $\epsilon_H > 0$.
    \item \texttt{max\_iter}, the iteration budget $k_{\max}$ is exhausted.
    \item \texttt{max\_time}, the wall-clock budget $t_{\max}$ is exhausted.
    \item \texttt{stalled}, $\abs{\text{ared}}$ has fallen to the rounding level of $f$ and the
          step is numerically nil. The ratio is then noise, every step is rejected and the radius
          contracts for ever. This is a real limit of the criticality-anchored rules in
          floating-point arithmetic rather than an implementation defect, so it is reported
          separately instead of being counted as an ordinary failure.
    \item \texttt{exception}, the model could not be built at the current iterate. For
          \texttt{SPDTarget} this is~\eqref{eqn: phi target} failing.
    \item \texttt{user}, a callback stopped the run.
\end{itemize}

A run counts as solved when it ends in \texttt{first\_order} \emph{or} in \texttt{second\_order}.
The disjunction is not a convenience. \texttt{second\_order} is the stronger outcome, not a
different kind of failure, and testing for \texttt{first\_order} alone marks every successful
$\tau$-anchored run as unsolved, so the configurations that do best appear in the profiles with a
reliability of zero. The remaining five states contribute $+\infty$ to the cost matrix of
Section~\ref{subsec: methodology}. A solver that stops early with a large gradient has not solved
the problem whatever its iteration count.

One convention needs stating because two reporting paths disagree about it. A problem already
critical at $x_0$ returns \texttt{first\_order} after zero iterations, and a cost of zero is
treated as a failure by the profile and as a success by the solve-rate table. Such problems are
screened out of the test set before the cost matrix is built.

\subsection{Settings}
\label{subsec: settings}

Every constant used anywhere in this \thisdoc{} is fixed in Table~\ref{tab: settings}. Parameters are
identical across mechanisms, or not tuned at all. Tuning each rule separately would measure the
tuning effort rather than the mechanism, which is the first recommendation of
\citet{BeiranvandHareLucet17}.

\begin{table}[htbp]
    \centering
    \small
    \begin{tabular}{@{}lllp{0.24\textwidth}@{}}
        \toprule
        & Symbol & Value & Applies to \\
        \midrule
        acceptance threshold      & $\eta$                      & $0.1$ & every rule \\
        scaling thresholds        & $\eta_1, \eta_2$            & $0.1$, $0.9$ & every rule \\
        retrospective thresholds  & $\widetilde\eta_1, \widetilde\eta_2$ & $0.05$, $0.9$
                                  & $\Rrtr$, $\Rgrtr$ \\
        branch factors            & $\gone, \gtwo, \gthree$    & $0.25$, $0.5$, $2$
                                  & the fixed-factor rules \\
        adaptive factor constants & $\gone, \gtwo, \gthree, M, \lambda_i$
                                  & $0.01$, $0.9$, $1.5$, $5$, $5$
                                  & $\RAdaptstep$, $\RAdaptiveGrad$ \\
        \midrule
        initial radius            & $\Delta_0$                  & $1$ & every rule \\
        radius floor, cap         & $\Delta_{\min}, \Delta_{\max}$ & $0.0$, $+\infty$
                                  & every rule \\
        criticality parameter     & $\zeta$                     & $100$ & $\Rdfo$ \\
        initial multiplier   &      $\mu_0$    & $1$ & $\Rgrad$ \\
        \midrule
        forcing constants, inner cap & $\chi, \theta, j_{\max}$ & $0.1$, $0.5$, $1000$
                                  & \texttt{SteihaugCG} \\
        quasi-Newton memory       & n/a                          & $5$
                                  & \texttt{LBFGSModel}, \texttt{SR1Model} \\
        transverse eigenvalue     & $\lambda_\perp$             & $1$ & \texttt{SPDTarget} \\
        \midrule
        first-order tolerance     & $\epsilon$                  & $10^{-5}$
                                  & every campaign but one \\
        curvature tolerance       & $\epsilon_H$                & $10^{-6}$
                                  & the second-order campaigns \\
        iteration and time budget & $k_{\max}, t_{\max}$        & $10^{4}$, $120$\,s
                                  & every campaign \\
        \bottomrule
    \end{tabular}
    \caption{Every constant, and where it applies. The adaptive rules are the one place where
    the branch factors are not shared, and Section~\ref{subsec: like for like} explains why matching
    them would not repair the comparison. The tolerance is $10^{-10}$ in
    Section~\ref{subsec: order rules} alone, where the asymptotic regime is the quantity being
    measured and a run stopped at $10^{-5}$ may never enter it. Setting $\epsilon_H \leq 0$
    disables the curvature test, and every run then ends in \texttt{first\_order}.}
    \label{tab: settings}
\end{table}

No campaign imposes a radius floor. $\Delta_{\min} = 0$ is a driver parameter and takes that value
under every rule in every campaign reported here, which is what makes the comparison one at fixed
constants. The choice is not neutral for $\Rstep$. On an accepted step that rule sets
$\Delta_{k+1} = \gthree\norm{s_k}$, so a short step gives a small radius and a shorter step still.
Near convergence, or whenever truncated conjugate gradients stops at its first inner iteration,
the sequence can collapse towards zero, after which every step is nil, $\rho_k$ is \texttt{NaN}
and the loop runs to $k_{\max}$. A floor would convert that hang into slowness. Running without
one leaves the rule unprotected, and the iteration counts $\Rstep$ records below are its own.

\paragraph{Where the runs were made.}
Every campaign ran on one machine, an Intel Core~i9-13900 carrying $24$ cores and $32$ threads,
eight of the cores being performance cores and sixteen efficiency cores, with $36$\,MB of shared
last-level cache and $128$\,GB of memory, under Linux and with Julia~1.12.1 on a single thread.
Only $t_{\max}$ makes the machine matter. A run is unsolved when it exhausts $10^{4}$ iterations
\emph{or} $120$ seconds of wall clock, so the second half of that test is a property of this
machine rather than of the mechanism, and on a processor carrying two kinds of core it is a
property of which kind the run was scheduled on. Every count of solved problems below inherits
that dependence. The iteration and evaluation counts do not, and they are what every comparison
in this \thisdoc{} reads, time being reported only where the comparison is between subproblem solvers.

\subsection{Comparing runs}
\label{subsec: methodology}

A run matrix executes every configuration on every problem, resetting the rule, the model, the
subproblem solver and the oracle between runs, and records the cost of each solved run. The cost
is one of iterations, objective evaluations, gradient evaluations, Hessian-vector products or
wall-clock time, and the choice is stated with every figure. Time is reported only where the
comparison is between subproblem solvers, being the one measure that depends on the
implementation as much as on the mechanism. From that matrix the package builds the profiles
below and exports them as \textsc{pgfplots} coordinates, so that the figures in this \thisdoc{} are
generated from the matrix that produced the tables.

\paragraph{Performance profiles.}
For a set $\calP$ of problems and a set $\calS$ of configurations, let $t_{p,s}$ be the cost of
configuration $s$ on problem $p$. The performance ratio is
\begin{equation}
    r_{p,s} = \frac{t_{p,s}}{\min\{t_{p,\sigma} : \sigma \in \calS\}} ,
    \label{eqn: performance ratio}
\end{equation}
set to $+\infty$ when $s$ does not end in \texttt{first\_order} or \texttt{second\_order} on $p$,
and the profile of \citet{DolanMore2002} is the empirical distribution
\begin{equation}
    \pi_s(\tau) = \frac{1}{\abs{\calP}}
        \abs{\{p \in \calP : r_{p,s} \leq \tau \}} ,
    \label{eqn: performance profile}
\end{equation}
the fraction of problems on which $s$ is within a factor $\tau$ of the best configuration in the
comparison.

Two numbers are read off a profile and they answer different questions. The intercept $\pi_s(1)$
is \emph{efficiency}, the fraction of problems on which $s$ is fastest. The right-hand asymptote
is \emph{reliability}, the fraction $s$ solves at all. The two move independently, and the
central observation of Section~\ref{sec: cost} is that a parameter choice can leave one untouched
while collapsing the other.

\paragraph{Data profiles.}
A performance profile normalises by the best configuration, which makes it a statement about
relative cost and not about absolute effort. The data profile of \citet{MoreWild2009} normalises
by the problem dimension instead,
\begin{equation}
    d_s(\alpha) = \frac{1}{\abs{\calP}}
        \left|\left\{ p \in \calP : \frac{t_{p,s}}{n_p + 1} \leq \alpha \right\}\right| ,
    \label{eqn: data profile}
\end{equation}
so $\alpha$ is a budget in units of one simplex gradient and $d_s(\alpha)$ answers what fraction
of problems $s$ solves within that budget. Not being normalised by the best configuration, it
does not inherit the set dependence noted next.

\paragraph{What a profile depends on.}
A performance ratio is normalised by the best configuration in $\calS$, so a profile is a
statement about a comparison and not about a configuration on its own. Adding or removing a
configuration changes every curve, and profiles drawn from different $\calS$ are not comparable,
which is why~\citet{GouldScott16} caution against reading a profile as a ranking once more than
two configurations are compared. Every profile below therefore names the set it is drawn from,
and no curve is carried from one figure to another. Reliability survives that choice, since a
configuration either solves a problem or it does not. Efficiency does not, and every efficiency
quoted below is an efficiency within the comparison its table names.

\section{Experiments}
\label{sec: cost}

\subsection{The \texorpdfstring{$\zeta$}{zeta} sweep}
\label{subsec: zeta sweep}

$\Rdfo$ ran at six values of $\zeta$ from $10^{-3}$ to $10^{2}$, one per decade, on every
unconstrained CUTEst problem with $2 \leq n \leq 1000$, which gives $196$ problems. Every other
component was held fixed. The model is the exact Hessian, the subproblem solver is truncated
conjugate gradients with $\chi = 0.1$, $\theta = 0.5$ and at most $1000$ inner iterations, the
criticality measure is the first-order one, $\omega_k = \norm{g_k}$, and the constants are
$\gone = 0.25$, $\gtwo = 0.5$, $\gthree = 2$, $\eta = \eta_1 = 0.1$, $\eta_2 = 0.9$ and
$\Delta_0 = 1$ with no radius floor and no radius cap. A run is solved when it reaches
$\norm{g_k} \leq 10^{-5}$, the second-order stopping test being disabled, and unsolved when it
exhausts $10^{4}$ iterations or $120$ seconds, stalls, or raises an exception.
Section~\ref{sec: threshold} showed that $\Rdfo$ propagates the radius multiplicatively, so a
generous $\Delta_0$ can rescue an otherwise fatal $\zeta$. Fixing $\Delta_0 = 1$ across the sweep
is what makes Table~\ref{tab: zeta sweep} a statement about $\zeta$.

\begin{table}[htbp]
    \centering
    \small
    \begin{tabular}{@{}rrrrrr@{}}
        \toprule
        $\zeta$ & solved & efficiency $\pi(1)$ & median iterations & mean iterations
          & tail activity \\
        \midrule
        $0.001$ & $68/196 = 0.347$  & $0.194$ & $19.5$ & $1297.1$ & $0.410$ \\
        $0.01$  & $96/196 = 0.490$  & $0.265$ & $21.0$ & $607.3$  & $0.376$ \\
        $0.1$   & $118/196 = 0.602$ & $0.403$ & $18.0$ & $545.7$  & $0.250$ \\
        $1$     & $145/196 = 0.740$ & $0.434$ & $19.0$ & $456.9$  & $0.254$ \\
        $10$    & $153/196 = 0.781$ & $0.459$ & $19.0$ & $238.6$  & $0.198$ \\
        $100$   & $162/196 = 0.827$ & $0.607$ & $19.5$ & $237.7$  & $0.151$ \\
        \bottomrule
    \end{tabular}
    \caption{$\Rdfo$ on the $196$-problem CUTEst subset, against $\zeta$. Efficiency is the
    intercept $\pi(1)$ of the iteration profile of Figure~\ref{fig: zeta profile}. The two
    iteration columns are taken over the runs each configuration solves, and those subsets
    differ from row to row. Tail activity is the fraction of the last tenth of the iterations
    on which the trust-region constraint bound.}
    \label{tab: zeta sweep}
\end{table}

\begin{figure}[htbp]
    \centering
    \includegraphics[width=0.78\textwidth]{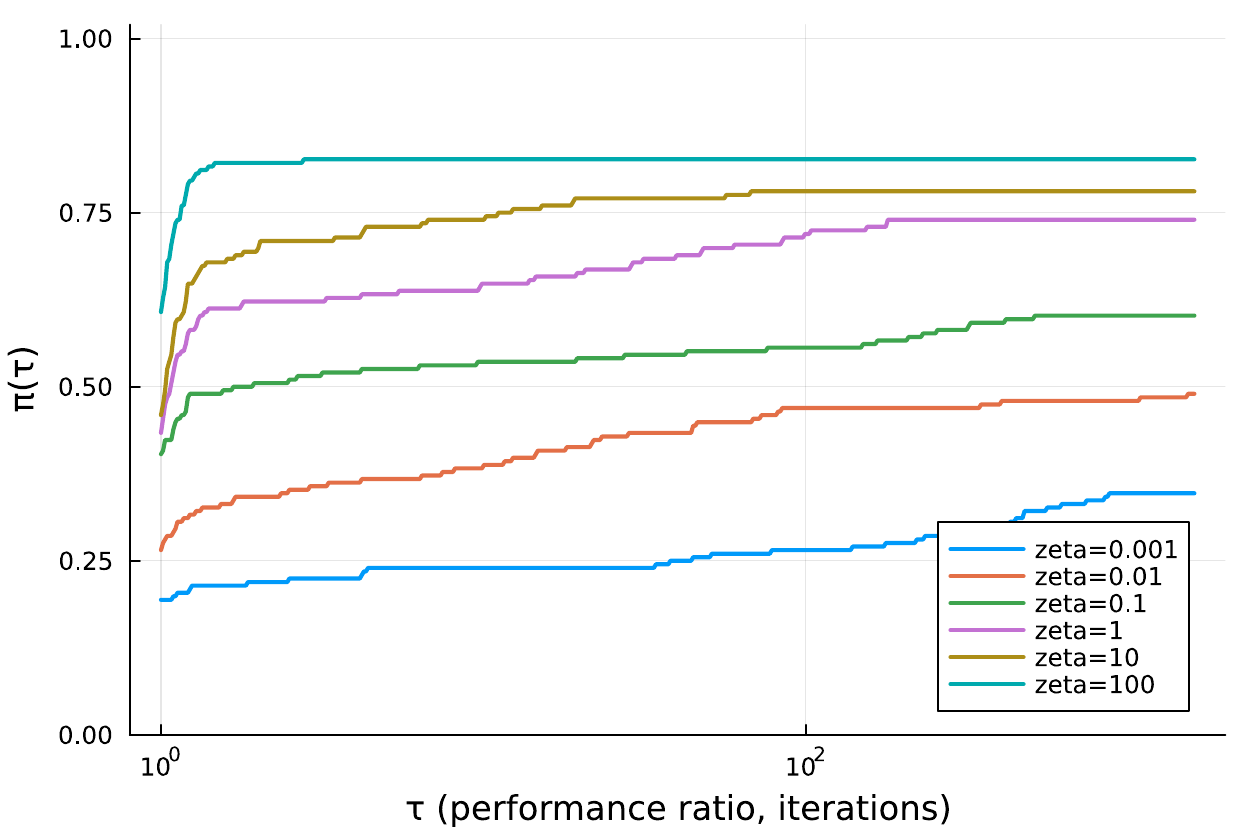}
    \caption{$\Rdfo$ against $\zeta$ on the $196$-problem subset: the iteration performance
        profile, as the run matrix exported it. The six curves are totally ordered and none
        crosses another, so raising $\zeta$ never costs a problem at any $\tau$. The legend
        carries the configuration names of the run record, \texttt{zeta=}$\zeta$.}
    \label{fig: zeta profile}
\end{figure}

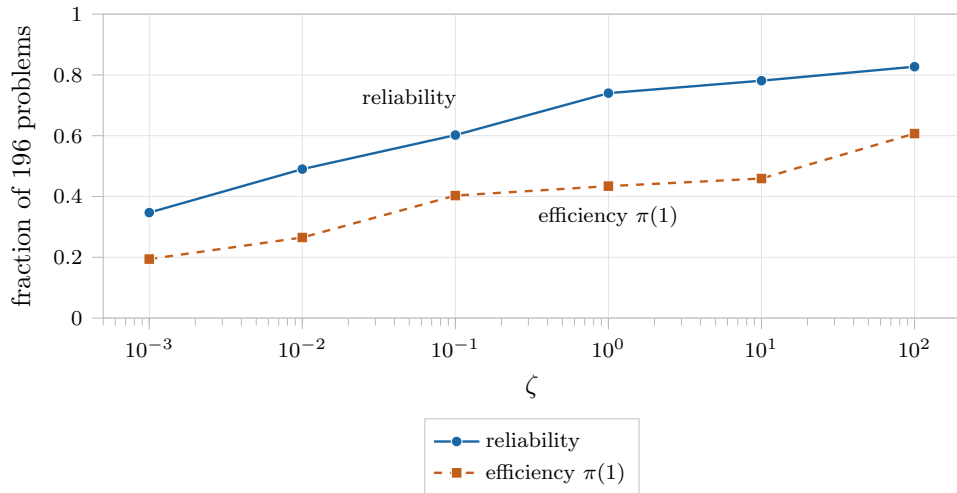
\begin{figure}[htbp]
    \centering
    \begin{tikzpicture}
        \begin{semilogxaxis}[rate axis, width=0.76\textwidth,
            xlabel={$\zeta$}, xmin=5e-4, xmax=200]
            \addplot[rel line] coordinates
                {(0.001,0.347) (0.01,0.490) (0.1,0.602) (1,0.740) (10,0.781) (100,0.827)};
            \addplot[eff line] coordinates
                {(0.001,0.194) (0.01,0.265) (0.1,0.403) (1,0.434) (10,0.459) (100,0.607)};
            \legend{reliability, efficiency $\pi(1)$}
            \node[anchor=south, font=\footnotesize] at (axis cs:0.05,0.66) {reliability};
            \node[anchor=north, font=\footnotesize] at (axis cs:1,0.40)
                {efficiency $\pi(1)$};
        \end{semilogxaxis}
    \end{tikzpicture}
    \caption{The two numbers the profile of Figure~\ref{fig: zeta profile} is read for, taken
        from Table~\ref{tab: zeta sweep}. Reliability is the right-hand asymptote of the profile
        and efficiency its intercept $\pi(1)$. The two rise together, which
        Section~\ref{subsec: zeta sweep} reads as one effect and not two.}
    \label{fig: zeta rates}
\end{figure}

The reading is asymmetric and it is the point of the experiment. Reliability rises monotonically
from $68$ to $162$ problems of $196$, a gain of $94$ problems over five decades of $\zeta$. The
median iteration count on the problems each configuration solves does not move, staying between
$18.0$ and $21.0$ over the same range. A small $\zeta$ therefore leaves the easy problems alone
and fails on the rest. Tail activity falls monotonically from $0.410$ to $0.151$, apart from a
reversal of $0.004$ between $\zeta = 0.1$ and $\zeta = 1$ that is smaller than the resolution of
the measure. That fall is the mechanism of Claim~\ref{claim: zeta} showing through the
aggregate. Below the threshold the gate $\Delta_k \leq \zeta \omega_k$ opens rarely, the radius
contracts on the intermediate branch at almost every successful iteration, and the constraint
never stops binding.

The flat median needs a caution, because two other columns contradict the reading it invites.
The mean iteration count falls by a factor of $5.5$, from $1297.1$ to $237.7$, and the profile
of Figure~\ref{fig: zeta profile} makes the same point without an average. Writing
$\tau_{0.9}$ for the factor at which a configuration reaches nine tenths of its own reliability,
$\tau_{0.9}$ falls from $766$ at $\zeta = 0.001$ to $1.2$ at $\zeta = 100$. A small $\zeta$ is
therefore slow, on a tail of problems that the median cannot see. The median is taken over the
solved subset, and that subset grows from $68$ problems to $162$ as $\zeta$ rises, so the rows
of Table~\ref{tab: zeta sweep} do not describe the same problems. The honest statement is that
$\zeta$ costs reliability first and cost second, and that no column of this table is a
like-for-like comparison of speed.

Efficiency rises with reliability rather than independently of it. A configuration that fails on
a problem is never fastest on it, so $\pi(1)$ is bounded above by reliability and the two are not
free to move separately. Dividing one by the other gives the fraction of \emph{solved} problems
on which a configuration is fastest, which takes the values $0.56$, $0.54$, $0.67$, $0.59$,
$0.59$ and $0.74$ in the order of the table. That ratio carries no trend across the first five
values and rises only at $\zeta = 100$. Almost all of the threefold gain in $\pi(1)$ is therefore
the gain in reliability, and the two curves of Figure~\ref{fig: zeta rates} rise together rather
than apart. What is flat in $\zeta$ is the cost of a solved run, in the median column of
Table~\ref{tab: zeta sweep}, and not the efficiency intercept.

The transition is smooth here and sharp in Table~\ref{tab: sinc numerics}. That is expected
rather than contradictory. Claim~\ref{claim: zeta} asserts a threshold that depends on the
solution, and averaging over $196$ problems with $196$ different thresholds smears any one of
them out. The aggregate regime can show the mechanism and cannot settle
Claim~\ref{claim: zeta}. Nothing here locates a threshold in $\zeta$, and the grid of one point
per decade could not resolve one if it existed.

\subsection{The \texorpdfstring{$\overline{\mu}$}{mu-bar} sweep}
\label{subsec: mu sweep}

$\Rgrad$ ran at six values of the cap from $10^{-3}$ to $10^{2}$ and once uncapped, on the same
$196$ problems and under the settings of Section~\ref{subsec: zeta sweep}. The starting
multiplier is $\mu_0 = 0.1$ for the uncapped comparator and for every cap that admits it, so the
four configurations at $\overline{\mu} \in \{0.1, 1, 10, 100\}$ and the uncapped one differ in the
cap and in nothing else. Update~\ref{update: rgrad generic} requires $\mu_0 \leq \overline{\mu}$,
so the two smallest caps cannot start there and keep $\mu_0 = \overline{\mu}$. Those two rows
move the initial multiplier with the cap and the other five do not.

\begin{table}[htbp]
    \centering
    \small
    \begin{tabular}{@{}lrrrrrr@{}}
        \toprule
        $\overline{\mu}$ & $\mu_0$ & solved & efficiency $\pi(1)$ & median iterations
          & mean iterations & tail activity \\
        \midrule
        $0.001$  & $0.001$ & $53/196 = 0.270$  & $0.066$ & $1055.0$ & $2206.4$ & $0.736$ \\
        $0.01$   & $0.01$  & $80/196 = 0.408$  & $0.077$ & $235.5$  & $806.1$  & $0.746$ \\
        \midrule
        $0.1$    & $0.1$   & $102/196 = 0.520$ & $0.153$ & $45.5$   & $1045.1$ & $0.582$ \\
        $1$      & $0.1$   & $122/196 = 0.622$ & $0.342$ & $29.0$   & $533.1$  & $0.424$ \\
        $10$     & $0.1$   & $139/196 = 0.709$ & $0.413$ & $33.0$   & $526.4$  & $0.414$ \\
        $100$    & $0.1$   & $150/196 = 0.765$ & $0.485$ & $25.5$   & $369.4$  & $0.350$ \\
        uncapped & $0.1$   & $161/196 = 0.821$ & $0.740$ & $26.0$   & $176.1$  & $0.299$ \\
        \bottomrule
    \end{tabular}
    \caption{$\Rgrad$ on the $196$-problem CUTEst subset, against the cap. The five rows below
    the rule share $\mu_0 = 0.1$ and vary only in $\overline{\mu}$, the last of them being
    $\overline{\mu} = +\infty$. The two rows above it cannot start at $0.1$ and are the only ones
    that move two parameters at once. The remaining columns are those of
    Table~\ref{tab: zeta sweep} and carry the same caution about the solved subset.}
    \label{tab: mu sweep}
\end{table}

\begin{figure}[htbp]
    \centering
    \includegraphics[width=0.78\textwidth]{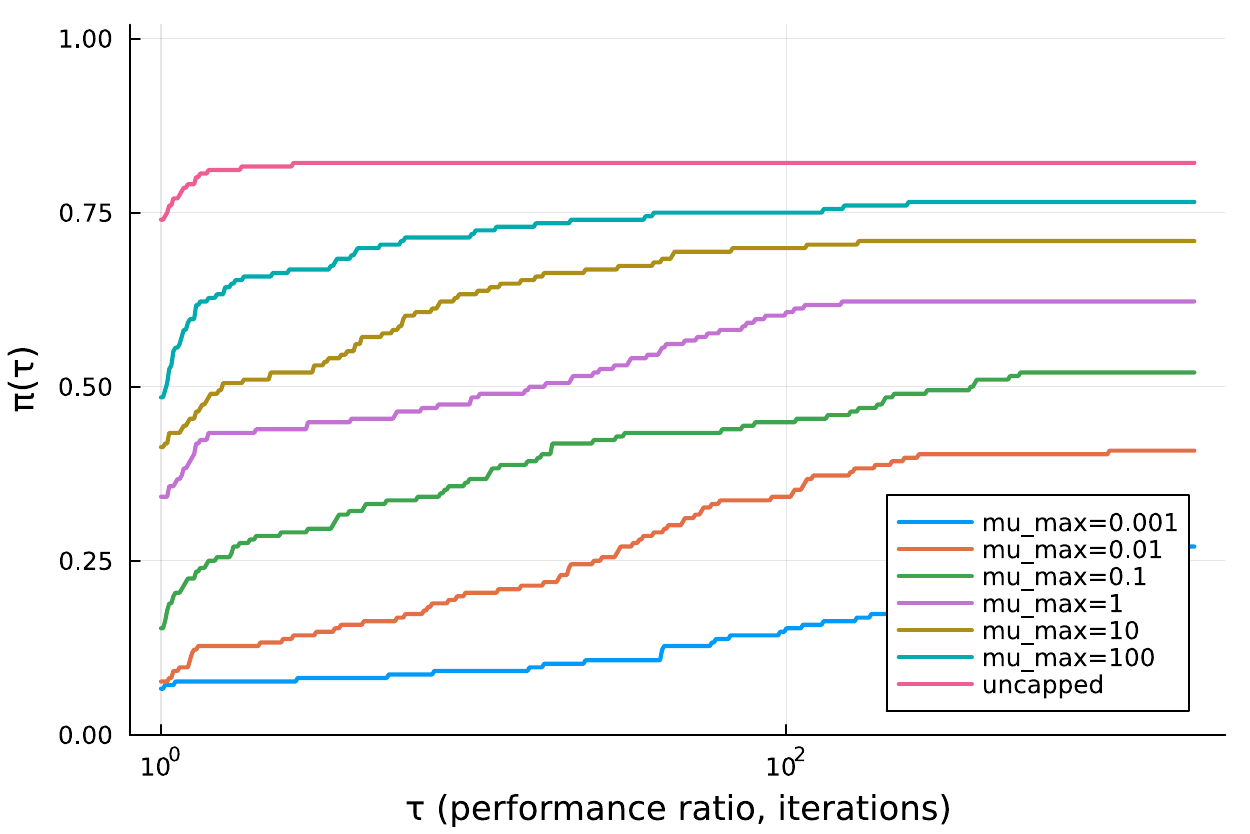}
    \caption{$\Rgrad$ against the cap on the $196$-problem subset: the iteration performance
        profile, as the run matrix exported it, and read as Figure~\ref{fig: zeta profile} is
        read. The seven curves are totally ordered and none crosses another. The legend carries
        the configuration names of the run record, \texttt{mu\_max=}$\overline{\mu}$ and
        \texttt{uncapped} for $\overline{\mu} = +\infty$.}
    \label{fig: mu profile}
\end{figure}

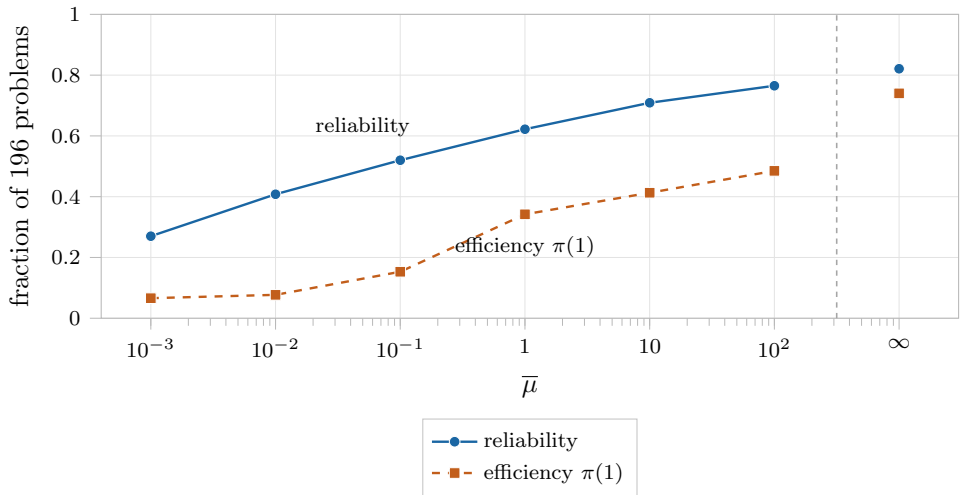
\begin{figure}[htbp]
    \centering
    \begin{tikzpicture}
        \begin{semilogxaxis}[rate axis, width=0.76\textwidth, xlabel={$\overline{\mu}$},
            xmin=4e-4, xmax=3000,
            xtick={0.001,0.01,0.1,1,10,100,1000},
            xticklabels={$10^{-3}$,$10^{-2}$,$10^{-1}$,$1$,$10$,$10^{2}$,$\infty$}]
            \addplot[rel line] coordinates
                {(0.001,0.270) (0.01,0.408) (0.1,0.520) (1,0.622) (10,0.709) (100,0.765)};
            \addplot[eff line] coordinates
                {(0.001,0.066) (0.01,0.077) (0.1,0.153) (1,0.342) (10,0.413) (100,0.485)};
            \legend{reliability, efficiency $\pi(1)$}
            \addplot[rel line, only marks, forget plot] coordinates {(1000,0.821)};
            \addplot[eff line, only marks, forget plot] coordinates {(1000,0.740)};
            \draw[ref rule] (axis cs:316,0) -- (axis cs:316,1);
            \node[anchor=south, font=\footnotesize] at (axis cs:0.05,0.57) {reliability};
            \node[anchor=north, font=\footnotesize] at (axis cs:1,0.30)
                {efficiency $\pi(1)$};
        \end{semilogxaxis}
    \end{tikzpicture}
    \caption{Reliability and efficiency of $\Rgrad$ against the cap, taken from
        Table~\ref{tab: mu sweep} and read as Figure~\ref{fig: zeta rates} is read. The uncapped
        configuration has $\overline{\mu} = +\infty$, so it sits beyond the dashed rule as two
        isolated marks rather than on the curve. Both curves rise, and unlike
        Figure~\ref{fig: zeta rates} the lower one rises faster, so the gap between them closes
        as the cap is relaxed.}
    \label{fig: mu rates}
\end{figure}

The cap behaves differently from $\zeta$, and the difference is in the column that does not move
in Table~\ref{tab: zeta sweep}. Reliability rises from $53$ to $150$ problems and the uncapped
variant solves $161$, which is the shape of the $\zeta$ sweep. The median iteration count is not
flat. It falls by a factor of forty, from $1055.0$ to $26.0$, so a small cap costs efficiency as
well as reliability. Dividing $\pi(1)$ by reliability gives $0.25$, $0.19$, $0.29$, $0.55$,
$0.58$, $0.63$ and $0.90$ down the table, which rises where the same ratio in the $\zeta$ sweep
carries no trend. Even among the problems a capped run solves, it is rarely the fastest. A small
$\zeta$ delays the moment the radius is allowed to grow. A small cap forbids it from ever growing
past $\overline{\mu}\norm{g_k}$, and the method pays for that on every problem and not only on
the ones it loses.

\paragraph{The cap is isolated over most of the range.}
At $\mu_0 = 0.1$ throughout, raising the cap from $0.1$ to $1$, $10$, $100$ and then removing it
takes reliability through $102$, $122$, $139$, $150$ and $161$ problems, and efficiency through
$0.153$, $0.342$, $0.413$, $0.485$ and $0.740$, both plotted against the cap in
Figure~\ref{fig: mu rates}. Both are monotone and neither curve of
Figure~\ref{fig: mu profile} crosses another at any $\tau$. That sequence is
Claim~\ref{claim: zeta} in the aggregate regime, at a fixed initial multiplier, and the cap is
worth $59$ problems of $196$ over that range.

An earlier archive ran the same sweep with $\mu_0 = \overline{\mu}$ in every row, and comparing
the two bounds what the initial multiplier is worth. Starting at $0.1$ rather than at the cap
costs $8$ problems at $\overline{\mu} = 1$, $4$ at $10$ and $1$ at $100$. The initial multiplier
therefore matters, and it matters less the more generous the cap, which is what a multiplier free
to climb should do. The two rows at $\overline{\mu} \leq 0.01$ still move both parameters, and on
that evidence the part attributable to $\mu_0$ there is small beside the part attributable to the
cap.

\paragraph{Where the cap binds.}
The driver records, per run, how many iterations the expansion branch was clipped at
$\overline{\mu}$ and the first iteration after which the constraint never binds again.
Table~\ref{tab: mu climb} splits the runs on the first quantity and reports the second.

\begin{table}[htbp]
    \centering
    \small
    \begin{tabular}{@{}lrrrr@{}}
        \toprule
        & \multicolumn{2}{c}{the cap bound}
          & \multicolumn{2}{c}{the region never released} \\
        \cmidrule(lr){2-3}\cmidrule(l){4-5}
        $\overline{\mu}$ & on how many runs & on how many iterations
          & where the cap bound & where it did not \\
        \midrule
        $0.001$  & $179/196$ & $8621$ & $163/179 = 0.91$ & $3/17 = 0.18$   \\
        $0.01$   & $170/196$ & $3291$ & $149/170 = 0.88$ & $8/26 = 0.31$   \\
        $0.1$    & $148/196$ & $2004$ & $120/148 = 0.81$ & $11/48 = 0.23$  \\
        $1$      & $117/196$ & $1189$ & $79/117 = 0.68$  & $17/79 = 0.22$  \\
        $10$     & $85/196$  & $391$  & $57/85 = 0.67$   & $21/111 = 0.19$ \\
        $100$    & $74/196$  & $62.5$ & $34/74 = 0.46$   & $22/122 = 0.18$ \\
        uncapped & $0/196$   & ---    & ---              & $36/196 = 0.18$ \\
        \bottomrule
    \end{tabular}
    \caption{The cap and eventual inactivity, over all $196$ runs of each configuration rather
    than over the solved ones. The cap \emph{bound} on a run when at least one expansion was
    clipped at $\overline{\mu}$. The second column is the median, over those runs alone, of
    the number of iterations per run on which an expansion was clipped, out of the $10^{4}$
    iterations a run is allowed. The region \emph{never released} when the trust-region
    constraint bound at the final iteration, so the inactivity index of
    Section~\ref{subsec: instruments} is undefined. The last two columns split that outcome on
    whether the cap bound, and their denominators are the two parts of the $196$ runs.}
    \label{tab: mu climb}
\end{table}

The two right-hand columns separate cleanly, and they do not separate by row. A run whose cap
never bound fails to release the constraint on between $18$ and $31$ per cent of problems,
whatever $\overline{\mu}$ is, and the uncapped configuration sits at $18$ per cent at the foot
of the same column. A run whose cap bound at least once fails to release it on between $46$ and
$91$ per cent, and that figure rises as $\overline{\mu}$ falls. At $\overline{\mu} = 10^{-3}$
the cap bound on $179$ of the $196$ runs, and across those runs the median number of iterations
carrying a clipped expansion is $8621$ of the $10^{4}$ allowed. Table~\ref{tab: mu theta}
settles what that means for the radius. Across the same $179$ runs the median fraction of
iterations with $\theta_k = \overline{\mu}$ is $1.0000$, so on the median run the multiplier
sits at the cap and the radius is $\overline{\mu}\norm{g_k}$ at every iteration. That fraction
falls to $0.62$ at $\overline{\mu} = 10^{2}$, so the effect belongs to the small caps rather
than to the cap as such. This is the mechanism of Claim~\ref{claim: zeta} measured directly
rather than inferred from the solve rate.

\begin{table}[htbp]
    \centering
    \small
    \begin{tabular}{@{}lrrrr@{}}
        \toprule
        & \multicolumn{3}{c}{fraction of iterations at the cap} & \\
        \cmidrule(lr){2-4}
        $\overline{\mu}$ & median & lower quartile & smallest & runs \\
        \midrule
        $0.001$ & $1.0000$ & $0.9963$ & $0.0138$ & $179$ \\
        $0.01$  & $1.0000$ & $0.9836$ & $0.0173$ & $170$ \\
        $0.1$   & $0.9997$ & $0.9436$ & $0.0088$ & $148$ \\
        $1$     & $0.9839$ & $0.3571$ & $0.0000$ & $117$ \\
        $10$    & $0.9473$ & $0.8293$ & $0.0160$ & $85$  \\
        $100$   & $0.6210$ & $0.2934$ & $0.0168$ & $74$  \\
        \bottomrule
    \end{tabular}
    \caption{How much of each run the multiplier spends at its cap. The ratio
    $\theta_k = \Delta_k/\norm{g_k}$ is the multiplier $\mu_k$ itself under
    $\Rgrad(\omega, \overline{\mu})$, so $\theta_k = \overline{\mu}$ exactly when the multiplier
    sits at the cap, tested to a relative $10^{-9}$. Each run contributes the fraction of its
    iterations on which that holds, and the three columns are the median, the lower quartile and
    the smallest of those fractions over the cap-bound runs of Table~\ref{tab: mu climb}, whose
    number is the last column. Two facts do not fit in the table. The median $\theta_k$ equals
    $\overline{\mu}$ on every row, so the multiplier is at the cap for more than half of the
    median run at every cap tested. At $\overline{\mu} \leq 10^{-2}$ the smallest $\theta_k$ of
    the median run also equals $\overline{\mu}$, so there the multiplier never leaves the cap.
    Table~\ref{tab: mu climb} counts expansions refused, which is a lower bound on the time
    spent at the cap rather than a measure of it, since an iteration that contracts moves
    $\mu_k$ below $\overline{\mu}$ and is never counted. This table measures that time directly.
    The last two columns show that the reading is about the median run and not about every run,
    and the lower quartile is not monotone in $\overline{\mu}$.}
    \label{tab: mu theta}
\end{table}

Two things weaken that reading and neither is repaired by more runs of the same kind. The split
is by outcome and not by design, since a run whose cap binds is a run that wanted a larger
radius, which is a property of the problem as much as of the cap. The correlation is therefore
not a controlled comparison. And a run that exhausts the iteration budget is counted as never
inactive by construction, so the left column inherits the failure counts of
Table~\ref{tab: mu sweep}. What survives both objections is the constancy of the right-hand
column across five decades of $\overline{\mu}$, which no property of the problems explains.

\paragraph{What the collection says about $\lambda^*_{\min}$.}
Capped $\Rgrad$ is trapped whenever $\overline{\mu} < 1/\lambda^*_{\min}$
by~\eqref{eqn: rgrad threshold}. Reading that backwards, a problem solved uncapped and not at
$\overline{\mu} = 100$ has $\lambda^*_{\min} < 10^{-2}$ at the point the uncapped run reaches,
provided that point is a nondegenerate minimiser and provided the one-dimensional argument of
Section~\ref{sec: threshold} carries to $n$ dimensions with $\lambda^*_{\min}$ the smallest
eigenvalue there. Eleven problems separate the two configurations. Neither proviso is checked
here, so this is a reading of the aggregate against the theory and not a measurement of any
$\lambda^*_{\min}$.

\paragraph{The Cauchy diagnostic does not yet discriminate.}
Claim~\ref{claim: cauchy} predicts that a cap small enough to truncate the first inner iteration
returns the Cauchy point, whose signature in the record of Remark~\ref{rem: cauchy degeneration}
is $\texttt{cg\_iters} = 1$ together with $\cos(s_k, -g_k) = 1$ and
$\texttt{active} = \textnormal{true}$. The diagnostic recorded for this campaign reports
$\texttt{cg\_iters} = 1$ and $\cos(s_k, -g_k) = 1$ to machine precision at every one of the six
caps, and $\texttt{active} = \textnormal{false}$ at every one of them. An inactive constraint
with a single inner iteration is a subproblem solved to the forcing tolerance in one step, not a
truncation, and the three columns take the same values at $\overline{\mu} = 10^{-3}$ and at
$\overline{\mu} = 10^{2}$. The diagnostic separates nothing as recorded and no claim rests on it.

\subsection{Like-for-like}
\label{subsec: like for like}

Eight of the eleven mechanisms of Table~\ref{tab: roster axes} ran on the same $196$ problems, with the
exact Hessian, truncated conjugate gradients at $\chi = 0.1$ and $\theta = 0.5$, and the settings
of Section~\ref{subsec: settings}. Six of the eight share $\gone = 0.25$, $\gtwo = 0.5$ and
$\gthree = 2$, which the run of the previous draft did not. $\Rdfo$ runs at $\zeta = 100$ and
$\Rgrad$ at $\mu_0 = 1$, the settings Sections~\ref{subsec: zeta sweep}
and~\ref{subsec: mu sweep} selected. Capped $\Rgrad$ is that same configuration with
$\overline{\mu} = 128$ and nothing else altered, so the two differ in one coordinate. $\Rrtr$
carries the retrospective thresholds $\widetilde\eta_1 = 0.05$ and $\widetilde\eta_2 = 0.9$,
which have no counterpart in the other rules.

The two adaptive rules are the exception and they are not comparable on the constants. They ran
at $\gone = 0.01$, $\gtwo = 0.9$, $\gthree = 1.5$, $M = 5$ and $\lambda_1 = \lambda_2 = 5$.
These constants are not an instance of the normalisation of \PartI{}, which would need
$\gtwo = 0.5$ to reach $R_{\eta_1}(\eta_1) = 1.5$ and then $R_{\eta_1} \leq 0.5$ below $\eta$,
where this factor approaches $0.9$. The two rules also ran at $\eta = \eta_1 = 0.1$, on the
boundary of the hypothesis $\eta < \eta_1$ of
Update~\ref{update: adaptive step-size dependent update}, where the four conditions still hold
because the factor jumps at $\eta_1$.
Matching those numbers to the other six would not repair the comparison, because in
$\RAdaptstep$ and $\RAdaptiveGrad$ the constants parametrise the smooth factor
$R_{\eta_1}$ of~\eqref{eqn: hei factor} rather than three branch multipliers, and
$R_{\eta_1}(\eta_1) = \gthree$ fixes the expansion at the acceptance threshold rather than in the
very successful branch. The factor axis of Table~\ref{tab: roster axes} is therefore not isolated
by this experiment. The two adaptive rows below are two configurations and not a measurement of
what an adaptive factor buys.

\begin{table}[htbp]
    \centering
    \small
    \begin{tabular}{@{}llrrrrr@{}}
        \toprule
        Rule & Anchor & solved & $\pi(1)$ & median & mean & tail act. \\
        \midrule
        $\Rrtr$            & retrospective & $169/196$ & $0.449$ & $19.0$ & $131.9$ & $0.138$ \\
        $\Rdelta$          & prev.\ radius & $165/196$ & $0.342$ & $18.0$ & $218.8$ & $0.129$ \\
        $\Rdfo$            & criticality   & $162/196$ & $0.301$ & $19.5$ & $237.7$ & $0.151$ \\
        $\RAdaptiveGrad$   & criticality   & $162/196$ & $0.240$ & $23.0$ & $166.2$ & $0.203$ \\
        $\RAdaptstep$      & step          & $149/196$ & $0.214$ & $20.0$ & $329.5$ & $0.298$ \\
        $\Rstep$           & step          & $138/196$ & $0.189$ & $20.0$ & $339.9$ & $0.358$ \\
        $\Rgrad$           & criticality   & $163/196$ & $0.184$ & $28.0$ & $222.4$ & $0.254$ \\
        $\Rgrad$, capped   & criticality   & $155/196$ & $0.174$ & $28.0$ & $341.1$ & $0.322$ \\
        \bottomrule
    \end{tabular}
    \caption{The eight configurations on the $196$-problem CUTEst subset, ordered by efficiency.
    Efficiency is the intercept $\pi(1)$ of the iteration profile of
    Figure~\ref{fig: comparison profile} and the solve rate is its right-hand asymptote. The
    median and the mean are iteration counts over the runs each configuration solves, and those
    subsets differ from row to row. Tail activity is the fraction of the last tenth of the
    iterations on which the trust-region constraint bound.}
    \label{tab: comparison}
\end{table}

\begin{figure}[htbp]
    \centering
    \includegraphics[width=0.78\textwidth]{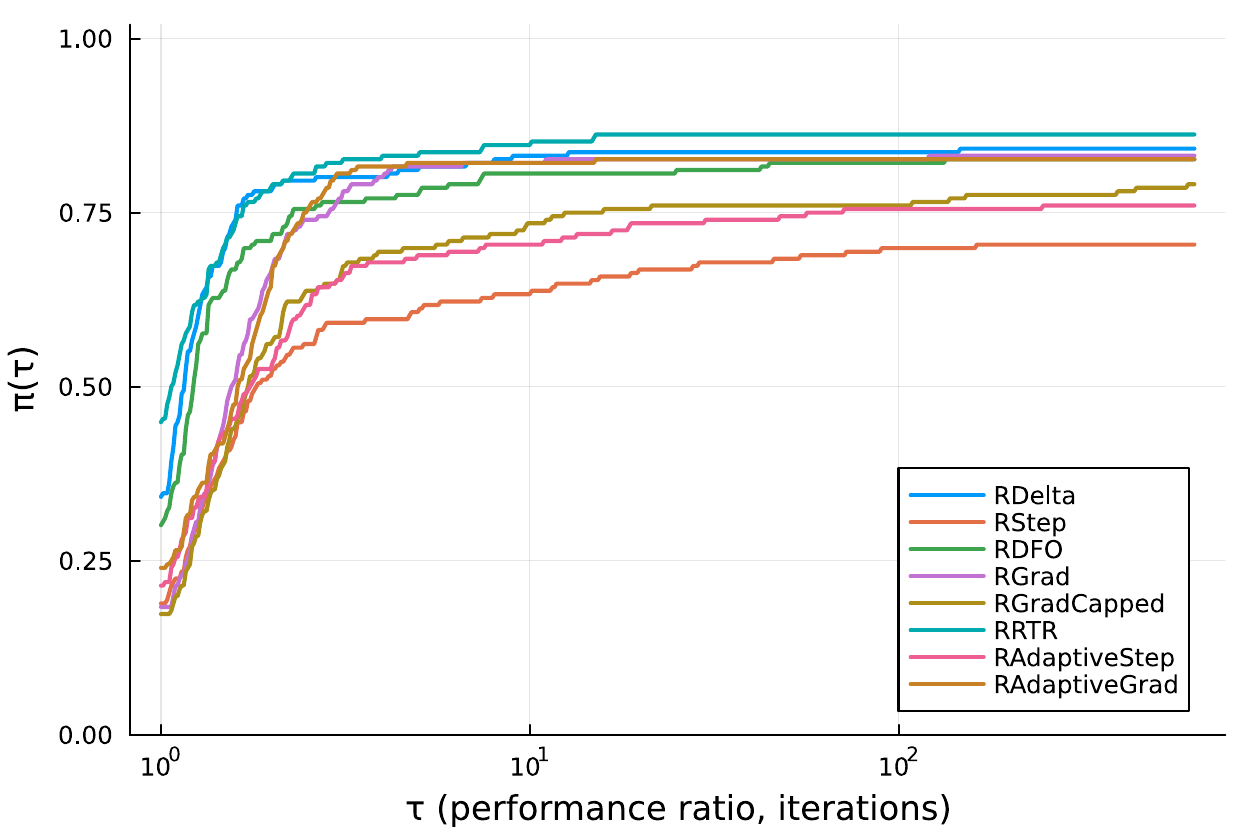}
    \caption{The eight configurations on the $196$-problem subset: the iteration performance
        profile, as the run matrix exported it. The legend carries the configuration names of the
        run record, and their order is the run matrix's rather than the ordering by efficiency of
        Table~\ref{tab: comparison}. Reliability is the right-hand asymptote and efficiency the
        intercept $\pi(1)$; Figure~\ref{fig: comparison activity} plots the second against tail
        activity, where the anchor becomes visible.}
    \label{fig: comparison profile}
\end{figure}

\begin{figure}[htbp]
    \centering
    \begin{tikzpicture}
        \begin{axis}[scatter axis, width=0.72\textwidth,
            xlabel={tail activity}, ylabel={efficiency $\pi(1)$},
            xmin=0.10, xmax=0.40, ymin=0.12, ymax=0.50]
            \addplot[anchor mark={anchorDelta}{*}]        coordinates {(0.129,0.342)};
            \addplot[anchor mark={anchorStep}{square*}]   coordinates {(0.358,0.189) (0.298,0.214)};
            \addplot[anchor mark={anchorCrit}{triangle*}] coordinates
                {(0.151,0.301) (0.254,0.184) (0.322,0.174) (0.203,0.240)};
            \addplot[anchor mark={anchorRetro}{diamond*}] coordinates {(0.138,0.449)};
            \legend{previous radius, step, criticality, retrospective}
            \node[anchor=west,  font=\scriptsize] at (axis cs:0.137,0.342) {$\Rdelta$};
            \node[anchor=west,  font=\scriptsize] at (axis cs:0.146,0.449) {$\Rrtr$};
            \node[anchor=west,  font=\scriptsize] at (axis cs:0.159,0.301) {$\Rdfo$};
            \node[anchor=south, font=\scriptsize] at (axis cs:0.203,0.251) {$\RAdaptiveGrad$};
            \node[anchor=north, font=\scriptsize] at (axis cs:0.254,0.177) {$\Rgrad$};
            \node[anchor=south, font=\scriptsize] at (axis cs:0.298,0.225) {$\RAdaptstep$};
            \node[anchor=north, font=\scriptsize] at (axis cs:0.322,0.166) {$\Rgrad$ capped};
            \node[anchor=south, font=\scriptsize] at (axis cs:0.358,0.200) {$\Rstep$};
        \end{axis}
    \end{tikzpicture}
    \caption{Efficiency against tail activity, one point per configuration, from
        Table~\ref{tab: comparison}. The mark and its colour are what the radius is anchored to.
        The rules that release the constraint are the rules that are fastest, and the anchor
        sorts the plot from left to right.}
    \label{fig: comparison activity}
\end{figure}
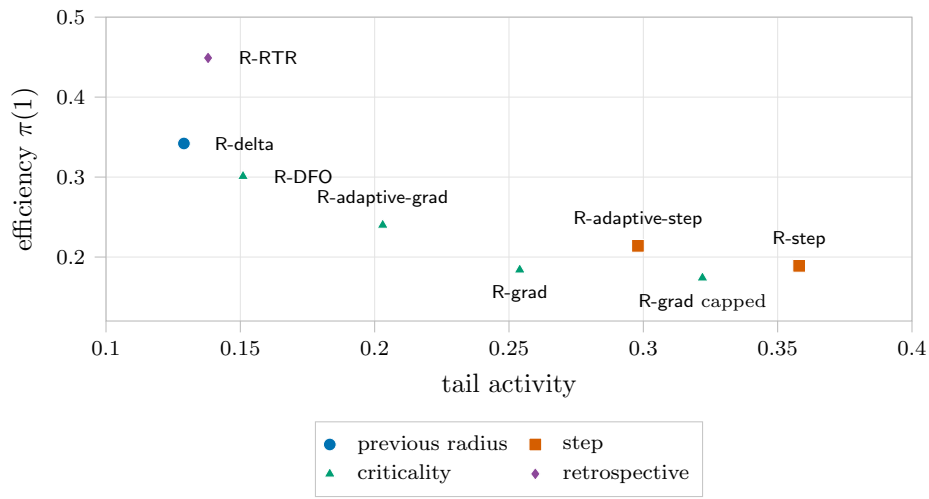

$\Rrtr$ is first on both numbers a profile is read for, with $\pi(1) = 0.449$ against $0.342$ for
its nearest rival and $169$ problems solved of $196$. $\Rdelta$ is second on both. The reading
that costs \prop{this programme} something is the second one. $\Rdelta$ is the textbook fixed-factor rule,
it carries no information about the problem beyond its own history, and only the retrospective
rule improves on it here. Five of the six mechanisms that use more information than $\Rdelta$
does are slower and less reliable than it is.

Efficiency and the solve rate order the configurations the same way except at one place.
$\Rgrad$ is third on the solve rate with $163$ problems and seventh on efficiency with
$\pi(1) = 0.184$, and its median iteration count is $28$ against $18$ to $20$ for the rules
above it. The gradient anchor reaches a first-order point on as many problems as $\Rdelta$ does
and pays about half as many iterations again to get there. That is the anchor axis of
\PartII{} appearing as a cost rather than as a failure, and it is invisible in a solve rate
quoted on its own.

Tail activity sorts the table almost exactly. The four configurations with the least active
tails, between $0.129$ and $0.203$, are the four with the highest efficiency, and the four with
the most active tails, between $0.254$ and $0.358$, are the four lowest. Across the eight the
correlation between tail activity and $\pi(1)$ is $-0.84$, and $-0.86$ on ranks. Eight
configurations are not eight independent draws, so this is a description of
Figure~\ref{fig: comparison activity} and not an estimate of anything. What it supports is the
use of the activity indicator as the instrument of Section~\ref{subsec: two regimes}, since it
is the one quantity in Table~\ref{tab: comparison} that both regimes measure and it tracks the
aggregate cost.

Two pairs differ in a single coordinate and both are readable. $\Rgrad$ against capped
$\Rgrad$, at $\mu_0 = 1$ throughout, isolates the cap: $163$ problems against $155$, tail
activity $0.254$ against $0.322$, and a mean iteration count of $222.4$ against $341.1$. A cap
of $128$ is large and it still costs eight problems, which is the reading of
Section~\ref{subsec: mu sweep} confirmed at a second value of $\mu_0$. $\Rdelta$ against
$\Rstep$ isolates the anchor at a fixed factor, and it costs $27$ problems and trebles the tail
activity. The step anchor rebuilds the radius from a step that the model has just chosen to keep
short, so a run that takes an interior step contracts the region it did not need to contract.

$\Rdfo$ reproduces the $\zeta = 100$ row of Table~\ref{tab: zeta sweep} exactly, in solve rate,
median, mean and tail activity, from a separate archive built at the same commit. Nothing in this
section is a claim of Parts~I and~II settled or refuted. It says that the anchor axis is visible
in what a user pays, that the factor axis is not measured, and that the ordering of the
mechanisms by aggregate cost is not the ordering their theory would suggest.

\paragraph{Two exported figures are not printed.}
The run matrix also exports a profile in objective evaluations and a data profile, and neither
adds anything as it stands. The objective-evaluation profile has the same intercept and the same
right-hand asymptote as Figure~\ref{fig: comparison profile} for all eight configurations, and
differs from it by at most $0.031$, which is six problems, anywhere in between. A trust-region
iteration evaluates the objective at exactly one trial point, so the two measures are the same
measure. A second cost axis has to be gradient evaluations, which are paid only on accepted
steps, or Hessian-vector products, which the subproblem solver decides.

\section{Inactivity and the local rate}
\label{sec: inactivity}

The activity indicator is the one instrument both regimes share, and every statement conditional
on inactivity is read from it. This section measures it directly on CUTEst, then estimates the
local convergence order on the part of each run where the constraint has stopped binding.
Section~\ref{subsec: inactivity} asks how far the mechanisms separate on that instrument when each
is set as a user would set it. Section~\ref{subsec: order rules} settles whether the rate after
the constraint releases depends on the rule.

\subsection{Measuring inactivity}
\label{subsec: inactivity}

\paragraph{What is recorded and how.}
The driver marks an iteration \emph{active} when $\norm{s_k} = \Delta_k$ within the subproblem
solver's own tolerance, and writes that flag to the run record. Two quantities are derived from
the flag sequence. The \emph{tail activity} is the fraction of the last tenth of the iterations
that are active. The \emph{inactivity index} $k^\star$ is the first iteration after which no
later iteration is active, and it is undefined when the final iteration is active.

Two decisions in that procedure carry weight and both go against the convenient choice. A tail
fraction is computed only for runs of at least $30$ iterations. On a ten-iteration run the last
tenth is a single iteration, the fraction can only be $0$ or $1$, and a table of such values
measures rounding rather than asymptotics. And no run is filtered on success. A rule that keeps
the constraint bound is exactly the rule that exhausts the iteration budget, so filtering on
solved runs would discard the evidence the experiment exists to collect. The two are the same
runs, as the outcome counts below show.

\paragraph{The roster.}
Six configurations ran on the $196$ problems under the settings of
Section~\ref{subsec: settings}, at a first-order tolerance of $10^{-5}$. Every one is the
mechanism as a user would set it, at the constants of Table~\ref{tab: settings}, and no
configuration was chosen to fail. That is the design decision this section rests on. The question
it can answer is not whether a rule placed below its threshold stays active, which
Section~\ref{sec: threshold} settles on problems where $\kbar$ is known, but whether the
separation \PartII{} predicts is still visible when nothing is tuned to make it visible.

Two of the six differ from a sibling in one parameter and nothing else. $\Rgrad$ and $\Rgrad$
capped share $\mu_0 = 1$, the same branch factors and the same subproblem solver, and differ only
in that the second imposes $\overline{\mu} = 128$. $\Rdfo$ runs at the $\zeta = 100$ of
Table~\ref{tab: settings}. That pair is the controlled contrast of this section, and $128$ is a
generous cap rather than a punitive one.

\begin{table}[htbp]
    \centering
    \small
    \begin{tabular}{@{}llrrrr@{}}
        \toprule
        Rule & Parameter & solved & measured & mean tail activity & never inactive \\
        \midrule
        $\Rdelta$        & ---                        & $165/196$ & $86$  & $0.399$ & $28/196 = 0.14$ \\
        $\Rgrad$         & $\overline{\mu} = +\infty$ & $163/196$ & $109$ & $0.413$ & $24/196 = 0.12$ \\
        $\RAdaptiveGrad$ & ---                        & $162/196$ & $98$  & $0.423$ & $32/196 = 0.16$ \\
        $\Rdfo$          & $\zeta = 100$              & $162/196$ & $96$  & $0.456$ & $28/196 = 0.14$ \\
        \midrule
        $\Rgrad$, capped & $\overline{\mu} = 128$     & $155/196$ & $114$ & $0.553$ & $51/196 = 0.26$ \\
        $\Rstep$         & ---                        & $138/196$ & $107$ & $0.797$ & $59/196 = 0.30$ \\
        \bottomrule
    \end{tabular}
    \caption{Trust-region activity on the $196$-problem CUTEst subset, over all runs and not only
    the solved ones, with every rule at the constants of Table~\ref{tab: settings}.
    \emph{Measured} counts the runs of at least $30$ iterations, which are the only runs a tail
    fraction is computed for, and the mean is taken over those; the same quantity in
    Table~\ref{tab: comparison} is averaged over the runs that solved, so the two columns are not
    comparable without saying which population they read. \emph{Never inactive} counts runs whose
    constraint bound at the final iteration, over all $196$. The two rows below the rule are the
    two configurations that keep the constraint bound on a quarter of the collection or more.}
    \label{tab: inactivity}
\end{table}

\paragraph{The mechanisms separate by a factor of two on their default settings.}
Mean tail activity runs from $0.399$ for $\Rdelta$ to $0.797$ for $\Rstep$, and the constraint
never releases on $12$ per cent of the collection for $\Rgrad$ against $30$ per cent for
$\Rstep$. Nothing was tuned to produce that spread. Four of the six sit together between $0.399$
and $0.456$, with a quarter or less of their runs bound at the last iteration, and the two that
stand out do so for reasons \PartII{} names in advance.

That is a weaker statement than a dichotomy and it is the statement this campaign supports. Every
configuration here is set as a user would set it, so none of them is below its threshold; what
separates them is structural. A section claiming that a rule below its threshold never releases
would be claiming more than Table~\ref{tab: inactivity} contains, and
Section~\ref{sec: threshold} settles that question on problems where $\kbar$ can be computed
rather than assumed.

\paragraph{The cap is what costs $\Rgrad$ its inactivity.}
$\Rgrad$ and $\Rgrad$ capped differ in one parameter. Imposing $\overline{\mu} = 128$ raises mean
tail activity from $0.413$ to $0.553$ and more than doubles the number of runs whose constraint
never releases, from $24$ to $51$. Section~\ref{subsec: like for like} reads the same contrast in
solved problems and in cost; this is what it does to the constraint. A cap of $128$ is not a small
constant, and \PartII{} predicts the effect: an unbounded multiplier is what lets the radius
outgrow any fixed multiple of $\norm{g_k}$, and capping it reintroduces a constant against which
the solution has to be measured. Being a one-parameter contrast rather than a comparison across
rules, it carries more than the column of means does.

\paragraph{$\Rstep$ has no parameter to blame.}
$\Rstep$ reaches $0.797$ and leaves the constraint bound at the last iteration on $59$ of $196$
problems, with no threshold constant set badly or at all. The step anchor rebuilds the radius from
a step whose length the radius itself bounded, so a run whose steps sit on the boundary is a run
whose radius follows the boundary. That is the anchor axis producing the symptom a badly chosen
$\zeta$ produces in Section~\ref{sec: threshold}, without a parameter to blame, and it is
consistent with the tail activity of $0.358$ that $\Rstep$ records in
Table~\ref{tab: comparison}.

The run outcomes corroborate the ordering on an instrument that does not depend on the
thirty-iteration filter. The wall-clock cap ends between $8$ and $9$ runs per rule, near enough
uniform across the six, so it carries no information about the mechanism. The iteration cap does:
it ends $32$ runs for $\Rstep$ and $6$ for $\Rgrad$ capped, against $5$, $2$, $1$ and $1$ for
$\Rdfo$, $\Rdelta$, $\Rgrad$ and $\RAdaptiveGrad$. The two configurations that exhaust the
iteration budget are the two with the highest tail activity, and they are the same two.

\paragraph{What the filter does to the numbers.}
Excluding runs shorter than $30$ iterations removes exactly the easy problems, on which a well-set
rule releases the constraint after a handful of iterations, so every mean in
Table~\ref{tab: inactivity} is computed on the harder part of that rule's collection. The effect
is large. $\Rdelta$ has $86$ measured runs of $196$ and a mean of $0.399$ here, against $0.129$ in
Table~\ref{tab: comparison}, which averages the same quantity over the runs that solved. The two
numbers are not in conflict and neither is wrong. They answer different questions, and any
comparison across the two tables has to say which population it is reading. The \emph{measured}
column is worth reading on its own account: it runs from $86$ for $\Rdelta$ to $114$ for $\Rgrad$
capped, so the rules that keep the constraint bound are also the rules that take longer to
finish, and the two columns are not independent evidence.

\subsection{The order after inactivity, across mechanisms}
\label{subsec: order rules}

\paragraph{The estimator.}
Claim~\ref{claim: rate} is conditional on inactivity, so the estimate has to be conditional on it
too. Each run is trimmed to its final run of inactive iterations, and the observed order is the
ordinary least-squares slope of $\log\norm{g_{k+1}}$ on $\log\norm{g_k}$ over that tail.
Including active iterations would mix the linear phase with the asymptotic one and return a
number between the two, which is an artefact and not a rate. The tolerance is $10^{-10}$ rather
than the $10^{-5}$ of every other experiment, because the asymptotic regime is what is being
measured and a run stopped at $10^{-5}$ may never enter it.

\paragraph{Four exclusions, and why each is necessary.}
A run contributes an order only if it converged, released the constraint, and left a tail long
enough and well enough conditioned to fit. The four exclusion counts appear in
Table~\ref{tab: order by rule} beside the fitted count, because each is a way the median could be
biased and none of them is uniform across rules.

\begin{itemize}
    \item \emph{never}. The constraint bound at the last iteration, so there is no inactive tail.
          For a configuration below its threshold that is the result and not a gap in it.
    \item \emph{unconv}. The run did not reach a first-order point. This is the exclusion that
          matters. A stalled run has an inactive tail, its steps being tiny and interior, so it
          passes every structural guard and yields a slope near $1$. That is the signature of a
          run going nowhere rather than a convergence order. On the earlier archive $100$ of
          $459$ otherwise fitted runs had not converged, and including them moved the pooled
          median from $1.422$ to $1.308$, with a contamination between $20$ and $31$ per cent
          across rules and therefore correlated with the axis under test.
    \item \emph{short}. Fewer than six inactive iterations. Below that the regression has two or
          three points and returns a number that looks like a rate.
    \item \emph{flat}. The tail is long enough and degenerate, with fewer than four finite
          positive values or no variance in the regressor. The estimator returns a NaN, which a
          median sorts last and which therefore biases the reported value downwards without
          raising an error. This count is zero on the present archive.
\end{itemize}

\begin{table}[htbp]
    \centering
    \small
    \begin{tabular}{@{}lrrrrrrrr@{}}
        \toprule
        Rule & fitted & unconv. & short & flat & never & median order & IQR of order
          & median $k^\star$ \\
        \midrule
        $\Rdfo$          & $66$ & $29$ & $21$ & $0$ & $80$ & $1.387$ & $1.21$--$1.75$ & $9.5$  \\
        $\Rrtr$          & $77$ & $31$ & $28$ & $0$ & $60$ & $1.392$ & $1.21$--$1.64$ & $12$   \\
        $\RAdaptiveGrad$ & $61$ & $40$ & $27$ & $0$ & $68$ & $1.401$ & $1.29$--$1.71$ & $13$   \\
        $\Rdelta$        & $76$ & $28$ & $24$ & $0$ & $68$ & $1.407$ & $1.21$--$1.68$ & $10$   \\
        $\Rgrad$         & $41$ & $38$ & $50$ & $0$ & $67$ & $1.447$ & $1.28$--$1.83$ & $15$   \\
        $\Rgrad$, capped & $38$ & $31$ & $38$ & $0$ & $89$ & $1.491$ & $1.35$--$1.84$ & $12.5$ \\
        $\RAdaptstep$    & $43$ & $29$ & $50$ & $0$ & $74$ & $1.576$ & $1.31$--$1.81$ & $5$    \\
        $\Rstep$         & $29$ & $20$ & $57$ & $0$ & $90$ & $1.583$ & $1.16$--$1.87$ & $3$    \\
        \bottomrule
    \end{tabular}
    \caption{Observed local order on the inactive tail, over $196$ problems with the exact
    Hessian, truncated conjugate gradients and a tolerance of $10^{-10}$, ordered by median
    order. The five count columns sum to $196$ in every row. $k^\star$ is the first permanently
    inactive iteration, over the fitted runs.}
    \label{tab: order by rule}
\end{table}

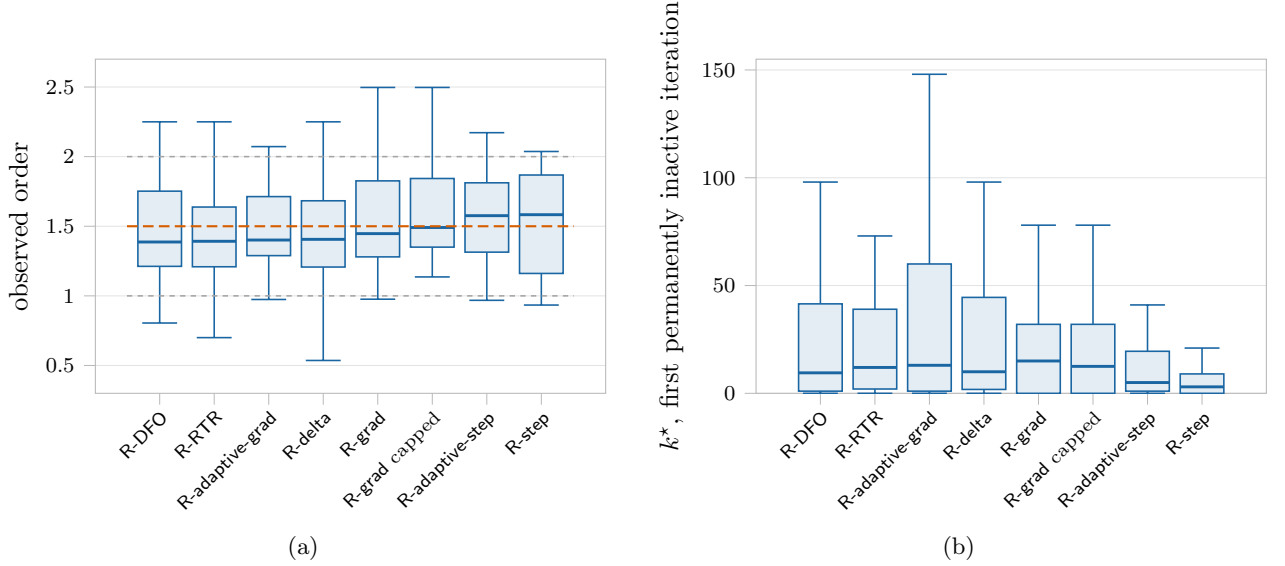
\begin{figure}[htbp]
    \centering
    \begin{subfigure}[b]{0.49\textwidth}
        \centering
        \begin{tikzpicture}
            \begin{axis}[box axis, ylabel={observed order}, ytick={0.5,1,1.5,2,2.5},
                ymin=0.3, ymax=2.7,
                xticklabels={$\Rdfo$, $\Rrtr$, $\RAdaptiveGrad$, $\Rdelta$, $\Rgrad$,
                             $\Rgrad$ capped, $\RAdaptstep$, $\Rstep$}]
                \addplot[boxplot prepared={lower whisker=0.805, lower quartile=1.212,
                    median=1.387, upper quartile=1.752, upper whisker=2.250}] coordinates {};
                \addplot[boxplot prepared={lower whisker=0.700, lower quartile=1.209,
                    median=1.392, upper quartile=1.638, upper whisker=2.250}] coordinates {};
                \addplot[boxplot prepared={lower whisker=0.974, lower quartile=1.289,
                    median=1.401, upper quartile=1.713, upper whisker=2.072}] coordinates {};
                \addplot[boxplot prepared={lower whisker=0.536, lower quartile=1.207,
                    median=1.406, upper quartile=1.683, upper whisker=2.250}] coordinates {};
                \addplot[boxplot prepared={lower whisker=0.976, lower quartile=1.280,
                    median=1.447, upper quartile=1.826, upper whisker=2.497}] coordinates {};
                \addplot[boxplot prepared={lower whisker=1.136, lower quartile=1.350,
                    median=1.491, upper quartile=1.843, upper whisker=2.497}] coordinates {};
                \addplot[boxplot prepared={lower whisker=0.968, lower quartile=1.314,
                    median=1.576, upper quartile=1.812, upper whisker=2.172}] coordinates {};
                \addplot[boxplot prepared={lower whisker=0.934, lower quartile=1.161,
                    median=1.583, upper quartile=1.868, upper whisker=2.037}] coordinates {};
                \draw[pred rule] (axis cs:0.4,1.5) -- (axis cs:8.6,1.5);
                \draw[ref rule]  (axis cs:0.4,1.0) -- (axis cs:8.6,1.0);
                \draw[ref rule]  (axis cs:0.4,2.0) -- (axis cs:8.6,2.0);
            \end{axis}
        \end{tikzpicture}
        \caption{}
        \label{fig: order box}
    \end{subfigure}
    \hfill
    \begin{subfigure}[b]{0.49\textwidth}
        \centering
        \begin{tikzpicture}
            \begin{axis}[box axis, ylabel={$k^\star$, first permanently inactive iteration},
                ymin=0, ymax=155,
                xticklabels={$\Rdfo$, $\Rrtr$, $\RAdaptiveGrad$, $\Rdelta$, $\Rgrad$,
                             $\Rgrad$ capped, $\RAdaptstep$, $\Rstep$}]
                \addplot[boxplot prepared={lower whisker=0, lower quartile=1,
                    median=9.5, upper quartile=41.5, upper whisker=98}] coordinates {};
                \addplot[boxplot prepared={lower whisker=0, lower quartile=2,
                    median=12, upper quartile=39, upper whisker=73}] coordinates {};
                \addplot[boxplot prepared={lower whisker=0, lower quartile=1,
                    median=13, upper quartile=60, upper whisker=148}] coordinates {};
                \addplot[boxplot prepared={lower whisker=0, lower quartile=1.8,
                    median=10, upper quartile=44.5, upper whisker=98}] coordinates {};
                \addplot[boxplot prepared={lower whisker=0, lower quartile=0,
                    median=15, upper quartile=32, upper whisker=78}] coordinates {};
                \addplot[boxplot prepared={lower whisker=0, lower quartile=0,
                    median=12.5, upper quartile=32, upper whisker=78}] coordinates {};
                \addplot[boxplot prepared={lower whisker=0, lower quartile=1,
                    median=5, upper quartile=19.5, upper whisker=41}] coordinates {};
                \addplot[boxplot prepared={lower whisker=0, lower quartile=0,
                    median=3, upper quartile=9, upper whisker=21}] coordinates {};
            \end{axis}
        \end{tikzpicture}
        \caption{}
        \label{fig: kstar box}
    \end{subfigure}
    \caption{The two halves of Claim~\ref{claim: rate}, over the fitted runs of
        Table~\ref{tab: order by rule}, in the order of that table. Left, the observed order.
        The dashed line at $1.5$ is the order an inexact Newton iteration with the forcing
        exponent $\theta = \tfrac12$ of Algorithm~\ref{alg: SteihaugCG} attains, and the two grey
        lines mark linear and quadratic convergence. Right, the first permanently inactive
        iteration. Whiskers are at $1.5$ times the interquartile range and the outliers beyond
        them are not drawn. The right panel is clipped at $k^\star = 155$, which hides between
        one and ten outliers per rule, the largest being $4802$.}
    \label{fig: order}
\end{figure}

\paragraph{The rate does not depend on the rule.}
The eight distributions of Figure~\ref{fig: order box} overlap almost completely. Every
interquartile range contains $1.4$, the lower quartiles run from $1.16$ to $1.35$ and the upper
from $1.64$ to $1.87$, and the medians span $1.387$ to $1.583$. Pooling the $431$ fitted runs
gives a median of $1.422$ with quartiles $1.235$ and $1.757$. On this evidence the first half of
Claim~\ref{claim: rate} stands: conditional on inactivity, the observed order is a property of
the model and the subproblem solver rather than of the radius rule.

The residual spread of $0.196$ between the medians is a selection effect and not a mechanism.
Across the eight rules the median order correlates at $+0.88$ with the number of runs excluded as
\emph{short} and at $-0.83$ with the number fitted. The two highest medians, $\Rstep$ at $1.583$
and $\RAdaptstep$ at $1.576$, come from the two rules with the fewest fitted runs, $29$ and $43$
of $196$, and with the smallest median $k^\star$. A rule that releases the constraint at
iteration $3$ yields an inactive tail only on the problems it solves quickly, so its median order
is taken over the easiest problems in the collection. Reading the ordering of the median column
as a difference between mechanisms would be reading the exclusion counts.

\paragraph{What the rule controls is weaker than it looks.}
The second half of Claim~\ref{claim: rate} asserts that the mechanisms differ in when inactivity
is reached. The medians do differ, from $k^\star = 3$ for $\Rstep$ to $15$ for $\Rgrad$, a factor
of five. Figure~\ref{fig: kstar box} puts that in proportion. Within a single rule $k^\star$ runs
from $0$ to the upper quartile at $9$ to $60$ and on to outliers in the thousands, so the spread
across problems inside one rule is an order of magnitude larger than the spread between the
rules. On this collection the mechanism is a real but weak predictor of when the constraint
releases, and the strong version of the claim, that the mechanism decides $k^\star$, is not what
Table~\ref{tab: order by rule} shows.

Two smaller readings remain. Thirty-nine of the $431$ fitted values fall below $1$ and
one is negative, which no convergence order can be, so the four guards do not catch every fit
that is not measuring a rate. Twenty-two exceed $2$, which the exact Hessian cannot deliver
either. Both counts are small and neither changes the verdict;
Section~\ref{sec: threats} places them among the limits of the design.

\paragraph{Verdict on Claim~\ref{claim: rate}. First half stands, second half holds only
weakly.}
The claim has two halves: conditional on inactivity the observed rate is that of the underlying
(quasi-)Newton iteration and independent of the mechanism, and the mechanisms differ in when
inactivity is reached. The first is what Figure~\ref{fig: order} shows. The eight distributions
overlap almost completely, every interquartile range contains $1.4$, and the pooled median of the
$431$ fitted runs is $1.422$, so the rate belongs to the model and the subproblem solver and not
to the rule. The second is true and weaker than the claim states. The mechanisms do differ, from
$k^\star = 3$ for $\Rstep$ to $15$ for $\Rgrad$, but within a single rule $k^\star$ spans an
order of magnitude more than that, so the mechanism is a real but weak predictor of when the
constraint releases, and the strong form of the claim is not what
Table~\ref{tab: order by rule} shows.

\section{Model vs radius rule} 
\label{sec: model vs radius rule}

\subsection{The  interaction between model and radius rule}
\label{subsec: interaction}

Four rules crossed with four model Hessians give sixteen configurations, each run on the same
$196$ problems, with truncated conjugate gradients throughout and the settings of
Section~\ref{subsec: settings}. The models are the exact Hessian, limited-memory BFGS and SR1
at memory $5$, and the scaled identity $B_k = I$. The runs are deterministic, so the grid has no
replication and the numbers below carry no sampling error of their own. $\Rdfo$ runs at the
$\zeta = 100$ of Table~\ref{tab: settings}, as everywhere else in this \thisdoc{}, so its
exact-Hessian cell reproduces the $162$ of Table~\ref{tab: comparison} and the grid varies the
model alone.

The identity column is the control that makes the grid worth running. With $B_k = I$ the model
minimiser is $-g_k$, so the step is $-\min\{\norm{g_k}, \Delta_k\}\,g_k/\norm{g_k}$ and the method
is steepest descent with a step length the radius rule alone decides. Whatever separates the
rules in that column is the rule and nothing else.

\paragraph{How the grid is read.}
Write $p_{ij}$ for the fraction of the $196$ problems that rule $i$ solves under model $j$. The
grid is decomposed into a grand mean, a row effect, a column effect and a residual,
\begin{equation}
    p_{ij} = \overline{p} + (\overline{p}_{i\cdot} - \overline{p})
           + (\overline{p}_{\cdot j} - \overline{p}) + r_{ij} ,
    \qquad
    r_{ij} = p_{ij} - \overline{p}_{i\cdot} - \overline{p}_{\cdot j} + \overline{p} ,
    \label{eqn: additive decomposition}
\end{equation}
and $r_{ij}$ is what the two factors do together that neither predicts alone. A grid on which the
rule and the model act independently has $r_{ij} = 0$ in every cell.

\begin{figure}[htbp]
    \centering
    \includegraphics[width=0.72\textwidth]{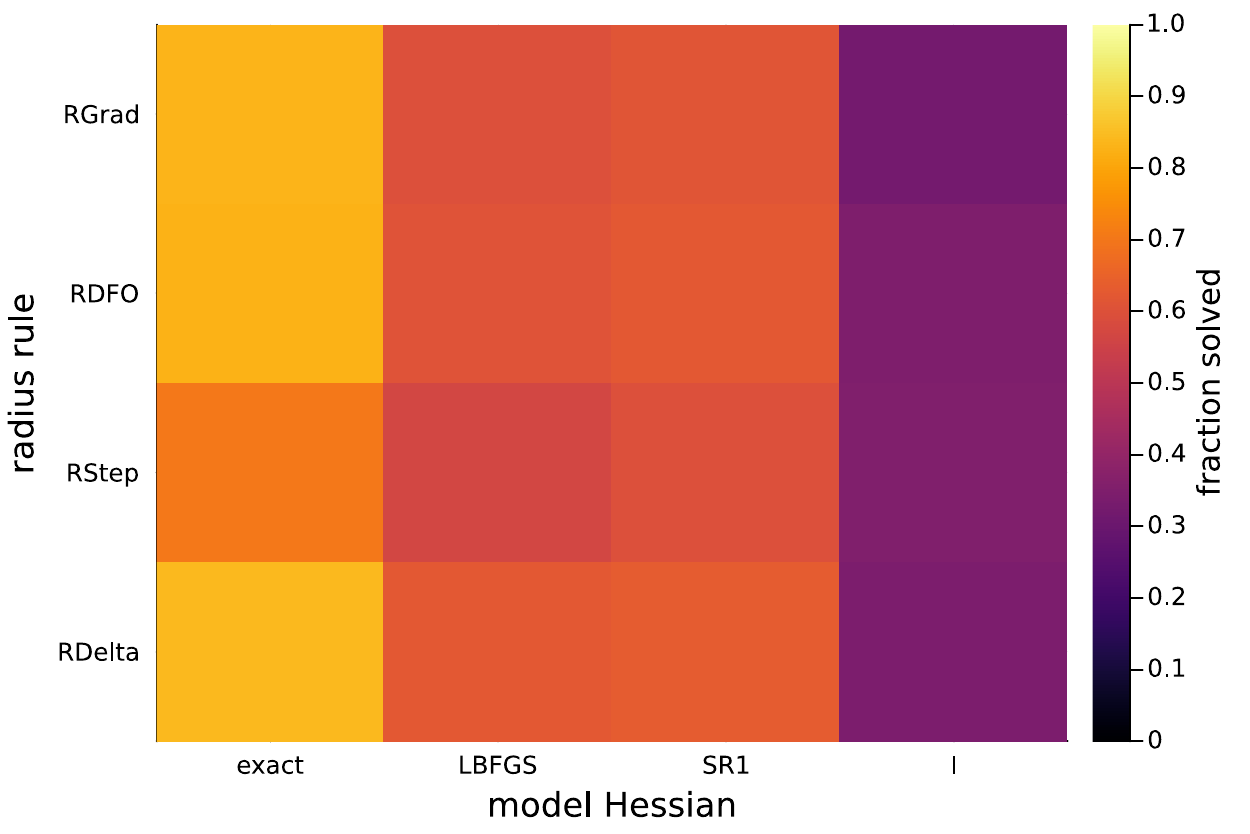}
    \caption{The fraction of the $196$ problems each pair solves, as the run matrix plots it. The
        colour scale runs from $0$ to $1$ and the four columns use most of it, so the figure
        carries the model effect and only that. The differences between the rules within a
        column, which are between $5$ and $26$ problems, fall below the resolution of the scale
        and the departure from additivity is smaller still. Both are read from
        Tables~\ref{tab: interaction} and~\ref{tab: additivity} instead.}
    \label{fig: interaction}
\end{figure}

\begin{table}[htbp]
    \centering
    \small
    \begin{tabular}{@{}lrrrrr@{}}
        \toprule
        & \multicolumn{4}{c}{model Hessian} & \\
        \cmidrule(lr){2-5}
        Rule & exact & L-BFGS & SR1 & $I$ & row effect \\
        \midrule
        $\Rdelta$ & $165\ (18)$ & $122\ (46)$ & $124\ (50)$ & $67\ (142)$ & $+0.021$ \\
        $\Rdfo$   & $162\ (20)$ & $119\ (52)$ & $122\ (50)$ & $68\ (146)$ & $+0.012$ \\
        $\Rgrad$  & $163\ (28)$ & $117\ (44)$ & $120\ (54)$ & $63\ (140)$ & $+0.002$ \\
        $\Rstep$  & $138\ (20)$ & $111\ (43)$ & $117\ (52)$ & $69\ (181)$ & $-0.034$ \\
        \midrule
        column effect & $+0.212$ & $+0.009$ & $+0.027$ & $-0.248$ & \\
        \bottomrule
    \end{tabular}
    \caption{Problems solved of $196$, with the median iteration count over the solved runs in
    brackets. The row and column effects are the deviations
    $\overline{p}_{i\cdot} - \overline{p}$ and $\overline{p}_{\cdot j} - \overline{p}$
    of~\eqref{eqn: additive decomposition}, on the fraction-solved scale, with
    $\overline{p} = 0.589$. The median columns are taken over the runs each cell solves and those
    sets differ from cell to cell.}
    \label{tab: interaction}
\end{table}

\begin{table}[htbp]
    \centering
    \small
    \begin{tabular}{@{}lrrrrr@{}}
        \toprule
        Rule & exact & L-BFGS & SR1 & $I$ & spread within row \\
        \midrule
        $\Rdelta$ & $+0.020$ & $+0.004$ & $-0.004$ & $-0.019$ & $0.040$ \\
        $\Rdfo$   & $+0.014$ & $-0.003$ & $-0.005$ & $-0.005$ & $0.019$ \\
        $\Rgrad$  & $+0.029$ & $-0.003$ & $-0.005$ & $-0.021$ & $0.050$ \\
        $\Rstep$  & $-0.063$ & $+0.002$ & $+0.015$ & $+0.046$ & $0.108$ \\
        \midrule
        spread within column & $0.092$ & $0.006$ & $0.020$ & $0.066$ & \\
        \bottomrule
    \end{tabular}
    \caption{The residuals $r_{ij}$ of~\eqref{eqn: additive decomposition}, in fractions of $196$
    problems, with the range of each row and column beside them. A residual of $0.05$ is ten
    problems. The two middle columns are additive to within $0.015$ and the two outer ones are
    not, and the column ranges say the same thing in one line.}
    \label{tab: additivity}
\end{table}

\paragraph{The model decides, by a factor of eight.}
The column effects run from $+0.212$ for the exact Hessian to $-0.248$ for the identity, a range
of $90$ problems, and that range is what Figure~\ref{fig: interaction} shows as four bands of
colour. The row effects run from $+0.021$ to $-0.034$, a range of $11$ problems, which the same
figure cannot show. Averaged over the four rules, moving from the exact Hessian to a quasi-Newton
model costs $40$ problems against L-BFGS and $36$ against SR1, and moving to the identity costs
about ninety. Partitioning the sum of squares of the grid puts $96.5$ per cent on the model,
$1.6$ per cent on the rule and $1.9$ per cent on the residual. On this collection the model
Hessian decides whether the method converges and the radius rule adjusts the cost, which is the
division of labour Sections~\ref{sec: threshold} and~\ref{sec: inactivity} describe.

\paragraph{The grid is not additive, and the departure is the size of the rule effect.}
The largest residual is $-0.063$ for $\Rstep$ under the exact Hessian, which is twelve problems.
That is small beside the model effect and it is not small in the only comparison that matters
here, since the entire rule main effect spans $11$ problems. The residual sum of squares exceeds
the row sum of squares, $1.9$ per cent against $1.6$. A previous draft of this section read the
same residual as evidence that the two axes act independently. That inference does not follow. What
the grid supports is that the rule effect and the interaction are both small beside the model
effect, and that they are the same size as each other.

The departure has a shape rather than being scattered, and the shape is one rule. Reading the last
row of Table~\ref{tab: additivity}, the residuals span $0.092$ across the exact-Hessian column and
$0.066$ across the identity column, against $0.006$ and $0.020$ across the two quasi-Newton
columns. Reading the last column instead, $\Rstep$ spans $0.108$ where the other three span
$0.019$ to $0.050$. The signs say the same thing: $\Rstep$ is the one rule that is negative under
the exact Hessian and positive under the identity, and $\Rdelta$, $\Rdfo$ and $\Rgrad$ are the
reverse. A singular value decomposition of the table puts $99.2$ per cent of the residual energy
in that single contrast. The interaction is therefore one effect and not four, and it lies along
the anchor axis of \PartII{}: the one rule that separates is the one whose radius is rebuilt from
a step the radius itself bounded.

\paragraph{The rule matters where the radius is the binding constraint.}
Row by row in Table~\ref{tab: interaction}, the four rules separate by $27$ problems under the
exact Hessian, by $11$ under L-BFGS, by $7$ under SR1 and by $6$ under the identity. The better
the model, the more the radius rule matters.

That ordering is the opposite of the intuitive one. With $B_k = I$ the radius is the only
quantity the method has left to choose, so one might expect the rules to separate most there.
They separate least. The reason is that with no curvature in the model the search direction is
wrong on most of the collection, and a run failing for want of a direction cannot be rescued by a
better radius. All four rules sit near the floor, at $63$ to $69$ problems, and there is little
room left to differ. Under the exact Hessian the direction is right and what remains in question
is whether the radius permits the Newton step, so the rule decides. This is the sentence of the
abstract measured rather than asserted: the mechanisms separate where the radius is the binding
constraint, and the model decides when that is.

\paragraph{Three limits of this grid.}
The residuals have no error bar. The runs are deterministic, so each cell is exact, but the grid
gives no way to say whether a residual of ten problems is large for a sample of $196$ paired
problems. Settling that needs the per-problem outcomes and a paired test over them, and the
archive exports only the cell totals. The median iteration counts are taken over the runs each
cell solves, so the brackets in Table~\ref{tab: interaction} compare different sets of problems
and the apparent cost of the identity model is measured on the sixty-odd problems it manages.
And four rules is not the roster. $\Rrtr$, which Section~\ref{subsec: like for like} places
first, is absent, as are the two adaptive rules, so nothing here bears on the retrospective ratio
or on the factor axis.

\subsection{How a small constant silences the model}
\label{subsec: silencing}

Claim~\ref{claim: cauchy} predicts that a cap small enough to truncate the first inner iteration
returns the Cauchy point and leaves the model Hessian without influence on the direction. The
configuration of Section~\ref{subsec: spd} turns that prediction into something an
experiment can see, because there the model demonstrably steers the run. Silencing it has a
visible consequence, and the run ends somewhere else.

Table~\ref{tab: spd cap sweep} sweeps the cap over four decades on that configuration.
Figure~\ref{fig: spd cap} draws four of the runs.

\begin{table}[htbp]
    \centering
    \caption{Capped $\Rgrad$ over \texttt{SPDTarget} pinned to the saddle, with truncated
        conjugate gradients, from $x_0 = (-0.5, 0.6)$ and a budget of $2000$ iterations.
        \emph{trunc} is the fraction of iterations on which the inner solve truncated on its
        first direction, which is $-g_k$, and it is the only one of the three step classes that
        is gradient descent in disguise. $\cos_{\min}$ is the smallest $\cos(s_k, -g_k)$ over the
        run, and \emph{tail act.} the fraction of the last tenth of the iterations on which the
        constraint bound. In every row the fraction of iterations within $10^{-12}$ of
        $\cos = 1$ equals \emph{trunc}, no iterate exits strictly inside the region, and the
        fraction exiting on the Krylov residual test is the complement of \emph{trunc}.}
    \label{tab: spd cap sweep}
    \small
    \begin{tabular}{rlrrrrlrr}
        \toprule
        $\overline{\mu}$ & status & iters & trunc & $\cos_{\min}$ & tail act.
          & limit & dist & $\norm{g}$ \\
        \midrule
        $0.01$   & \texttt{max\_iter}    & $2000$ & $1.0000$ & $1.000000$ & $1.0000$
          & origin & $1.42\times10^{-2}$ & $6.17\times10^{-4}$ \\
        $0.05$   & \texttt{max\_iter}    & $2000$ & $1.0000$ & $1.000000$ & $1.0000$
          & origin & $3.19\times10^{-3}$ & $3.07\times10^{-5}$ \\
        $0.10$   & \texttt{max\_iter}    & $2000$ & $1.0000$ & $1.000000$ & $1.0000$
          & origin & $1.62\times10^{-3}$ & $7.93\times10^{-6}$ \\
        $0.30$   & \texttt{max\_iter}    & $2000$ & $1.0000$ & $1.000000$ & $1.0000$
          & origin & $5.05\times10^{-4}$ & $7.65\times10^{-7}$ \\
        \midrule
        $0.50$   & \texttt{exception}    & $7$    & $0.8571$ & $0.823669$ & $1.0000$
          & saddle & $8.74\times10^{-2}$ & $5.44\times10^{-2}$ \\
        $1.00$   & \texttt{exception}    & $5$    & $0.8000$ & $0.950326$ & $1.0000$
          & saddle & $3.90\times10^{-2}$ & $2.74\times10^{-2}$ \\
        \midrule
        $2.00$   & \texttt{first\_order} & $5$    & $0.8000$ & $0.909498$ & $0.0000$
          & saddle & $6.94\times10^{-17}$ & $1.73\times10^{-17}$ \\
        $4.00$   & \texttt{first\_order} & $5$    & $0.8000$ & $0.909498$ & $0.0000$
          & saddle & $6.94\times10^{-17}$ & $1.73\times10^{-17}$ \\
        $8.00$   & \texttt{first\_order} & $5$    & $0.8000$ & $0.909498$ & $0.0000$
          & saddle & $6.94\times10^{-17}$ & $1.73\times10^{-17}$ \\
        $128.00$ & \texttt{first\_order} & $5$    & $0.8000$ & $0.909498$ & $0.0000$
          & saddle & $6.94\times10^{-17}$ & $1.73\times10^{-17}$ \\
        \bottomrule
    \end{tabular}
\end{table}

The table splits into three regimes and the middle one is the smallest. From $\overline{\mu} = 2$
upward the run reaches the saddle in five iterations with the constraint released, which is the
behaviour Section~\ref{subsec: spd} reports. Raising the cap by a further factor of $64$
changes nothing. At $\overline{\mu} = 0.5$ and $1$ the run still heads for the saddle, with the
constraint binding throughout. It stops on the \texttt{DomainError} of~\eqref{eqn: phi target},
the iterate having wandered to a point from which the target is no longer downhill.

Below $\overline{\mu} = 0.5$ the run goes to the origin instead. It exhausts $2000$ iterations,
the constraint binds at every one of them, and \emph{trunc} and $\cos_{\min}$ are both exactly
$1$. Every step is $-g_k$ to machine precision. The model that pins its minimiser to the saddle is
still there, still positive definite, and still ignored. The limit is the degenerate critical
point that lies in the gradient direction from the start. Shrinking one user parameter by a factor
of five moved the limit point of the method.

\begin{figure}[htbp]
    \centering
    \includegraphics[width=\textwidth]{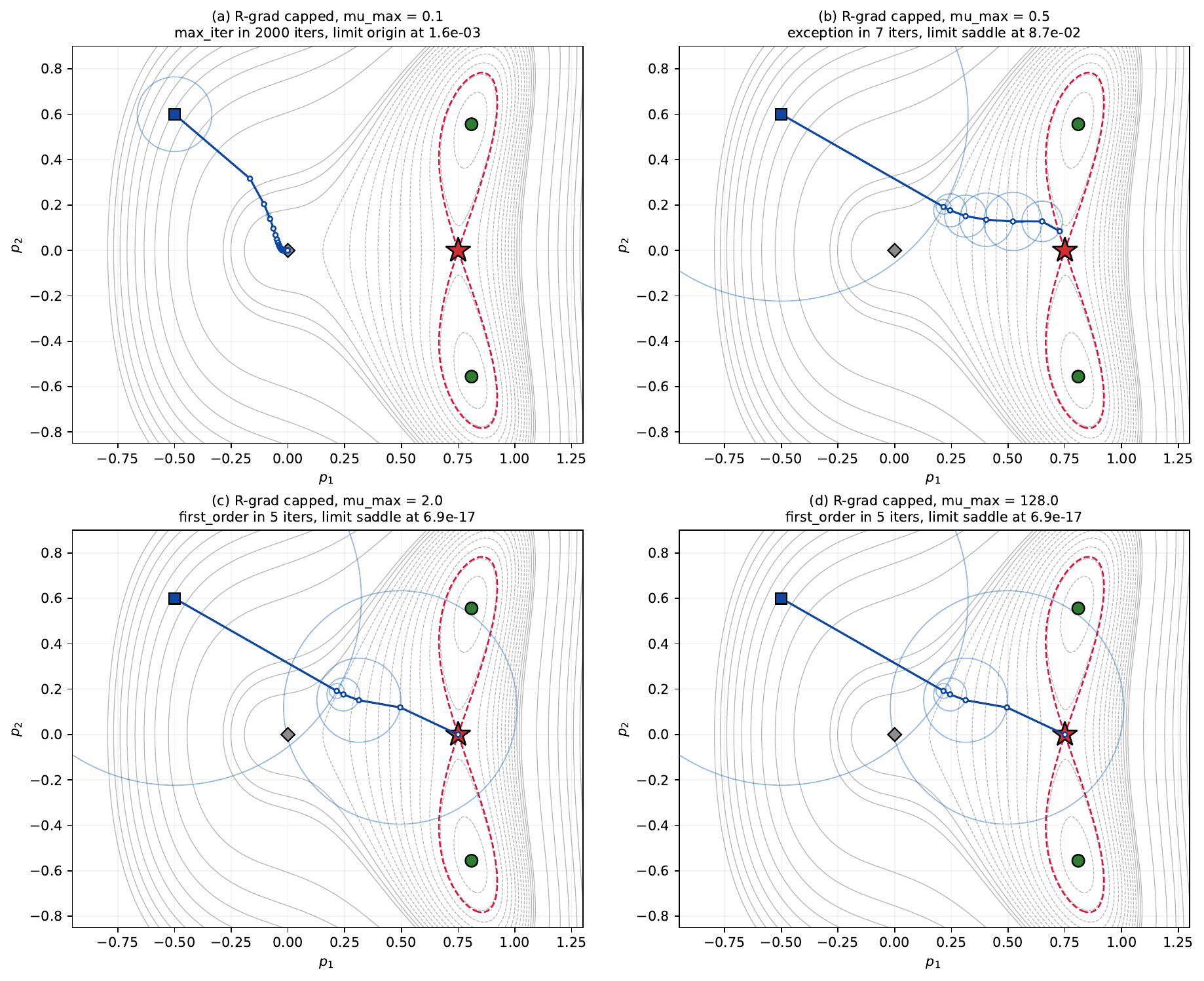}
    \caption{Capped $\Rgrad$ over \texttt{SPDTarget} pinned to the saddle, at
        $\overline{\mu} = 0.1$, $0.5$, $2$ and $128$ in panels (a) to (d). The trust region of
        every accepted step is drawn. The square is $x_0$, the diamond the degenerate origin,
        the star the saddle at $(0.75, 0)$ and the two circles the minimisers. In~(a) the regions
        are too small to leave the gradient direction and the run walks to the origin. From~(c)
        on they are large enough for the model to decide, and the run reaches the saddle in five
        iterations.}
    \label{fig: spd cap}
\end{figure}

\paragraph{The same happens to the other anchored rule.}
Figure~\ref{fig: spd dfo} repeats the sweep on $\Rdfo$, whose radius is held below
$\zeta\omega_k$ rather than tied to a multiple of $\norm{g_k}$. At $\zeta = 0.01$ the run walks to
the origin and exhausts the budget. At $\zeta = 1$, $10$ and $100$ it reaches the saddle in four
iterations. The mechanism is the same and only the parameter that produces a small region differs,
which is what makes this a property of the anchor rather than of one rule.

\begin{figure}[htbp]
    \centering
    \includegraphics[width=\textwidth]{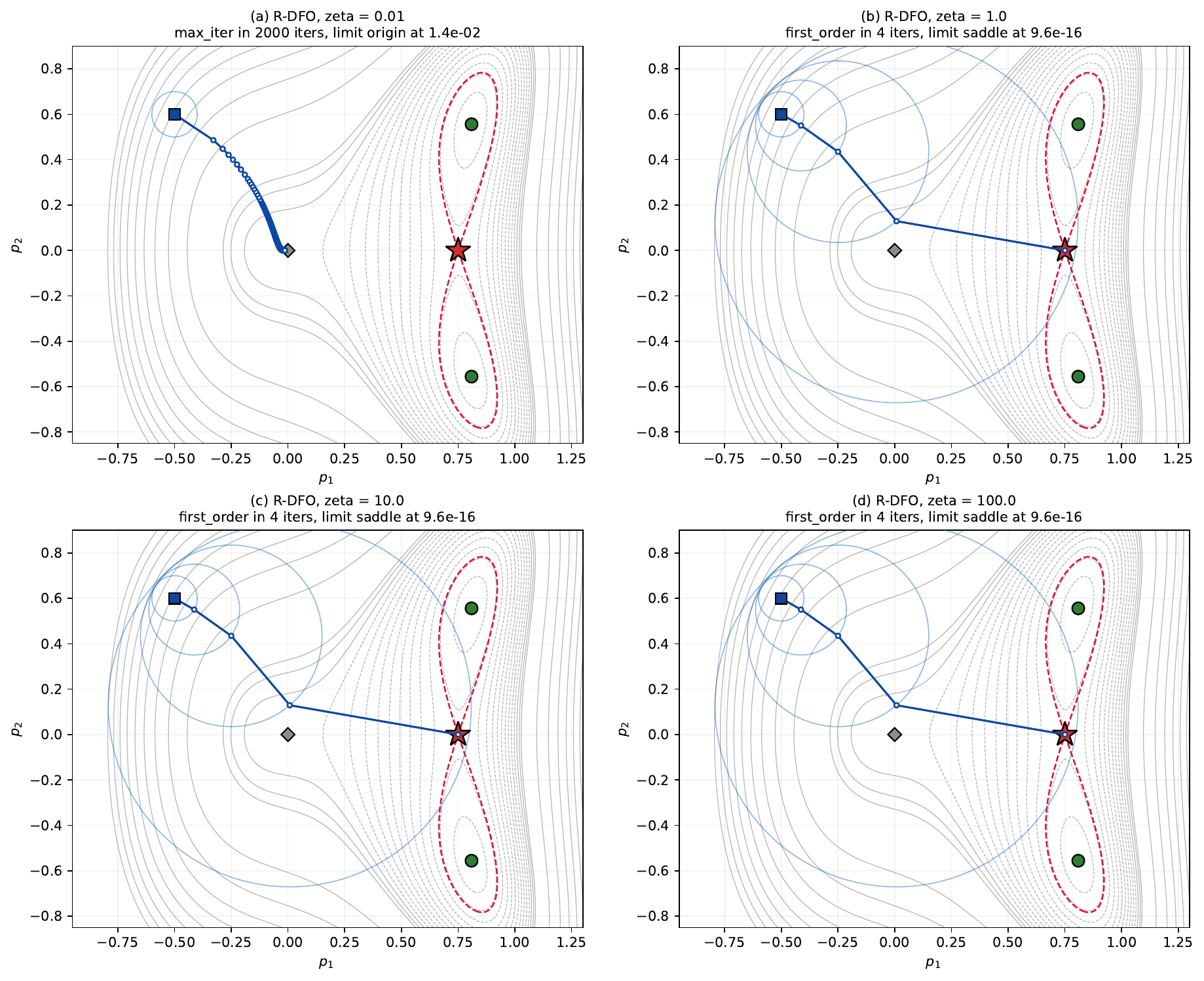}
    \caption{$\Rdfo$ over the same model and from the same starting point, at $\zeta = 0.01$,
        $1$, $10$ and $100$ in panels (a) to (d). Panel (a) exhausts $2000$ iterations at
        $1.4\times10^{-2}$ from the origin. The other three reach the saddle in four iterations.}
    \label{fig: spd dfo}
\end{figure}

\paragraph{Why the step is exactly the gradient direction.}
A small region makes the linear term of the model dominate the quadratic one, so the step that
comes back points along $-g_k$ and the model Hessian has stopped influencing the direction. With
truncated conjugate gradients that degeneration is exact rather than approximate. The first inner
direction is $-g_k$, the region truncates it, and the returned step is a positive multiple of
$-g_k$ with no residue. Over the four rows with $\overline{\mu} \leq 0.30$ the quantity
$1 - \cos(s_k, -g_k)$ has median zero and lies within $4.44\times10^{-16}$ of zero at every
iteration.

The subsolver is not what causes this. Exact solution of the secular equation returns
$s_k = -(H_k + \sigma I)^{-1}g_k$ with $\sigma \to \infty$ as $\Delta_k \to 0$, so the step
approaches the gradient direction without ever reaching it. Run the same sweep under \texttt{ExactMS} and take the median
ratio of the observed $1 - \cos(s_k, -g_k)$ to the value the small-radius expansion predicts. It
is $0.995$, $0.960$ and $0.921$ at the three smallest caps, and falls to $0.050$ at
$\overline{\mu} = 8$. The law holds where the region is small and stops
holding where it is not. A small cap therefore degenerates any subsolver, and truncated conjugate
gradients makes the degeneration exact instead of asymptotic.

\paragraph{The exact Hessian hides the whole effect.}
Running the same problem from the same starting point under \texttt{ExactHessian} shows none of
this. Figure~\ref{fig: exact origin} puts capped $\Rgrad$ at $\overline{\mu} = 0.05$ and $128$
beside $\Rdfo$ at $\zeta = 0.1$ and $20$, and all four reach the origin. Sweeping $\zeta$ over
more than two decades leaves the iteration count at $16$, the status at \texttt{first\_order} and
the limit at the origin on every setting. Only the final radius moves, and it takes three distinct
values over six settings. On this problem, under the exact Hessian, the constant does not decide
the outcome.

\begin{figure}[htbp]
    \centering
    \includegraphics[width=\textwidth]{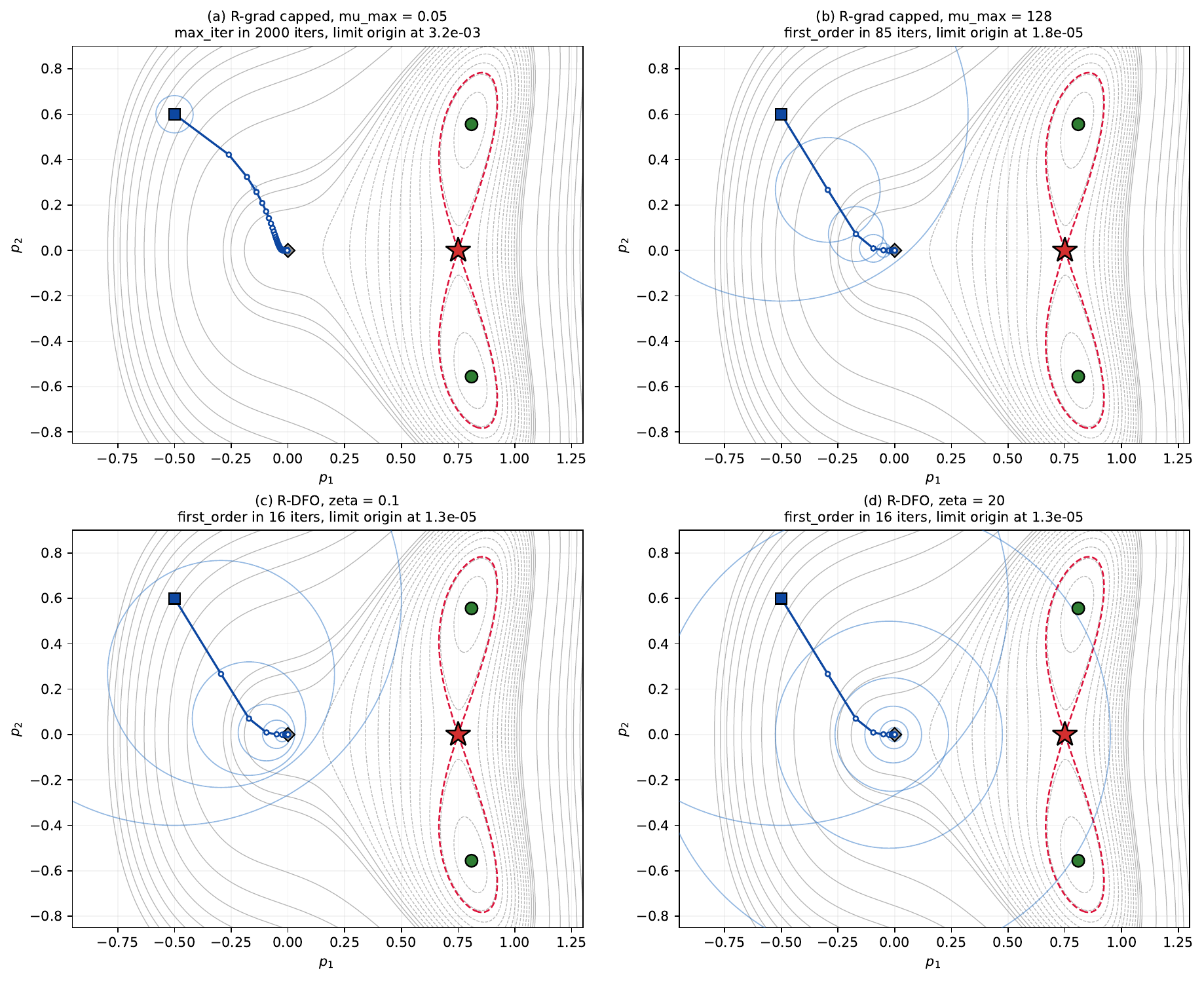}
    \caption{The same problem and starting point under \texttt{ExactHessian}. Capped $\Rgrad$ at
        $\overline{\mu} = 0.05$ and $128$ in panels (a) and (b), $\Rdfo$ at $\zeta = 0.1$ and
        $20$ in (c) and (d). Every panel ends at the degenerate origin, and the constant changes
        the iteration count from $2000$ to $85$ and $16$ without changing the limit.}
    \label{fig: exact origin}
\end{figure}

The two figures together are the finding, and neither carries it alone. Under the exact Hessian
the degenerate origin absorbs every rule at every constant, so nothing separates a run whose model
was silenced from a run whose model pointed at the origin anyway. Building a model that
demonstrably steers elsewhere is what makes the silencing visible. A diagnostic that reports the
limit point alone would call these two experiments the same, since both end at the origin, and
\emph{trunc} and $\cos_{\min}$ are what tell them apart.

\paragraph{Verdict on Claim~\ref{claim: cauchy}. Settled on the designed problems, not in the
aggregate.}
The claim is that with truncated conjugate gradients and a cap small enough that the first inner
iterate leaves the region, the step is the Cauchy point and the method inherits gradient-descent
behaviour. Section~\ref{subsec: silencing} exhibits that outright. Over the four rows of
Table~\ref{tab: spd cap sweep} with $\overline{\mu} \leq 0.30$ the quantity $1 - \cos(s_k, -g_k)$
has median zero and stays within $4.44 \times 10^{-16}$ of zero at every iteration, the model
that pins its minimiser to the saddle is still positive definite and still ignored, and the run
goes to the origin instead. In the aggregate regime the diagnostic recorded for this campaign
reports $\texttt{active} = \textnormal{false}$ at every one of the six caps and takes the same
values at $\overline{\mu} = 10^{-3}$ as at $10^{2}$, so it separates nothing and no reading rests
on it. The failure there is in the recorded quantity, not in the claim.

\paragraph{Verdict on Claim~\ref{claim: bad limits}. Settled.}
The claim is that a configuration can satisfy every hypothesis of \PartI{}, with every $\rho_k$
successful and a well-behaved radius, and still converge to a point that is not a minimiser. The
\texttt{SPDTarget} construction of Section~\ref{subsec: spd} removes the one objection an
invariant-line argument leaves open, by removing the negative curvature from the model
altogether. From $x_0 = (-0.5,\, 0.6)$, which is off the line, every rule and every subproblem
solver converges to the saddle at $(\tfrac34, 0)$ with every $\rho_k$ successful,
$\norm{g_k} \to 0$ and a well-behaved radius. Every hypothesis of \PartI{} holds and the limit is
not a minimiser.

\section{The threshold that predicts inactivity}    
\label{sec: threshold}

Claim~\ref{claim: zeta} asserts a problem-dependent threshold in $\zeta$ and
in $\overline{\mu}$. P2 is the one problem in this study where such a threshold can be located
rather than inferred. Section~\ref{subsec: P2} supplies the curvature at every minimiser in closed
form, and with it the exact activity criterion~\eqref{eqn: sinc activity}.

\paragraph{What the theory gives, and the gap it leaves.}
\PartII{} traps the two criticality-anchored rules below the threshold for different reasons.
Capped $\Rgrad$ sets $\theta_k = \mu_k \leq \overline{\mu}$ at every iteration and $a_k \to
\lambda^*_{\min}$, so~\eqref{eqn: sinc activity} makes the constraint active at every large $k$
whenever
\begin{equation}
    \overline{\mu} \;<\; \frac{1}{\lambda^*_{\min}} .
    \label{eqn: rgrad threshold}
\end{equation}
$\Rdfo$ carries no such bound on $\theta_k$, and the argument instead follows the dynamics of
$\theta_k$ under the update. On an active iteration the map is
$\theta_{k+1} = \gamma\theta_k/(1 - \theta_k a_k)$ with $\gamma \in \{\gtwo, \gthree\}$ chosen by
the criticality test, and the interval $[0, (1-\gtwo)/\overline{a}]$ is forward invariant for it
whenever
\begin{equation}
    \zeta \;\leq\; \frac{1 - \gtwo}{\overline{a}\,(\gthree + 1 - \gtwo)}
    \qquad \Bigl(= \tfrac{1}{5\overline{a}} \ \text{ at } \gtwo = \tfrac12,\ \gthree = 2\Bigr),
    \label{eqn: sinc zeta threshold}
\end{equation}
where $\overline{a} = \max_{x \in N} f_2''(x)$ over a neighbourhood $N$ of the minimiser on which
$f_2'' > 0$. The trap then closes provided $\theta_{k_1} \leq (1-\gtwo)/\overline{a}$ at some
$k_1$ with $x_{k_1} \in N$. Since $\overline{a} \geq \lambda^*_{\min}$, the provable $\Rdfo$
threshold is at most $1/(5\lambda^*_{\min})$ and reaches that value only as $N$ shrinks to the
minimiser.

Neither statement settles the interval between them. Below $1/\lambda^*_{\min}$ the constraint is
active by~\eqref{eqn: rgrad threshold}. Above $\kbar = 8/\lambda^*_{\min}$ inactivity arrives from
the other direction, through the step bound of \PartII{} that defines $\kbar$. On
$\theta \in [1/\lambda^*_{\min},\, 8/\lambda^*_{\min})$ neither theorem speaks, while the exact
criterion~\eqref{eqn: sinc activity} decides the question outright because the problem is
one-dimensional. What follows locates the constant inside that interval.

\paragraph{The roots, to more digits than the experiment needs.}
Section~\ref{subsec: P2} states the expansion of the roots to one term. The experiment starts its
root solve from the truncated asymptotic form, so it needs the next one.

\begin{lemma}[Two terms of the gap to the asymptote]
\label{lem: sinc expansion}
    Write $A_j \defeq (j+\tfrac12)\pi$ and $\delta_j \defeq A_j - y_j$ for the gap between the
    $j$-th root of $\tan y = y$ and the asymptote above it. Then $\delta_j \downarrow 0$ and
    \begin{equation}
        \delta_j = \frac{1}{A_j} + \frac{2}{3A_j^{3}} + \calO\bigl(A_j^{-5}\bigr) .
        \label{eqn: sinc expansion}
    \end{equation}
\end{lemma}

\begin{proof}
    On $\bigl((j-\tfrac12)\pi,\, j\pi\bigr)$ the tangent is non-positive and the line is
    positive, so no root lies there. On $\bigl(j\pi,\, A_j\bigr)$ the map $y \mapsto \tan y - y$
    has derivative $\tan^2 y \geq 0$ and increases from $-j\pi$ to $+\infty$, so it vanishes
    once. Writing $y_j = A_j - \delta_j$ with $\delta_j \in (0, \tfrac\pi2)$, and using
    $\tan(\tfrac\pi2 - \delta) = \cot\delta$ with the $\pi$-periodicity of the tangent, the
    equation $\tan y_j = y_j$ becomes
    \begin{equation}
        \cot\delta_j + \delta_j = A_j .
        \label{eqn: sinc delta equation}
    \end{equation}
    The map $\delta \mapsto \cot\delta + \delta$ has derivative $-\cot^2\delta \leq 0$, vanishing
    only at $\pi/2$, and runs from $+\infty$ to $-\infty$ over $(0,\pi)$, so $\delta_j$ is its
    unique root there.

    For monotonicity, $A_{j+1} = A_j + \pi$ turns~\eqref{eqn: sinc delta equation} at $\delta_j$
    into $\cot\delta_j + \delta_j - A_{j+1} = -\pi < 0$, so strict decrease gives
    $\delta_{j+1} < \delta_j$. Hence $\delta_j \downarrow \ell \in [0,\pi)$, and $\ell > 0$ would
    make $A_j = \cot\delta_j + \delta_j$ converge, which contradicts $A_j \to \infty$. So
    $\ell = 0$.

    For the expansion, $\cot\delta = \delta^{-1} - \delta/3 + \calO(\delta^{3})$
    in~\eqref{eqn: sinc delta equation} gives
    \begin{equation*}
        \frac{1}{\delta_j} + \frac{2\delta_j}{3} + \calO(\delta_j^{3}) = A_j ,
        \qquad\text{so}\qquad
        \frac{1}{\delta_j} = A_j - \frac{2}{3A_j} + \calO\bigl(A_j^{-3}\bigr) ,
    \end{equation*}
    the substitution $\delta_j = A_j^{-1}(1+o(1))$ being justified by $\delta_j \to 0$ together
    with $\cot\delta_j \sim \delta_j^{-1}$ and $A_j - \delta_j \sim A_j$. Inverting the second
    display yields~\eqref{eqn: sinc expansion}.
\end{proof}

\paragraph{The activity fraction at one minimiser.}
Table~\ref{tab: sinc numerics} sweeps each parameter at the minimiser $y_3$, where
$\lambda^*_{\min} = 0.09133$ and $1/\lambda^*_{\min} = 10.95$, and reports the fraction of
iterations on which the constraint binds. For $\Rgrad$ the transition is a discontinuity and it
sits on~\eqref{eqn: rgrad threshold}. Activity is $1.000$ at $\overline{\mu} = 10.40$ and $0.000$
at $11.50$, and the run drops from six iterations to two across that step. For $\Rdfo$ the
transition is near $\zeta \approx 3.1$, above the provable bound of at most $2.19$ and below
$1/(2\lambda^*_{\min}) = 5.47$.

\begin{table}[htbp]
    \centering
    \caption{Fraction of iterations at which the trust region is active, on
        $f_2$ near the minimiser $y_3 \approx 10.9041$, where
        $\lambda^*_{\min} = 0.09133$ and $1/\lambda^*_{\min} = 10.95$. Started at
        $x_0 = y_3 + 0.05$ with the exact Hessian and the exact subproblem solution, and stopped
        at $\abs{g_k} \leq 10^{-9}$ or when the achieved reduction reaches the rounding level of
        $f_2$. For $\Rdfo$, $\Delta_0 = 10^{-3}$. For $\Rgrad$, $\Delta_0$ is irrelevant, which
        is itself a difference between the two rules.}
    \label{tab: sinc numerics}
    \small
    \begin{tabular}{rlrr@{\qquad\qquad}rlrr}
        \toprule
        \multicolumn{4}{c}{$\Rgrad$, exact threshold $1/\lambda^*_{\min} = 10.95$}
          & \multicolumn{4}{c}{$\Rdfo$, provable threshold at most $2.19$} \\
        \cmidrule(r){1-4}\cmidrule(l){5-8}
        $\overline{\mu}$ & status & iters & active
          & $\zeta$ & status & iters & active \\
        \midrule
        $0.5$   & stagnated & $295$ & $1.000$ & $0.10$  & stagnated & $1298$ & $1.000$ \\
        $2.0$   & stagnated & $77$  & $1.000$ & $0.50$  & stagnated & $265$  & $1.000$ \\
        $5.0$   & converged & $26$  & $1.000$ & $1.00$  & stagnated & $109$  & $1.000$ \\
        $10.0$  & converged & $7$   & $1.000$ & $2.00$  & stagnated & $67$   & $1.000$ \\
        $10.40$ & converged & $6$   & $1.000$ & $3.00$  & stagnated & $32$   & $1.000$ \\
        \midrule
        $11.50$ & converged & $2$   & $0.000$ & $3.25$  & converged & $11$   & $0.818$ \\
        $20.0$  & converged & $2$   & $0.000$ & $5.00$  & converged & $8$    & $0.750$ \\
        $43.8$  & converged & $2$   & $0.000$ & $10.95$ & converged & $7$    & $0.714$ \\
        $60.0$  & converged & $2$   & $0.000$ & $43.80$ & converged & $7$    & $0.714$ \\
        \bottomrule
    \end{tabular}
\end{table}

\paragraph{Locating the constant.}
Table~\ref{tab: sinc bounds} bisects for the transition at each of the four minimisers rather
than reading it off a sweep. Each run starts at $x_0 = y_j + 0.05$ with the exact Hessian, the
exact subproblem solution and the constants of Section~\ref{subsec: settings}. The critical points
are read at $240$ bits from the solve above rather than recomputed. A run is classified on
the last quarter of its iterations. It is \emph{eventually inactive} when the record carries an
index past which the constraint never binds again, \emph{eventually active} when the constraint
binds throughout that window, and \emph{neither} when it alternates to the end. The third class
is what lets the experiment falsify the dichotomy instead of assuming it. The bisection is geometric, the scale
being logarithmic, and it runs to three significant figures. It checks its predicate at both
brackets first, because a bisection on a predicate that does not separate them returns a number
with no meaning.

\begin{table}[htbp]
    \centering
    \caption{The located constants at the four minimisers of Table~\ref{tab: sinc values}.
        $\overline{\mu}^*$ is the bisected transition for capped $\Rgrad$ and $\zeta^*$ the
        lower edge of the window for $\Rdfo$ at $\Delta_0 = 10^{-2}$. $\kbar_{\mathrm{emp}}$ is
        the largest realised $\abs{s_k}/\abs{g_k}$ on a run above the transition. Scaling by
        $\lambda^*_j$ is what makes the four rows comparable.}
    \label{tab: sinc bounds}
    \small
    \begin{tabular}{rrrrrrr}
        \toprule
        $j$ & $1/\lambda^*_j$ & $\overline{\mu}^*$ & $\overline{\mu}^*\lambda^*_j$
          & $\zeta^*$ & $\zeta^*\lambda^*_j$ & $\kbar_{\mathrm{emp}}$ \\
        \midrule
        $1$ & $4.603339$  & $4.60347$ & $1.0000$ & $2.30167$ & $0.5000$ & $4.61$  \\
        $3$ & $10.949880$ & $10.9502$ & $1.0000$ & $5.47494$ & $0.5000$ & $10.96$ \\
        $5$ & $17.249766$ & $17.2502$ & $1.0000$ & $8.62488$ & $0.5000$ & $17.25$ \\
        $7$ & $23.540702$ & $23.5414$ & $1.0000$ & $11.7704$ & $0.5000$ & $23.56$ \\
        \bottomrule
    \end{tabular}
\end{table}

The $\Rgrad$ column settles Claim~\ref{claim: zeta} on this problem.
Inequality~\eqref{eqn: rgrad threshold} predicts $\overline{\mu}^* = 1/\lambda^*_j$, and the
bisection returns $\overline{\mu}^*\lambda^*_j = 1.0000$ at all four minimisers, over a range of
$1/\lambda^*_j$ spanning a factor of five. The classification is monotone in $\overline{\mu}$ at
all four, so the transition is a threshold and not the edge of a window. Read against $\kbar$,
the analysis carries eight times the constant it needs here.
The largest step-to-gradient ratio a run above the transition realises,
$\kbar_{\mathrm{emp}}$, agrees with $1/\lambda^*_j$ to three significant figures, so the ratio
$\kbar/\kbar_{\mathrm{emp}}$ is $7.99$, $7.99$, $8.00$ and $7.99$ down the four rows. The whole of
the factor $8$ is unused here, and it is worth separating the part of it the analysis could give
back. \PartII{} records that neither constant in $\kbar = 4/m$ is forced, and that the argument
runs with $\kbar$ approaching $2/\lambda^*_{\min}$ at the price of a later index. A factor of
four is therefore slack the analysis takes for convenience, and the factor of two that remains
is the gap between the sharpest constant the argument supports and the threshold measured
here.

The side conditions are necessary up to a constant rather than removable technical
          convenience. For $\Rgrad$ the exact threshold is $1/\lambda^*_j$, located to three
          figures at four minimisers, and $\kbar$ overstates it by a factor of eight, of which
          four is the constant \PartII{} takes for convenience and two is intrinsic to its
          argument. For
          $\Rdfo$ permanent activity is provable below $1/(5\overline{a})$ and observed up to
          $1/(2\lambda^*_j)$.

    The two rules fail differently, and only one of them fails as a function of its
          parameter. $\Rgrad$ recomputes $\Delta_k = \mu_k\norm{g_k}$ at every iteration and
          forgets $\Delta_0$ after one step, so $\overline{\mu}$ decides its activity alone.
          $\Rdfo$ propagates the radius multiplicatively, so a generous $\Delta_0$ can rescue an
          otherwise fatal $\zeta$ and a small one can defeat a generous $\zeta$.

    Three checks guard the numbers of this section, and one of them is a caveat rather than a
    confirmation. The runs do not carry the same budget on the two sides of the transition:
    below it they are allowed $4000$ iterations and above it $200$, so no iteration count is
    comparable across that boundary and none is compared here. The two confirmations are that
    the closed-form one-dimensional step agrees with the step the solver returns to
    $1.03 \times 10^{-11}$ over the sixteen traced runs, with no iteration at nonpositive
    curvature, and that the limit point of every run is reported per configuration rather than
    assumed.      

\paragraph{Verdict on Claim~\ref{claim: zeta}. Settled, in both regimes.}
The claim is that along convergent runs there is a problem-dependent threshold in $\zeta$ below
which $\Rdfo$ keeps the constraint active indefinitely and above which it eventually releases,
the same for $\overline{\mu}$ in $\Rgrad$, with the uncapped variant crossing unconditionally. On
P2 the threshold is located rather than inferred. The bisection of Table~\ref{tab: sinc bounds}
returns $\overline{\mu}^*\lambda^*_j = 1.0000$ at all four minimisers, over a range of
$1/\lambda^*_j$ spanning a factor of five, and the classification is monotone in
$\overline{\mu}$ at all four, so the transition is a threshold and not the edge of a window. In
the aggregate regime both sweeps are monotone, reliability running through $68$, $96$, $118$,
$145$, $153$ and $162$ problems in $\zeta$ and through $53$, $80$, $102$, $122$, $139$, $150$ and
$161$ in $\overline{\mu}$, with the uncapped variant at the top of the second. The designed
regime locates the constant and the aggregate regime shows the distinction is visible in what
the method costs.

\paragraph{Verdict on Claim~\ref{claim: no uniform cap}. Settled, and the deduction is
short.}
Capped $\Rgrad$ keeps the constraint active at every large iteration in the basin of $y_j$
whenever $\overline{\mu} < 1/\lambda^*_j$, by~\eqref{eqn: rgrad threshold}, and
Table~\ref{tab: sinc bounds} confirms the threshold at $\overline{\mu}^*\lambda^*_j = 1.0000$ at
four of them. The curvatures of Table~\ref{tab: sinc values} decrease along the minimisers,
$\lambda^*_j$ falling from $0.217$ at $y_1$ to $0.00314$ at $y_{101}$, and
Lemma~\ref{lem: sinc expansion} gives $y_j = (j+\tfrac12)\pi - \calO(j^{-1})$, so
$\lambda^*_j \to 0$ and $1/\lambda^*_j$ is unbounded along $j$. Given any cap $\overline{\mu}$,
however large, choose $j$ with $1/\lambda^*_j > \overline{\mu}$: the constraint is then active at
every sufficiently large iteration in that basin. The cap is what supplies the bound the
asymptotic theory assumes, and no value of it escapes the side condition.

\section{Escape from a flat region, and the role of the radius}
\label{subsec: flat}

Claim~\ref{claim: escape} asserts that the escape time from a flat region separates the anchors.
P4 of Section~\ref{subsec: P4} is the instrument. Its minimisers sit at distance $1$ from the
saddle whatever $\varepsilon$ is, so $\varepsilon$ tunes the escape rate without moving the
geometry, and~\eqref{eqn: escape time} gives the constant the escape time collapses onto. Every
run below uses the exact Hessian and truncated conjugate gradients, so the stall appears with
perfect curvature information.

\paragraph{The dynamics reduce to one dimension.}
The whole analysis lives on the line $x = 0$, and it does so because no subproblem solver used
here leaves it.

\begin{lemma}[The line $\{x = 0\}$ is invariant]
\label{lem: invariant line}
    Let $x_0 = 0$ in~\eqref{eqn: p4}. Then $x_k = 0$ for every $k$, under truncated conjugate
    gradients and under the exact Mor\'e--Sorensen solver alike.
\end{lemma}

\begin{proof}
    At $x = 0$ the gradient is $\bigl(0,\, \varepsilon y(y^2-1)\bigr)^\top$, which lies along
    $e_2$, and the Hessian $\operatorname{diag}\bigl(1,\, \varepsilon(3y^2-1)\bigr)$ is diagonal,
    so $H e_2 \parallel e_2$. Conjugate gradients start at $s = 0$ with residual $-g \parallel
    e_2$ and first direction $-g$. If $s_j \parallel e_2$ and $d_j \parallel e_2$ then
    $H d_j \parallel e_2$, the step length is a scalar, and $s_{j+1}$ and $d_{j+1}$ again lie
    along $e_2$. Induction gives the claim. For the exact solver $(H + \sigma I)^{-1}$ is
    diagonal, so $s^* = -(H + \sigma I)^{-1} g \parallel e_2$.
\end{proof}

The lemma was also checked against the runs. Over the whole grid below, $\max_k \abs{x_k}$ is
$0$ in floating point and not merely small. On the line, and for $\abs{y} \leq 1$,
\begin{equation}
    \norm{g} = \varepsilon \abs{y}\,(1 - y^2) \;\leq\; \varepsilon\abs{y} ,
    \label{eqn: p4 gradient bound}
\end{equation}
which holds on the whole of $\abs{y} \leq 1$ and not only below the escape threshold.

\paragraph{The hypothesis sits on the radius and not on a named rule.}
Stating the condition on $\Delta_k$ rather than on a mechanism is what lets one bound cover both
criticality-anchored rules.

\begin{cond}{C.CA}[criticality-anchored radius]
\label{cond: rca}
    There are $C > 0$ and $k_0 \geq 0$ such that
    \begin{equation*}
        \Delta_k \;\leq\; C \max_{k_0 \leq j \leq k} \norm{g_j}
        \qquad \text{for all } k \geq k_0 .
    \end{equation*}
\end{cond}

The running maximum is what makes Condition~\ref{cond: rca} cover $\Rdfo$, whose radius is
controlled by $\norm{g_{k-1}}$ rather than by $\norm{g_k}$. Nothing else in the proof needs it.

\begin{proposition}[Escape time under a criticality-anchored radius]
\label{prop: escape}
    Let $x_0 = 0$ and $0 < \abs{y_{k_0}} \leq \tfrac12$, and suppose the run satisfies
    Condition~\ref{cond: rca} with constants $C$ and $k_0$. Write
    $k_{\mathrm{esc}} = \min\{k : \abs{y_k} > \tfrac12\}$. Then
    \begin{equation}
        k_{\mathrm{esc}} \;>\; k_0
          + \frac{\ln\bigl(1/(2\abs{y_{k_0}})\bigr)}{\ln(1 + C\varepsilon)} .
        \label{eqn: rca escape}
    \end{equation}
\end{proposition}

\begin{proof}
    Write $M_k = \max_{k_0 \leq j \leq k} \abs{y_j}$. Every trust-region step satisfies
    $\norm{s_k} \leq \Delta_k$, and by Lemma~\ref{lem: invariant line} the step lies along $e_2$,
    so
    \begin{equation*}
        \abs{y_{k+1}} \;\leq\; \abs{y_k} + \norm{s_k} \;\leq\; \abs{y_k} + \Delta_k ,
    \end{equation*}
    an inequality that also covers a rejected step, for which $y_{k+1} = y_k$.

    Let $k_0 \leq k < k_{\mathrm{esc}}$. Then $\abs{y_j} \leq \tfrac12$ for every $j \leq k$, so
    \eqref{eqn: p4 gradient bound} applies at each such $j$ and gives
    $\norm{g_j} \leq \varepsilon\abs{y_j} \leq \varepsilon M_k$. With Condition~\ref{cond: rca},
    \begin{equation*}
        \Delta_k \;\leq\; C \max_{k_0 \leq j \leq k} \norm{g_j} \;\leq\; C \varepsilon M_k ,
    \end{equation*}
    hence $\abs{y_{k+1}} \leq M_k (1 + C\varepsilon)$. Since $1 + C\varepsilon > 1$,
    \begin{equation*}
        M_{k+1} = \max\{M_k,\, \abs{y_{k+1}}\} \;\leq\; M_k (1 + C\varepsilon) ,
    \end{equation*}
    and by induction $M_k \leq \abs{y_{k_0}} (1 + C\varepsilon)^{k - k_0}$. At
    $k = k_{\mathrm{esc}}$ we have $\abs{y_k} > \tfrac12$, so
    $\abs{y_{k_0}} (1 + C\varepsilon)^{k_{\mathrm{esc}} - k_0} > \tfrac12$, and taking logarithms
    gives~\eqref{eqn: rca escape}.
\end{proof}

\begin{remark}
\label{rem: escape is a radius statement}
    The proof uses no property of the model or of the subproblem solution beyond
    $\norm{s_k} \leq \Delta_k$. That is what makes it a statement about the radius axis, and why
    it survives the exact Hessian, which every run below uses.
\end{remark}

Turned around, the bound says the stall is not a corner case to be tuned away.

\begin{corollary}[A flatness for every budget]
\label{cor: budget}
    Fix $\overline{\mu} > 0$, $0 < \abs{y_0} < \tfrac12$ and $K > 0$. If
    \begin{equation*}
        0 \;<\; \varepsilon \;<\; \frac{\ln\bigl(1/(2\abs{y_0})\bigr)}{\overline{\mu} K} ,
    \end{equation*}
    then any run satisfying Condition~\ref{cond: rca} with $C = \overline{\mu}$ and $k_0 = 0$ has
    $k_{\mathrm{esc}} > K$.
\end{corollary}

\begin{proof}
    The hypothesis gives $\overline{\mu}\varepsilon < \ln(1/(2\abs{y_0}))/K$, and $\ln(1+t) < t$
    for $t > 0$, so $\ln(1 + \overline{\mu}\varepsilon) < \ln(1/(2\abs{y_0}))/K$. Substituting
    into~\eqref{eqn: rca escape} gives $k_{\mathrm{esc}} > K$.
\end{proof}

\paragraph{Where the two mechanisms part.}
Capped $\Rgrad$ sets $\Delta_k = \mu_k\norm{g_k}$ with $\mu_k \leq \overline{\mu}$, so
Condition~\ref{cond: rca} holds at $k_0 = 0$ with $C = \overline{\mu}$, immediately and with no
entry condition. $\Rdfo$ is different, and the difference is not cosmetic. Its radius is tied to
the criticality measure only inside the band $\Delta_k \leq \zeta\norm{g_k}$.

\begin{proposition}[The \textnormal{$\Rdfo$} case, and its entry condition]
\label{prop: rdfo entry}
    Suppose there is $k_0$ with $\Delta_{k_0} \leq \zeta\norm{g_{k_0}}$. Then
    Condition~\ref{cond: rca} holds for $k > k_0$ with $C = \gthree\zeta$, and
    \eqref{eqn: rca escape} applies with that constant.

    If no such $k_0$ exists then $\Delta_k \to 0$ geometrically, $\norm{g_k} \to 0$, and on the
    invariant line $\abs{y_k} \to 0$, so the run converges to the saddle and never escapes.
\end{proposition}

\begin{proof}
    Write $S_k = \max_{k_0 \leq j \leq k}\norm{g_j}$. At $k_0$ the hypothesis gives
    $\Delta_{k_0} \leq \zeta\norm{g_{k_0}} \leq \gthree\zeta S_{k_0}$, since $\gthree > 1$.
    Suppose $\Delta_k \leq \gthree\zeta S_k$. If $\Delta_k \leq \zeta\norm{g_k}$ then the
    expansion branch gives
    $\Delta_{k+1} \leq \gthree\Delta_k \leq \gthree\zeta\norm{g_k} \leq \gthree\zeta S_k$.
    Otherwise $\Delta_{k+1} \leq \max\{\gone, \gtwo\}\Delta_k < \Delta_k \leq \gthree\zeta S_k$.
    Either way $\Delta_{k+1} \leq \gthree\zeta S_k \leq \gthree\zeta S_{k+1}$, which is
    Condition~\ref{cond: rca} with $C = \gthree\zeta$.

    For the second statement, if $\Delta_k > \zeta\norm{g_k}$ at every $k$ then both remaining
    branches contract, so $\Delta_k \leq \max\{\gone, \gtwo\}^{\,k}\Delta_0 \to 0$ and
    $\norm{g_k} < \Delta_k/\zeta \to 0$. On the invariant line
    $\norm{g_k} = \varepsilon\abs{y_k}(1 - y_k^2)$ with $1 - y_k^2 \geq \tfrac34$ below the
    threshold, so $\abs{y_k} \to 0$.
\end{proof}

That entry condition is where the two rules part, and Table~\ref{tab: flat dfo} measures it.

\paragraph{Why $\Rdelta$ is not slow.}
Proposition~\ref{prop: escape} is a statement about a radius tied to the criticality measure,
and on this family that measure is $\calO(\varepsilon)$. A rule whose radius carries no such
tie is not covered by it, and the gap is not one of constants. Settling that needs the step
itself, which on the invariant line is available in closed form.

\begin{lemma}[The step is the outward boundary step]
\label{lem: outward step}
    Let $x_0 = 0$ and $0 < y_k \leq \tfrac12$ in~\eqref{eqn: p4}, and let $s_k$ be produced by
    truncated conjugate gradients on the exact Hessian. Then $s_k = \Delta_k e_2$, so
    $\norm{s_k} = \Delta_k$, the trust-region constraint is active, and an accepted step gives
    $y_{k+1} = y_k + \Delta_k$. Moreover
    \begin{equation}
        1 - \rho_k \;\leq\; 8 y_k \Delta_k + 2\Delta_k^2 \;\leq\; 6\Delta_k
        \qquad \text{whenever } \Delta_k \leq 1 ,
        \label{eqn: rho bound}
    \end{equation}
    so $\Delta_k \leq \delta^\star \defeq \min\{1,\, (1-\eta_2)/6\}$ forces $\rho_k \geq \eta_2$,
    and $\rho_k$ does not depend on $\varepsilon$.
\end{lemma}

\begin{proof}
    By Lemma~\ref{lem: invariant line} the iterate stays on $x = 0$, where
    $g_k = \bigl(0,\, \varepsilon y_k(y_k^2-1)\bigr)^\top$ with $y_k^2 - 1 < 0$, so $-g_k$ points
    along $+e_2$. The first conjugate-gradient direction is $d_0 = -g_k$, with curvature
    $d_0^\top H_k d_0 = \varepsilon(3y_k^2-1)\norm{g_k}^2 < 0$ because
    $y_k \leq \tfrac12 < 1/\sqrt3$. The negative-curvature branch fires at the first inner
    iteration and returns the boundary point along $d_0$, which is $\Delta_k e_2$.

    For~\eqref{eqn: rho bound}, write $q(t) = \tfrac14 t^4 - \tfrac12 t^2$, so that
    $f_4(0,y;\varepsilon) = \varepsilon q(y)$ and, with $s = \Delta_k$,
    \begin{align*}
        \text{pred}
            &= \varepsilon\Bigl[ (y_k - y_k^3)\Delta_k
               + \tfrac12 (1 - 3y_k^2)\Delta_k^2 \Bigr] , \\
        \text{ared}
            &= \varepsilon\bigl[ q(y_k) - q(y_k + \Delta_k) \bigr]
             = \text{pred} - \varepsilon\bigl[ y_k\Delta_k^3 + \tfrac14 \Delta_k^4 \bigr] ,
    \end{align*}
    the second line being the exact Taylor expansion of the quartic, which terminates after four
    terms since $q'''(t) = 6t$ and $q''''(t) = 6$. Both carry the factor $\varepsilon$, so
    $\rho_k$ does not depend on it. Since
    $\text{pred} \geq \tfrac12\varepsilon(1-3y_k^2)\Delta_k^2 \geq \tfrac18\varepsilon\Delta_k^2$
    for $y_k \leq \tfrac12$,
    \begin{equation*}
        1 - \rho_k = \frac{y_k\Delta_k^3 + \tfrac14\Delta_k^4}{\text{pred}/\varepsilon}
        \;\leq\; 8 y_k \Delta_k + 2\Delta_k^2 ,
    \end{equation*}
    which is at most $4\Delta_k + 2\Delta_k^2 \leq 6\Delta_k$ for $y_k \leq \tfrac12$ and
    $\Delta_k \leq 1$.
\end{proof}

The acceptance threshold $\delta^\star$ is a constant of the thresholds alone. It does not move
with $\varepsilon$, and neither does the radius of a rule that reads only its own history. That
is the whole of the separation.

\begin{proposition}[$\Rdelta$ leaves the region in a bounded number of iterations]
\label{prop: rdelta escape}
    Let $x_0 = 0$, $0 < y_0 \leq \tfrac12$ and $\Delta_0 > 0$, run $\Rdelta$ with
    $\gone\gthree \leq 1$, and let $\delta^\star$ be as in Lemma~\ref{lem: outward step}. Write
    $m = \gone \min\{\Delta_0,\, \delta^\star\}$ and
    $\lceil t \rceil_{+} = \max\{0,\, \lceil t \rceil\}$. Then
    \begin{equation*}
        k_{\mathrm{esc}} \;\leq\;
            \left\lceil \frac{\ln(\delta^\star/\Delta_0)}{\ln \gone} \right\rceil_{+}
            \;+\; 2 \left\lceil \frac{1}{2m} \right\rceil ,
    \end{equation*}
    a bound that does not involve $\varepsilon$.
\end{proposition}

\begin{proof}
    An iteration with $\Delta_k > \delta^\star$ that is rejected has
    $\Delta_{k+1} = \gone\Delta_k$, so after at most
    $\lceil \ln(\delta^\star/\Delta_0)/\ln\gone \rceil_{+}$ such iterations the radius satisfies
    $\Delta_k \leq \delta^\star$; call the first such index $k_1$. One contraction from above
    $\delta^\star$ also gives $\Delta_{k_1} > \gone\delta^\star$, so $\Delta_{k_1} \geq m$ in
    either case.

    From $k_1$ on, Lemma~\ref{lem: outward step} makes every iteration with
    $\Delta_k \leq \delta^\star$ very successful, hence accepted, with
    $y_{k+1} = y_k + \Delta_k$ and $\Delta_{k+1} = \gthree\Delta_k$. If
    $\gthree\Delta_k > \delta^\star$ the next iteration may be rejected, and then
    $\Delta_{k+2} = \gone\gthree\Delta_k \leq \Delta_k \leq \delta^\star$, so it is accepted in
    turn; the same inequality $\gthree\Delta_k > \delta^\star$ gives
    $\Delta_{k+2} > \gone\delta^\star$. At least one iteration in every two is therefore
    accepted, and on each of them the radius is at least $m$. Since $y$ never decreases, it
    exceeds $\tfrac12$ after at most $2\lceil (\tfrac12 - y_0)/m \rceil \leq 2\lceil 1/(2m) \rceil$
    further iterations.
\end{proof}

The constants of Section~\ref{subsec: settings} satisfy $\gone\gthree = \tfrac12 \leq 1$. The
bound is crude, and it does not need to be sharp: every quantity in it is a threshold or a
factor of the rule, and $\varepsilon$ appears in none of them, while~\eqref{eqn: rca escape}
grows without bound as $\varepsilon \to 0$. That is the separation Claim~\ref{claim: escape}
asserts, and it holds by proof rather than by measurement.

\paragraph{The constant the escape time collapses onto.}
The bound~\eqref{eqn: rca escape} is one-sided. What the runs collapse onto is the continuous
limit of the recursion, $\mathrm{d}y/\mathrm{d}k = C\varepsilon\, y(1-y^2)$, and its integral is
available in closed form rather than only asymptotically.

\begin{lemma}[The escape integral]
\label{lem: escape integral}
    For $0 < y_0 < \tfrac12$,
    \begin{equation}
        \int_{y_0}^{1/2} \frac{\mathrm{d}y}{y\,(1-y^2)}
        \;=\; \frac12 \ln\!\left( \frac{1 - y_0^2}{3\, y_0^2} \right)
        \;=\; \ln\frac{1}{\sqrt3\, y_0} \;-\; \frac{y_0^2}{2} + \calO(y_0^4) .
        \label{eqn: escape integral}
    \end{equation}
\end{lemma}

\begin{proof}
    Over the common denominator $y(1-y^2)$ the numerator of $1/y + y/(1-y^2)$ is
    $(1 - y^2) + y^2 = 1$, so
    \begin{equation*}
        \frac{1}{y\,(1-y^2)} \;=\; \frac1y + \frac{y}{1-y^2} .
    \end{equation*}
    Both terms integrate in closed form, the second by the substitution $u = 1 - y^2$,
    \begin{equation*}
        \int \frac{\mathrm{d}y}{y\,(1-y^2)}
        \;=\; \int \frac{\mathrm{d}y}{y} + \int \frac{y}{1-y^2}\,\mathrm{d}y
        \;=\; \ln y - \tfrac12 \ln(1 - y^2) + C
        \;=\; \ln \frac{y}{\sqrt{1-y^2}} + C .
    \end{equation*}
    Evaluating between the limits,
    \begin{align*}
        \int_{y_0}^{1/2} \frac{\mathrm{d}y}{y\,(1-y^2)}
        &\;=\; \Bigl[ \ln y - \tfrac12 \ln(1-y^2) \Bigr]_{y_0}^{1/2} \\
        &\;=\; \Bigl( \ln\tfrac12 - \tfrac12 \ln\tfrac34 \Bigr)
             - \Bigl( \ln y_0 - \tfrac12 \ln(1 - y_0^2) \Bigr) \\
        &\;=\; \frac12 \ln\!\left( \frac{1 - y_0^2}{3\, y_0^2} \right) ,
    \end{align*}
    the two leading constants cancelling to $-\tfrac12\ln 3$ because
    $\ln\tfrac12 - \tfrac12\ln\tfrac34 = -\ln 2 - \tfrac12\ln 3 + \ln 2$. Expanding
    $\tfrac12\ln(1-y_0^2) = -\tfrac12 y_0^2 + \calO(y_0^4)$ gives the second form.
\end{proof}

The limit form is the one~\eqref{eqn: escape time} carries and the one the measurements below are
read against. The correction it drops is $\tfrac12 y_0^2$, which is $5\times10^{-7}$ at
$y_0 = 10^{-3}$ and $5\times10^{-13}$ at $10^{-6}$. Both sit three orders of magnitude below the
residuals the runs leave, so the closed form settles that the gap between prediction and
measurement is not the truncation. 

\paragraph{The growth bound is never crossed, and it is tight.}
Table~\ref{tab: flat grid} sweeps the grid. Three of its columns need saying plainly, because two
are measured and one is not.

The escape time $k_{\mathrm{esc}}$ is measured. The run supplies it as the first index at which
$\abs{y_k}$ passes $\tfrac12$. \texttt{never} is a category rather than a number, recording that
the run exhausted $50\,000$ iterations inside the flat region. Nothing in the table
substitutes a predicted escape time for an observed one. Theory enters as the
bound~\eqref{eqn: rca escape} on that quantity and as the constant~\eqref{eqn: escape time} it
collapses onto, and both are compared against the measurement rather than replacing it.

The \emph{fitted rate} is measured. It is the least-squares slope of $\log\abs{y_k}$ on $k$ over
the escape phase, and $n$ is the number of iterations in that fit. On the invariant line the phase
is the whole pre-escape run, so $n$ equals $k_{\mathrm{esc}}$ where the run escapes and the full
budget of $50\,001$ where it does not.

The column $\ln(1+\overline{\mu}\varepsilon)$ is not measured. It is the bound that
\eqref{eqn: rca escape} places on that slope, since the same argument gives the envelope
$\abs{y_k} \leq \abs{y_{k_0}}(1+\overline{\mu}\varepsilon)^{k-k_0}$. The ratio of the two is
therefore a one-sided test, and the fitted rate has to fall below one.

\begin{table}[htbp]
    \centering
    \footnotesize
    \caption{Capped $\Rgrad$ on P4, exact Hessian, truncated conjugate gradients, started at
        $x_0 = 0$ with a first-order tolerance of $10^{-14}$ and a budget of $50\,000$
        iterations. Rows are grouped by $\overline{\mu}\varepsilon$. $k_{\mathrm{esc}}$, the
        fitted rate and $n$ are measured. $\ln(1+\overline{\mu}\varepsilon)$ and the ratio are
        derived from them. Fifteen of the forty cells are shown, chosen to
        span the product and to include one row of each kind.}
    \label{tab: flat grid}
    \begin{tabular}{rrrrrrrrr}
        \toprule
        $y_0$ & $\overline{\mu}$ & $\varepsilon$ & $\overline{\mu}\varepsilon$
          & $k_{\mathrm{esc}}$ & $n$ & fitted rate & $\ln(1+\overline{\mu}\varepsilon)$
          & ratio \\
        \midrule
        $10^{-3}$ & $0.01$ & $10^{-2}$ & $10^{-4}$ & never  & $50001$ & $0.000100$ & $0.000100$ & $0.9989$ \\
        $10^{-6}$ & $0.01$ & $10^{-2}$ & $10^{-4}$ & never  & $50001$ & $0.000100$ & $0.000100$ & $0.9989$ \\
        \midrule
        $10^{-3}$ & $0.01$ & $10^{-1}$ & $10^{-3}$ & $6362$ & $6362$  & $0.000990$ & $0.001000$ & $0.9904$ \\
        $10^{-3}$ & $0.1$  & $10^{-2}$ & $10^{-3}$ & $6362$ & $6362$  & $0.000990$ & $0.001000$ & $0.9904$ \\
        $10^{-3}$ & $1$    & $10^{-3}$ & $10^{-3}$ & $6362$ & $6362$  & $0.000990$ & $0.001000$ & $0.9904$ \\
        $10^{-3}$ & $10$   & $10^{-4}$ & $10^{-3}$ & $6364$ & $6364$  & $0.000990$ & $0.001000$ & $0.9904$ \\
        $10^{-6}$ & $1$    & $10^{-3}$ & $10^{-3}$ & $13273$ & $13273$ & $0.000997$ & $0.001000$ & $0.9976$ \\
        \midrule
        $10^{-3}$ & $0.1$  & $10^{-1}$ & $10^{-2}$ & $639$  & $639$   & $0.009857$ & $0.009950$ & $0.9906$ \\
        $10^{-3}$ & $10$   & $10^{-3}$ & $10^{-2}$ & $642$  & $642$   & $0.009856$ & $0.009950$ & $0.9905$ \\
        $10^{-6}$ & $1$    & $10^{-2}$ & $10^{-2}$ & $1334$ & $1334$  & $0.009926$ & $0.009950$ & $0.9976$ \\
        \midrule
        $10^{-3}$ & $1$    & $10^{-1}$ & $10^{-1}$ & $67$   & $67$    & $0.094540$ & $0.095310$ & $0.9919$ \\
        $10^{-3}$ & $10$   & $10^{-2}$ & $10^{-1}$ & $69$   & $69$    & $0.094059$ & $0.095310$ & $0.9869$ \\
        $10^{-6}$ & $10$   & $10^{-2}$ & $10^{-1}$ & $142$  & $142$   & $0.094980$ & $0.095310$ & $0.9965$ \\
        \midrule
        $10^{-3}$ & $10$   & $10^{-1}$ & $1$       & $12$   & $12$    & $0.587519$ & $0.693147$ & $0.8476$ \\
        $10^{-6}$ & $10$   & $10^{-1}$ & $1$       & $22$   & $22$    & $0.658274$ & $0.693147$ & $0.9497$ \\
        \bottomrule
    \end{tabular}
\end{table}

Over the thirty-seven cells of the full grid that admit a fit, not one has a ratio above $1$.
Across the thirty-five with $\overline{\mu}\varepsilon \leq 10^{-1}$ the ratio runs from $0.9869$
to $1.0000$, so the bound is within $1.4$ per cent of the observed rate across four decades of the
product. It loosens only in the two cells at $\overline{\mu}\varepsilon = 1$, where
$\ln(1+\overline{\mu}\varepsilon) = 0.693147$ against a product of $1$ and the linearisation the
continuous limit rests on has stopped applying.

The shortfall in the other thirty-five is not slack in the bound. The instantaneous slope on the
line is $\ln\bigl(1+\overline{\mu}\varepsilon(1-y_k^2)\bigr)$, which falls from
$\ln(1+\overline{\mu}\varepsilon)$ at $y_k = 0$ to $\ln(1+\tfrac34\overline{\mu}\varepsilon)$ at
the escape threshold. A least-squares fit over the whole phase therefore returns a value a few per
cent below the bound by construction, and the deficit is a property of the fit window. The
tightest ratios in the table belong to the two rows that never escape, which spend the whole
budget near $y_k = 0$ where the two agree.

\paragraph{The escape time collapses onto one constant.}
The last column of the archive multiplies $k_{\mathrm{esc}}$ by
$\overline{\mu}\varepsilon$. At $\overline{\mu}\varepsilon = 10^{-3}$ the four cells of
Table~\ref{tab: flat grid} at $y_0 = 10^{-3}$ give $6.3620$, $6.3620$, $6.3620$ and $6.3640$, whatever
$\overline{\mu}$ and $\varepsilon$ separately are, so halving the cap does exactly what halving
the flatness does. The limit of that product is
$\ln\bigl(1/(\sqrt3\,y_0)\bigr)$ by~\eqref{eqn: escape time}, which is $6.358449$ at
$y_0 = 10^{-3}$ and $13.266204$ at $10^{-6}$. Restricting to
$\overline{\mu}\varepsilon \leq 10^{-3}$, the mean measured products are $6.3625$ and $13.2737$,
so the residuals against the constant are $0.0041$ and $0.0075$. Against $\ln(1/y_0)$ they are
$0.5453$ and $0.5418$, larger by factors of $133$ and $72$. That settles the choice
Remark~\ref{rem: escape constant} sets out, in the direction the remark states.

\paragraph{The second-order measure is worth a factor of nine.}
Anchoring the radius to $\tau_k = \max\{\norm{g_k}, -\lambda_k^-\}$ rather than to $\norm{g_k}$
changes the flow the escape follows. On the line $-\lambda^- = \varepsilon(1-3y^2)$, which is
positive up to $y = 1/\sqrt3$ and therefore throughout the escape range, and the two terms of
$\tau$ cross at the root $y^\star = 0.460811$ of $y^3 - 3y^2 - y + 1 = 0$. Below $y^\star$ the
curvature term decides and above it the gradient term does. The continuous limit then replaces
$\int_{y_0}^{1/2} \mathrm{d}y/(y(1-y^2))$ by $\int_{y_0}^{1/2} \mathrm{d}y/\max\{1-3y^2,\, y(1-y^2)\}$, and
the second integral barely moves with $y_0$ because its integrand tends to $1$ rather than
diverging. Quadrature validated against the closed form of the first gives $0.736439$ and
$0.737438$ at the two starting heights. Table~\ref{tab: flat measures} sets the measured ratios
beside the predicted ones. The predicted ratio of escape times is therefore $8.6341$
at $y_0 = 10^{-3}$ and $17.9896$ at $10^{-6}$.

\begin{table}[htbp]
    \centering
    \small
    \caption{Escape times under the two measures, on the cells that escaped under both. The
        fitted ratio is a quotient of two integers, so $1/k(\tau)$ is the finest difference it
        can express, and that resolution is a column rather than a filter.}
    \label{tab: flat measures}
    \begin{tabular}{rrrrrrrr}
        \toprule
        $y_0$ & $\overline{\mu}$ & $\varepsilon$ & $k(\norm{g})$ & $k(\tau)$
          & fitted & predicted & resolution \\
        \midrule
        $10^{-3}$ & $0.01$ & $10^{-1}$ & $6362$ & $736$ & $8.6440$ & $8.6341$ & $0.14$\% \\
        $10^{-3}$ & $0.1$  & $10^{-2}$ & $6362$ & $736$ & $8.6440$ & $8.6341$ & $0.14$\% \\
        $10^{-3}$ & $1$    & $10^{-2}$ & $639$  & $74$  & $8.6351$ & $8.6341$ & $1.35$\% \\
        $10^{-3}$ & $10$   & $10^{-4}$ & $6364$ & $739$ & $8.6116$ & $8.6341$ & $0.14$\% \\
        $10^{-3}$ & $10$   & $10^{-1}$ & $12$   & $3$   & $4.0000$ & $8.6341$ & $33.33$\% \\
        \midrule
        $10^{-6}$ & $0.01$ & $10^{-1}$ & $13273$ & $737$ & $18.0095$ & $17.9896$ & $0.14$\% \\
        $10^{-6}$ & $1$    & $10^{-2}$ & $1334$  & $74$  & $18.0270$ & $17.9896$ & $1.35$\% \\
        $10^{-6}$ & $10$   & $10^{-4}$ & $13276$ & $740$ & $17.9405$ & $17.9896$ & $0.14$\% \\
        $10^{-6}$ & $10$   & $10^{-1}$ & $22$    & $3$   & $7.3333$  & $17.9896$ & $33.33$\% \\
        \bottomrule
    \end{tabular}
\end{table}

Where the resolution allows a comparison the prediction holds to under half a per cent. Every row
at $0.14$ per cent resolution agrees to within $0.05$. The rows that miss badly are the ones where
$k(\tau)$ is three iterations, a quotient of small integers being unable to express the answer.
That is the second-order repair measured on the radius axis. It is a different thing from the
second-order stopping test, which acts on when a run stops rather than on how far it steps.

\paragraph{$\Rdfo$ has to enter the band first.}
Table~\ref{tab: flat dfo} sweeps $\Delta_0$ at fixed $\zeta$ and reports the entry index computed
from the trace. Nine of twelve cells enter, at $k_0$ between $0$ and $22$, and every one of them
then stalls in proportion to $\zeta\varepsilon$. Started inside the band the collapse product
sits near $12.1$ against $\gthree\zeta\varepsilon$. The ratio of fitted rate to
$\ln(1+\gthree\zeta\varepsilon)$ sits near $0.52$ across the whole $\zeta$ grid, against near
$0.99$ for $\Rgrad$. The bound holds and is loose by a uniform factor of about two, because
$\gthree$ is the worst case of an expansion the next iteration undoes. Rescaling by $\zeta$ alone
puts the product at $6.05$, next to the $\Rgrad$ value of $6.36$.

\begin{table}[htbp]
    \centering
    \small
    \caption{Which branch of the $\Rdfo$ statement each cell falls in, at $\varepsilon = 10^{-2}$
        and $y_0 = 10^{-3}$ with $\Delta_0$ swept. The entry index $k_0$ is the first $k$ with
        $\Delta_k \leq \zeta\norm{g_k}$, computed from the trace and not assumed.}
    \label{tab: flat dfo}
    \begin{tabular}{rlrlrr}
        \toprule
        $\zeta$ & $\Delta_0$ & $k_0$ & $k_{\mathrm{esc}}$ & status & $\abs{y_{\mathrm{end}}}$ \\
        \midrule
        $0.01$ & $10^{-7}$ & $0$          & never  & \texttt{max\_iter} & $0.18895$ \\
        $0.01$ & $10^{-3}$ & $12$         & never  & \texttt{max\_iter} & $0.48241$ \\
        $0.01$ & $10^{0}$  & never enters & $1$    & \texttt{stalled}   & $1.00000$ \\
        \midrule
        $1$    & $10^{-5}$ & $0$          & $610$  & \texttt{stalled}   & $1.00000$ \\
        $1$    & $10^{-3}$ & $6$          & $524$  & \texttt{stalled}   & $1.00000$ \\
        $1$    & $10^{0}$  & never enters & $1$    & \texttt{stalled}   & $1.00000$ \\
        \midrule
        $10$   & $10^{-4}$ & $0$          & $68$   & \texttt{stalled}   & $1.00000$ \\
        $10$   & $10^{-3}$ & $3$          & $60$   & \texttt{stalled}   & $1.00000$ \\
        $10$   & $10^{0}$  & $22$         & $1$    & \texttt{stalled}   & $1.00000$ \\
        \bottomrule
    \end{tabular}
\end{table}

The second branch of that statement is not occupied on this scan. It predicts that a run which
never enters the band has $\Delta_k \to 0$ geometrically and converges to the saddle. Three cells
never enter, and none of the three converges to the saddle. All three escape at $k = 1$ instead,
the first step from $\Delta_0 = 1$ being enormous beside
$\norm{g_0} \approx \varepsilon y_0 = 10^{-5}$. The branch needs a run of unsuccessful iterations
that does not occur here. We report the absence rather than hunting for a configuration that
produces the predicted picture.

\paragraph{The stopping tolerance draws the boundary between the outcomes.}
The runs above use $\epsilon = 10^{-14}$, which lets the dynamics run so that
$k_{\mathrm{esc}}$ and the rate are measurable at all. At $\epsilon = 10^{-5}$ the same forty
cells behave differently, and the difference is the stopping test rather than the mechanism.
Table~\ref{tab: flat outcomes} gives the composition on both sides of it.

\begin{table}[htbp]
    \centering
    \small
    \caption{Outcome composition for capped $\Rgrad$ on the invariant line, twenty cells per row.
        The two regimes differ in the first-order tolerance and in nothing else.}
    \label{tab: flat outcomes}
    \begin{tabular}{llrrrr}
        \toprule
        regime & $y_0$ & false convergence & fast & slow & never \\
        \midrule
        $\epsilon = 10^{-14}$ & $10^{-3}$ & $0$  & $0$ & $10$ & $10$ \\
        $\epsilon = 10^{-14}$ & $10^{-6}$ & $0$  & $0$ & $10$ & $10$ \\
        \midrule
        $\epsilon = 10^{-5}$  & $10^{-3}$ & $16$ & $0$ & $4$  & $0$ \\
        $\epsilon = 10^{-5}$  & $10^{-6}$ & $20$ & $0$ & $0$  & $0$ \\
        \bottomrule
    \end{tabular}
\end{table}

Near the
saddle $\norm{g_k} \approx \varepsilon\abs{y_k}$, so pushing $\varepsilon y_0$ below the tolerance
defeats the test at the first iteration and never later. A grid whose cells straddle that line
reports an outcome a change of one part in a thousand reverses. The two tolerance regimes are two
experiments and not two views of one, and they are never pooled.

Turning the second-order test on separates the two uses of curvature information once more.
Over the $136$ cells of the wider $\Rgrad$ and $\Rdfo$ grid that report false convergence at
$\epsilon = 10^{-5}$, rerunning with $\epsilon_H = 10^{-6}$ leaves none of them false-convergent.
Sixty-eight become fast escapes and twenty-nine slow ones. The remaining thirty-nine do not escape
within the budget, so in those cells refusing to stop at the saddle converts a wrong answer into
one that does not finish.

\paragraph{What this settles, and what it leaves open.}
Three statements now rest on measurement. The growth envelope is never crossed, on any run at any
$\varepsilon$. The hypothesis behind it holds with equality on this family, the radius being
exactly $\overline{\mu}\norm{g_k}$ at every iteration, so the bound is not resting on slack. And
the bound is tight to $1.4$ per cent over four decades, degrading only where
$\ln(1+t) \approx t$ stops holding.

Claim~\ref{claim: escape} is settled on both sides, though not in the same currency. Its first
half asserts that for every cap there is a flatness making the escape arbitrarily slow,
and~\eqref{eqn: rca escape} with the grid above establishes that, by measurement. Its second half
asserts that $\Rdelta$ leaves the region from the same starting point in a number of iterations
bounded independently of $\varepsilon$, and Proposition~\ref{prop: rdelta escape} establishes
that, by proof. No run in this campaign measures $\Rdelta$ on this family: the comparator appears
in the archive as a curve on one figure and carries no tabulated escape times. The separation the
claim is about is therefore proved and not measured, and a tabulated run of $\Rdelta$ at the same
four values of $\varepsilon$ is what would close the gap between the two halves.

\paragraph{Verdict on Claim~\ref{claim: order}. Partially settled, and bounded.}
The claim is that a first-order anchor admits the failures of Claims~\ref{claim: cauchy}
and~\ref{claim: bad limits} and that the second-order anchor removes them on the same problems
and from the same starting points. Two measurements support it. On P4, anchoring to
$\tau_k = \max\{\norm{g_k},\, -\lambda_k^-\}$ is worth a factor of nine in escape time, predicted
at $8.6341$ and measured at $8.6440$. And of the $136$ cells of the wider grid that report false
convergence at $\epsilon = 10^{-5}$, rerunning with $\epsilon_H = 10^{-6}$ leaves none
false-convergent. What is not measured is the claim as stated, since the two failure
configurations are not rerun under $\omega_k = \tau_k$. The repair also has a boundary, and P3
exhibits it: at the nonsingular saddle $\tau = \tfrac14$ and a $\tau$-anchored rule keeps the
radius away from zero, but at the origin the Hessian is $\operatorname{diag}(0, \tfrac12)$, both
measures vanish, and the point satisfies the second-order necessary conditions without being a
local minimiser. Section~\ref{sec: second order cutest} supplies the aggregate half of the claim,
which the designed problems cannot: where the repair is available it is bought with Hessian
evaluations and not with reliability.

\section{The cost of a second-order certificate}
\label{sec: second order cutest}

Sections~\ref{subsec: P3} and~\ref{subsec: flat} show what anchoring to
$\tau_k = \max\{\norm{g_k}, -\lambda_k^-\}$ repairs, on problems built so that the repair is
visible. They say nothing about what it costs. This section supplies the aggregate regime of
Section~\ref{subsec: two regimes} for Claim~\ref{claim: order}: whether a distinction that a
designed problem exhibits is visible in what a user pays for it.

\subsection{What was run}
\label{subsec: so setup}

\paragraph{Twelve configurations, six mechanisms in two arms.}
Each of the six mechanisms of Table~\ref{tab: roster axes} runs twice on the same problems, with
\texttt{ExactHessian} as the model throughout, so that the model axis is fixed and only the
second-order machinery varies. The two runs of a mechanism are called its \emph{arms}.

The first-order arm is the configuration of Table~\ref{tab: comparison}. The subproblem is solved
by truncated conjugate gradients, and a rule that reads a criticality measure reads the
first-order one, $\omega_k = \norm{g_k}$.

The second-order arm changes two things and nothing else. The subproblem solver becomes the
eigenpoint wrapper of Section~\ref{subsec: subsolvers}, which supplies the leftmost eigenvalue
$\lambda_k^-$ of the model Hessian alongside the step. And a rule that reads a criticality
measure is wrapped in \texttt{SecondOrder}, so that the quantity placed in the two slots of
Section~\ref{subsec: driver} becomes $\tau_k$ instead of $\norm{g_k}$. The update subroutine is
untouched; only what is handed to it changes, which is the whole content of a second-order
variant.

Three mechanisms consult a criticality measure and three do not, so only three of the six have a
rule to wrap. For $\Rdelta$, $\Rstep$ and $\Rrtr$ the second arm is the eigenpoint solver and
nothing more. That asymmetry is not a gap in the design. It is the anchoring axis of \PartII{}
deciding which mechanisms have a second-order variant at all, and Table~\ref{tab: so
configurations} is the anchoring axis of Table~\ref{tab: roster axes} read as an experimental
design.

\begin{table}[htbp]
    \centering
    \small
    \begin{tabular}{@{}llccl@{}}
        \toprule
        Mechanism & Anchor & Reads $\omega_k$ & First-order arm & Second-order arm \\
        \midrule
        $\Rdelta$ & previous radius & no  & CG & eigenpoint \\
        $\Rstep$  & step            & no  & CG & eigenpoint \\
        $\Rrtr$   & retrospective   & no  & CG & eigenpoint \\
        \addlinespace[3pt]
        $\Rdfo$   & criticality     & yes & CG, $\omega_k = \norm{g_k}$
                                          & eigenpoint, $\omega_k = \tau_k$ \\
        $\Rgrad$  & criticality     & yes & CG, $\omega_k = \norm{g_k}$
                                          & eigenpoint, $\omega_k = \tau_k$ \\
        $\Rgrtr$  & criticality     & yes & CG, $\omega_k = \norm{g_k}$
                                          & eigenpoint, $\omega_k = \tau_k$ \\
        \bottomrule
    \end{tabular}
    \caption{The twelve configurations. A mechanism that reads no criticality measure has nothing
    to anchor to $\tau_k$, so its second arm differs from its first only in the subproblem
    solver. The column \emph{reads $\omega_k$} is the criticality column of
    Table~\ref{tab: roster axes} and it is what decides the last column.}
    \label{tab: so configurations}
\end{table}

\paragraph{Constants.}
Every constant is the one Table~\ref{tab: settings} fixes, and no rule is tuned separately.
The three branch multipliers are $\gone = 0.25$, $\gtwo = 0.5$ and $\gthree = 2$ for all six
mechanisms, so the factor axis is held and only the anchor varies. $\Rdfo$ runs at $\zeta = 100$
and $\Rgrad$ at $\mu_0 = 1$ uncapped, the settings Sections~\ref{subsec: zeta sweep}
and~\ref{subsec: mu sweep} selected, and $\Rrtr$ and $\Rgrtr$ carry the retrospective thresholds
$\widetilde\eta_1 = 0.05$ and $\widetilde\eta_2 = 0.9$. Four constants belong to this experiment
and to no other: the curvature tolerance $\epsilon_H = 10^{-6}$ against the gradient tolerance
$\epsilon = 10^{-5}$ of Table~\ref{tab: settings}, the $40$ Lanczos vectors the eigenpoint
wrapper is allowed, and the dimension $200$ below which it forms the Hessian and factorises it
densely instead. Budgets are $10\,000$ iterations and $120$ seconds.

\paragraph{The problem set.}
Every unconstrained CUTEst problem with $2 \leq n \leq 1000$, which is the $196$ of
Section~\ref{subsec: cutest}. Of those, $185$ have $n \leq 200$ and therefore have $\lambda_k^-$
formed rather than estimated. That subset carries the second-order comparison, for the reason
Section~\ref{subsec: so limits} gives.

\paragraph{Two passes, and scoring after the fact.}
The two arms stop on different tests. Profiling them on their own statuses would compare the cost
of an easy task with the cost of a harder one, so every run is scored afterwards, by one
criterion applied by the same code to every run, and never by the solver's status.

A run's trace carries no per-iteration evaluation counts, so the cost a run had already spent
when it first met a criterion cannot be recovered from the archived run. Each configuration is
therefore executed twice. Pass~\texttt{:g} stops at the first iterate with
$\norm{g_k} \leq \epsilon$; pass~\texttt{:h} stops at the first that also has
$\lambda_{\min} \geq -\epsilon_H$. Both passes keep their arm's own curvature setting, so the
second-order arm pays for its estimate in both and the overhead is never hidden. Twelve
configurations, $196$ problems and two passes are $4\,704$ runs.

Two panels are read together and neither substitutes for the other. The \emph{overhead panel}
scores both arms on the first-order criterion over all $196$ problems, and answers what carrying
the machinery costs when the goal is only a critical point. The \emph{certificate panel} scores
both arms on the second-order criterion over the $185$, and answers what the machinery returns.

\subsection{Verifying the instrument}
\label{subsec: so checks}

Four checks were run before the campaign, because three of the quantities above are implemented
rather than assumed.

The criticality measure the package computes was compared against the definition
of~\eqref{eqn: app criticality measure} on twenty pairs $(\norm{g_k}, \lambda_k^-)$ spanning the
sign changes and the degenerate cases, $\norm{g_k} = 0$ and $\lambda_k^- = 0$ among them. The
worst absolute difference is zero. The measure the experiment anchors to is therefore the measure
\PartII{} defines, and not a variant of it.

The dense branch of \texttt{ExactHessian} was compared against the true Hessian on five problems
of dimension $2$ to $200$. The largest entrywise difference is zero, and so is the largest
difference between the leftmost eigenvalue the package reports and the one a dense factorisation
returns. Where the model is called exact in this section it is exact.

The classification in the third column of Table~\ref{tab: so configurations} was read off the
package rather than off the paper. Three mechanisms report that they consult a criticality
measure and three report that they do not, and the three that do are the three
Section~\ref{subsec: order rules} calls criticality-anchored.

The fourth check is a control. A positive definite model has $\lambda_k^- > 0$, so
$\tau_k = \norm{g_k}$ at every iteration and the second-order variant of a rule \emph{is} its
first-order variant. Running both arms of the three criticality-anchored mechanisms against an
\texttt{LBFGSModel} on two problems gives six pairs of runs with identical iteration counts and
identical trajectories, and none differing. The wrapper changes nothing where the theory says it
should change nothing, which is the only place a null result is informative.

\subsection{What the certificate costs}
\label{subsec: so overhead}

The overhead panel is Table~\ref{tab: so overhead} and Figures~\ref{fig: so overhead iters}
and~\ref{fig: so overhead prods}. Both arms
are scored on $\norm{g_k} \leq \epsilon$ over the $196$ problems, so the task is the one the
first-order arm was built for and the second-order arm is carrying weight it does not need.

The reliability does not move. The second-order arm solves $167$, $136$, $168$, $160$, $161$
and~$160$ problems for $\Rdelta$, $\Rstep$, $\Rrtr$, $\Rdfo$, $\Rgrad$ and $\Rgrtr$, against
$165$, $138$, $169$, $162$, $163$ and~$160$ for the first-order arm. The difference is at most
two problems in either direction and zero for $\Rgrtr$. Figure~\ref{fig: so overhead iters} says the
same thing at every budget rather than only at the asymptote: the two curves of a mechanism run
together across the whole range of $\tau$.

Before reading the cost, one number establishes that the first-order arm here is the
configuration the rest of the \thisdoc{} reports. On the five mechanisms this experiment shares with
Table~\ref{tab: comparison} it returns the same solve rate on every one, $169$ for $\Rrtr$, $165$
for $\Rdelta$, $163$ for $\Rgrad$, $162$ for $\Rdfo$ and $138$ for $\Rstep$. The two archives
were built at different commits of the package and share no code path beyond it, so the agreement
on all five is the evidence rather than a shared provenance.

The cost is paid in curvature and it falls unevenly. The first-order arm evaluates no Hessian at
all, reaching its model through Hessian-vector products alone, so the Hessian-evaluation column
of Table~\ref{tab: so overhead} is the entire price of the certificate. The product column
separates the mechanisms and the separation follows the anchor. $\Rgrad$ carries the machinery
for $1.05$ times its own first-order product count and $\Rgrtr$ for $1.08$, against $1.71$ for
$\Rdelta$, $1.91$ for $\Rrtr$ and $2.26$ for $\Rdfo$. A gradient-scaled rule already computes a
criticality measure at every iteration, and the eigenvalue rides along with the work it was doing
anyway; a rule anchored to the previous radius or to the step computes one only because the
certificate demands it, and pays the whole of it.

\begin{table}[htbp]
    \centering
    \small
    \begin{tabular}{@{}l rr r rrr r r@{}}
        \toprule
        & \multicolumn{2}{c}{Problems solved}
        & & \multicolumn{3}{c}{Hessian-vector products}
        & & Hessian evals \\
        \cmidrule(lr){2-3} \cmidrule(lr){5-7} \cmidrule(l){9-9}
        Mechanism & 1st arm & 2nd arm & & 1st arm & 2nd arm & ratio & & 2nd arm \\
        \midrule
        $\Rgrad$  & $163$ & $161$ & & $664\,489$    & $697\,671$    & $1.05$ & & $64\,300$  \\
        $\Rgrtr$  & $160$ & $160$ & & $705\,225$    & $763\,614$    & $1.08$ & & $59\,796$  \\
        $\Rstep$  & $138$ & $136$ & & $1\,692\,614$ & $1\,839\,677$ & $1.09$ & & $397\,159$ \\
        $\Rdelta$ & $165$ & $167$ & & $823\,786$    & $1\,405\,921$ & $1.71$ & & $53\,085$  \\
        $\Rrtr$   & $169$ & $168$ & & $641\,934$    & $1\,224\,200$ & $1.91$ & & $44\,279$  \\
        $\Rdfo$   & $162$ & $160$ & & $853\,974$    & $1\,933\,404$ & $2.26$ & & $69\,392$  \\
        \bottomrule
    \end{tabular}
    \caption{The overhead panel, $196$ problems, both arms scored on
    $\norm{g_k} \leq \epsilon$, ordered by the product ratio. Totals are over every problem,
    solved or not, with no mean taken over survivors. The first-order arm evaluates no Hessian on
    any problem, so the last column has no first-arm counterpart: it is the whole price of the
    certificate.}
    \label{tab: so overhead}
\end{table}

\begin{figure}[htbp]
    \centering
    \includegraphics[width=\textwidth]{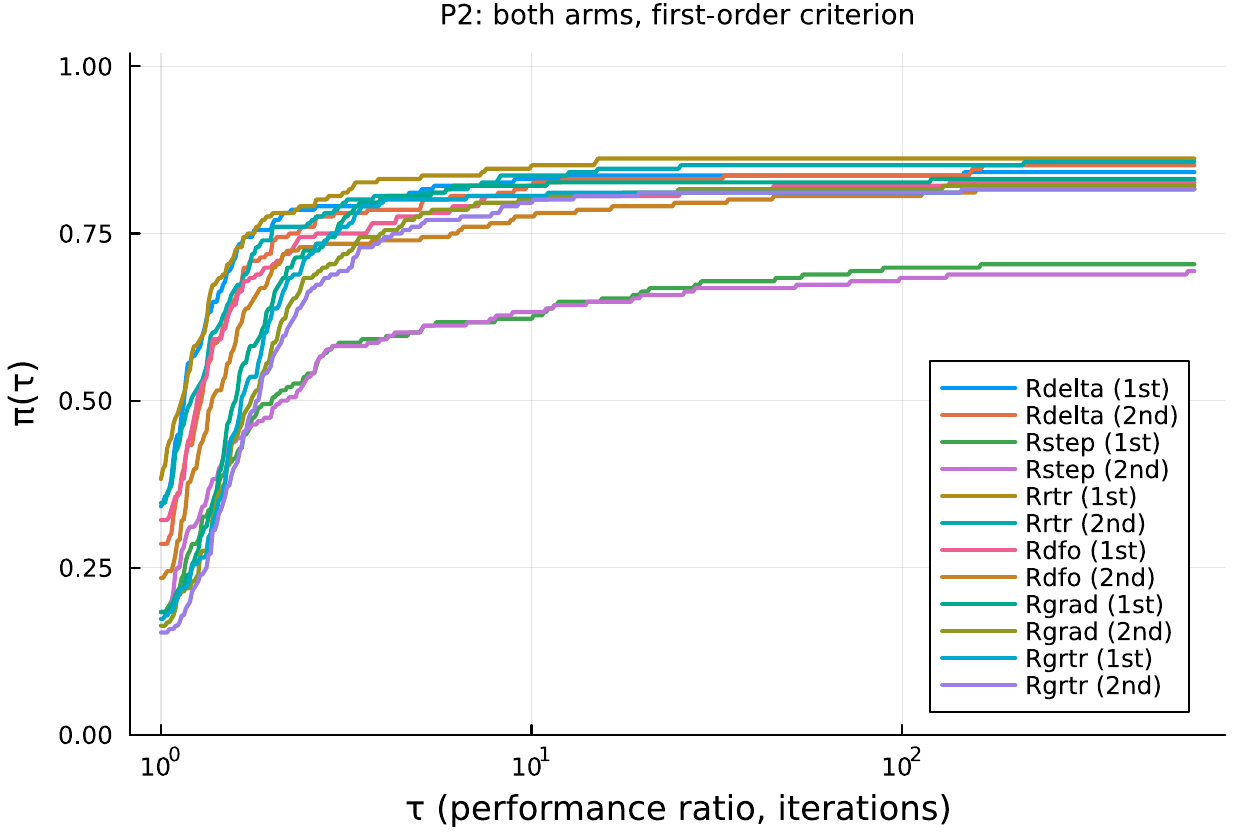}
    \caption{The overhead panel in iterations, both arms of all six mechanisms on the $196$
    problems, scored on the first-order criterion. The two curves of a mechanism run together
    across the whole range of $\tau$, so the reliability reading of Table~\ref{tab: so overhead}
    holds at every budget and not only at the asymptote.}
    \label{fig: so overhead iters}
\end{figure}

\begin{figure}[htbp]
    \centering
    \includegraphics[width=\textwidth]{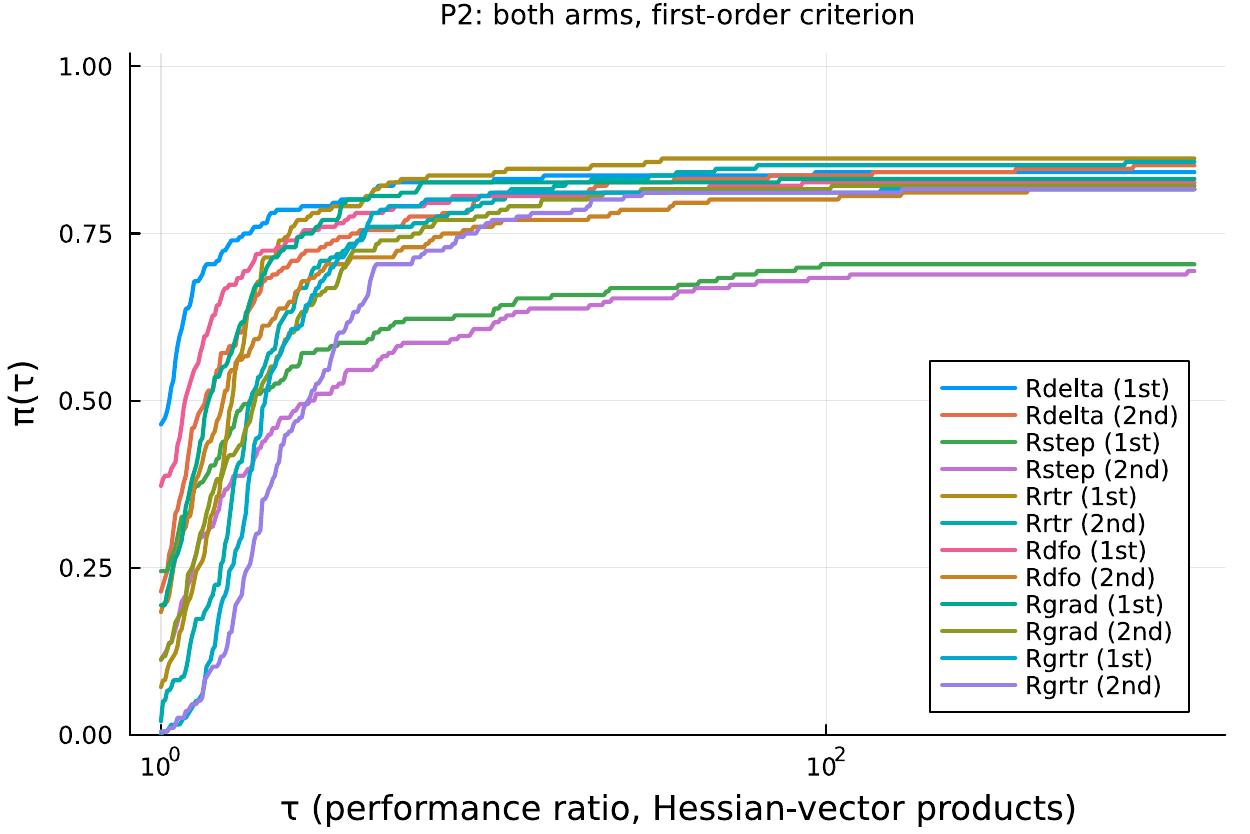}
    \caption{The same panel in Hessian-vector products, where the two arms of a mechanism
    separate by the ratio of Table~\ref{tab: so overhead}. The two gradient-scaled arms stay
    together and the other four do not, which is the anchor axis appearing as a cost.}
    \label{fig: so overhead prods}
\end{figure}

\subsection{What the certificate buys}
\label{subsec: so payoff}

The certificate panel asks the converse question on the $185$ problems with $n \leq 200$, where
$\lambda_k^-$ is formed and a certificate the test issues therefore holds. Each arm is scored on
what it can achieve: the first-order arm on reaching a critical point, the second-order arm on
returning a point certified to satisfy the second-order necessary conditions.

The two are within three problems of each other for every mechanism, and
Table~\ref{tab: so certificate} gives the exchange. $\Rdelta$ certifies $158$ second-order points
where its first-order arm reaches $155$ critical points of any kind, and is the only mechanism
ahead on the trade. $\Rgrtr$ draws at $150$. The remaining four give up two problems each.

That is the reading of this section. On this problem set the second-order arm returns a certified
point on very nearly the same problems on which the first-order arm returns an uncertified one,
and the exchange costs at most two problems out of $185$. The certificate is bought with Hessian
evaluations and not with reliability.

The ordering of the mechanisms is that of Table~\ref{tab: comparison} with one exchange at the
top. $\Rdelta$ and $\Rrtr$ lead, $\Rdfo$, $\Rgrad$ and $\Rgrtr$ sit six to eight problems behind,
and $\Rstep$ is last by twenty-four. The step anchor is as poor a carrier of a second-order
certificate as it is of a first-order one, and for the reason Section~\ref{sec: cost} gives: it
rebuilds the radius from a step the model has just chosen to keep short. Figure~\ref{fig: so arms}
compares the six second-order arms among themselves and reproduces that ordering.

\begin{table}[htbp]
    \centering
    \small
    \begin{tabular}{@{}lrrr@{}}
        \toprule
        Mechanism & Critical points, 1st arm & Certified, 2nd arm & Exchange \\
        \midrule
        $\Rdelta$ & $155$ & $158$ & $+3$ \\
        $\Rrtr$   & $159$ & $157$ & $-2$ \\
        $\Rdfo$   & $153$ & $151$ & $-2$ \\
        $\Rgrad$  & $153$ & $151$ & $-2$ \\
        $\Rgrtr$  & $150$ & $150$ & $0$  \\
        $\Rstep$  & $128$ & $126$ & $-2$ \\
        \bottomrule
    \end{tabular}
    \caption{The certificate panel, over the $185$ problems with $n \leq 200$. The first column
    counts the runs of the first-order arm that end in \texttt{first\_order}, which is what
    Section~\ref{subsec: stopping} calls solved for that arm. The second counts the runs of the
    second-order arm whose returned point meets the second-order criterion when it is applied
    afterwards. The two columns are not the same test, and that is the point: each arm is scored
    on the strongest statement it is equipped to make.}
    \label{tab: so certificate}
\end{table}

\begin{figure}[htbp]
    \centering
    \includegraphics[width=\textwidth]{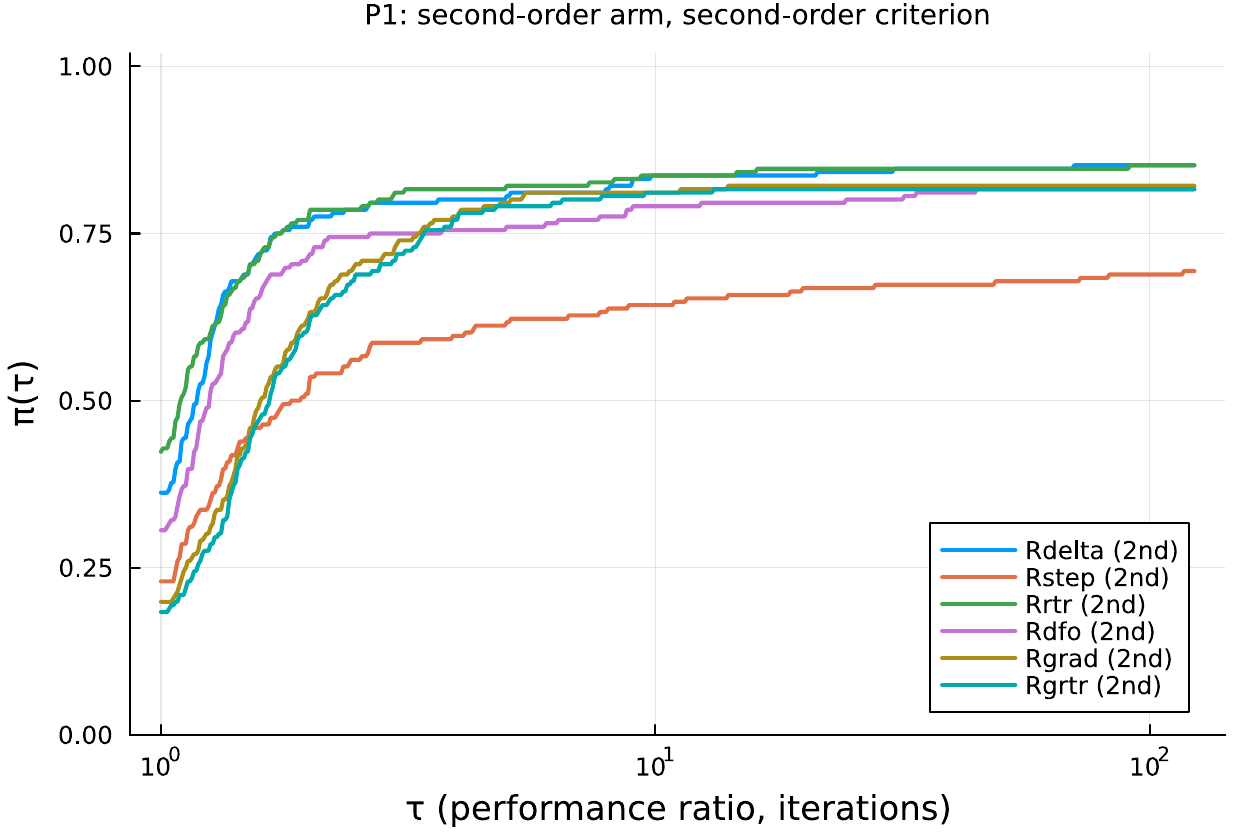}
    \caption{The six second-order arms compared among themselves in iterations, on the
    second-order criterion. $\Rstep$ is the low curve. This profile is drawn on all $196$
    problems, so it is read for the ordering of the arms and not for the counts, which
    Section~\ref{subsec: so limits} takes off the $185$ instead.}
    \label{fig: so arms}
\end{figure}

\subsection{The certificate panel in full}
\label{subsec: so p3}

Twelve configurations over the $185$ problems with $n \leq 200$, scored on the second-order
criterion. This panel forms the leftmost eigenvalue rather than estimating it, so a certificate
it issues holds. Appendix~\ref{app: so tables} carries the totals.

\paragraph{The ordering by anchor survives.}
Second-order over first-order Hessian-vector products, formed from
Table~\ref{tab: so cost p3}, run $0.90$ for $\Rgrad$, $0.93$ for $\Rgrtr$, $1.00$ for $\Rstep$,
$1.64$ for $\Rdelta$, $1.79$ for $\Rrtr$ and $2.29$ for $\Rdfo$. That reproduces the ordering of
Table~\ref{tab: so overhead} mechanism for mechanism, on a different problem set and against a
different criterion. Every ratio is lower here, so the eleven excluded problems carry the
overhead without deciding who pays it.

\paragraph{Two mechanisms pay in iterations rather than in products.}
$\Rgrad$ and $\Rgrtr$ spend fewer products with the curvature test than without it. $\Rgrad$
converges later, $76\,181$ objective evaluations against $46\,598$, while its products per
objective evaluation fall from $13.3$ to $7.3$. Anchoring to $\tau_k$ rather than to
$\norm{g_k}$ shrinks the radius wherever the curvature binds, and a smaller radius stops
truncated conjugate gradients sooner.

\paragraph{The Hessian column is the price.}
It has no first-arm counterpart, and below $n = 200$ it equals the gradient column exactly, so
the certificate costs one Hessian per gradient whatever the rule. A rule taking many short steps
therefore pays most. On this panel $\Rstep$ evaluates $397\,158$ Hessians against $44\,264$ for
$\Rrtr$, while certifying $126$ points against $157$.

\begin{figure}[H]
    \centering
    \begin{subfigure}[b]{0.49\textwidth}
        \includegraphics[width=\linewidth]{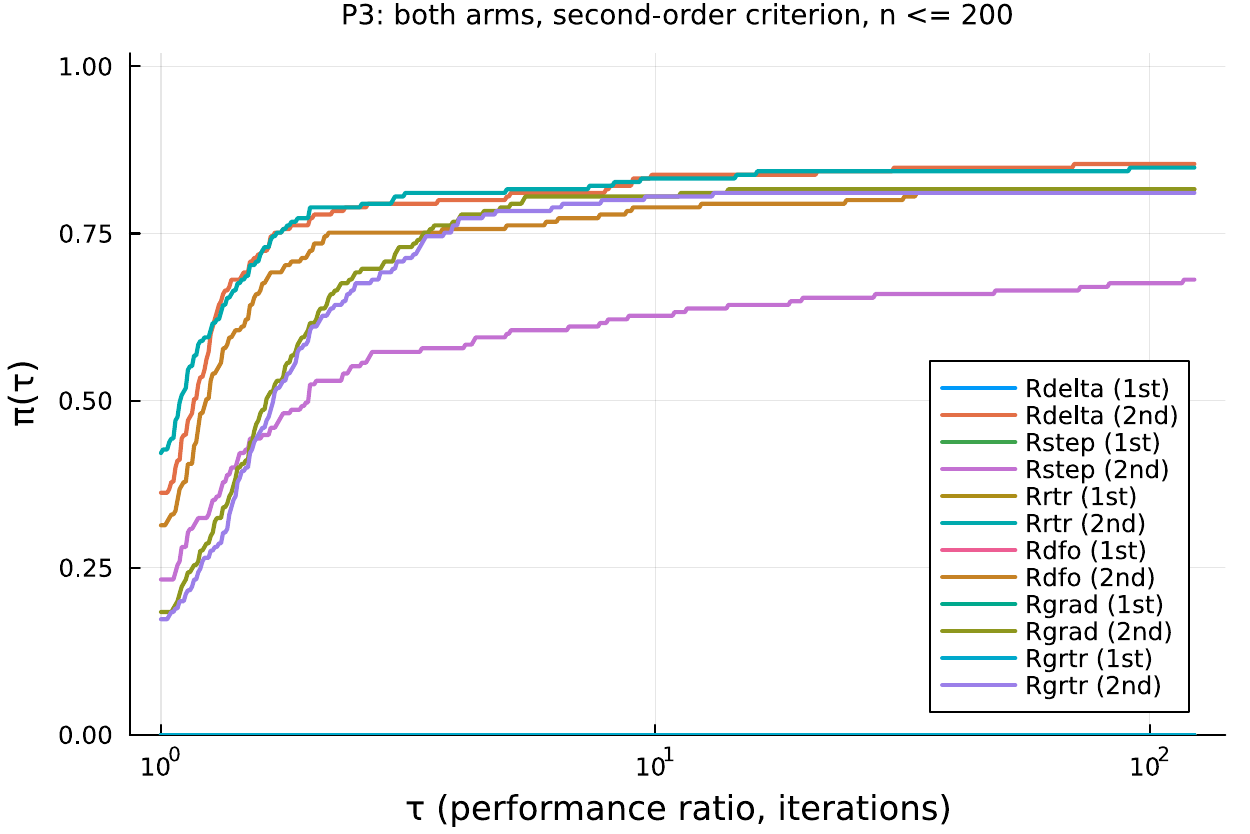}
        \caption{iterations}
    \end{subfigure}\hfill
    \begin{subfigure}[b]{0.49\textwidth}
        \includegraphics[width=\linewidth]{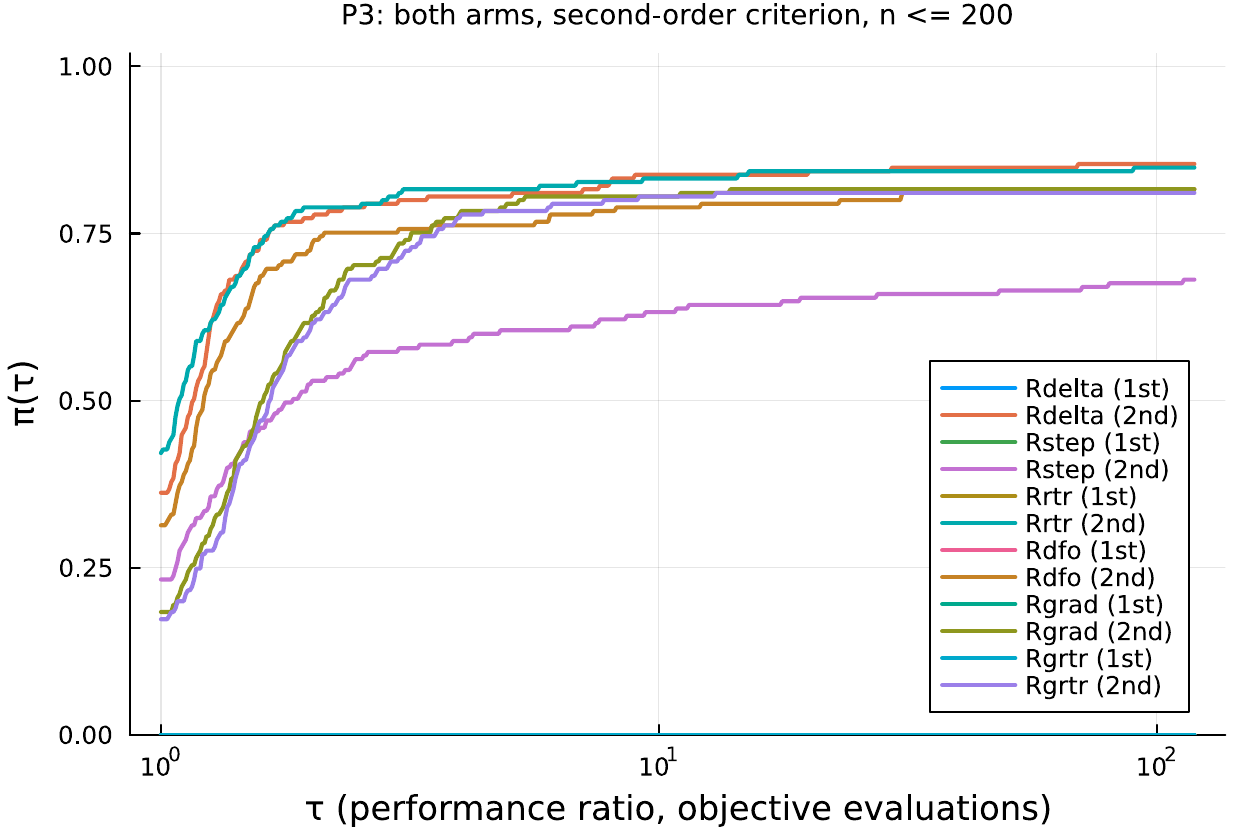}
        \caption{objective evaluations}
    \end{subfigure}

    \medskip
    \begin{subfigure}[b]{0.49\textwidth}
        \includegraphics[width=\linewidth]{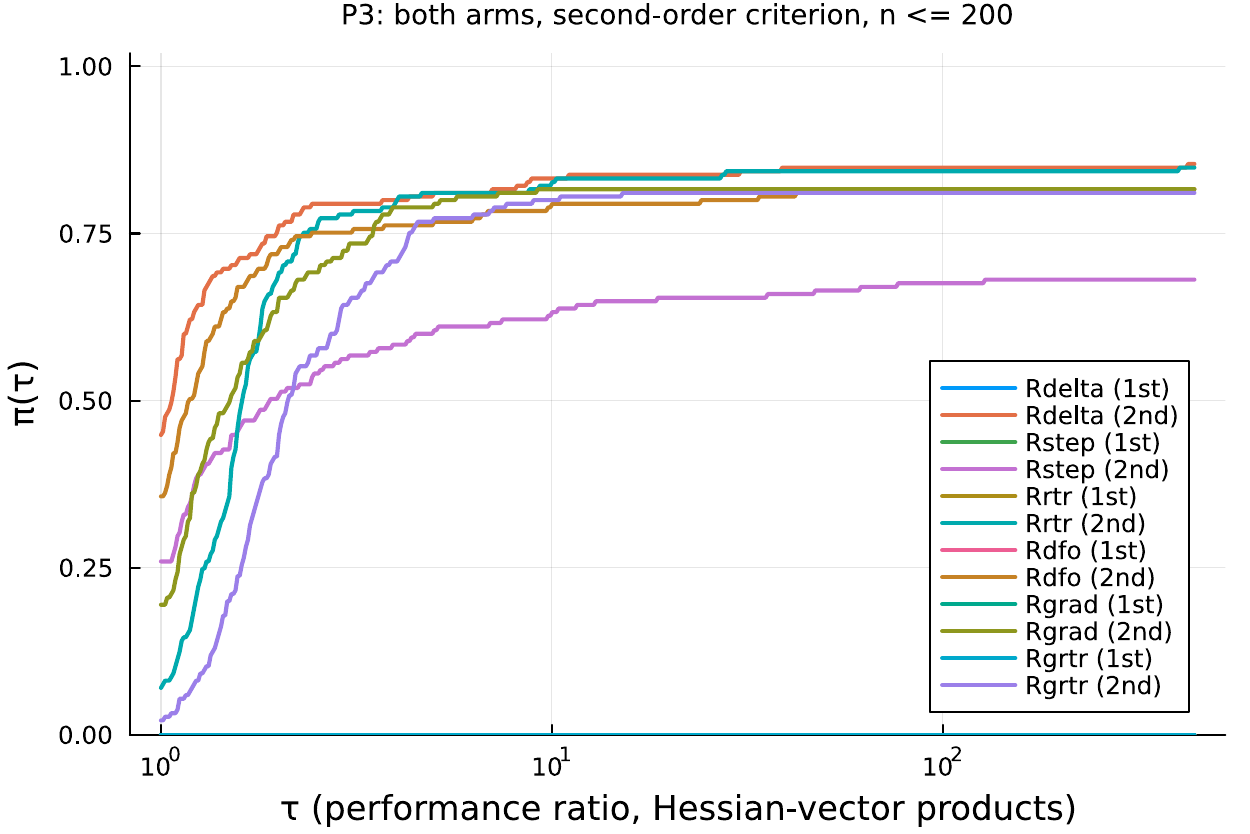}
        \caption{Hessian-vector products}
    \end{subfigure}\hfill
    \begin{subfigure}[b]{0.49\textwidth}
        \includegraphics[width=\linewidth]{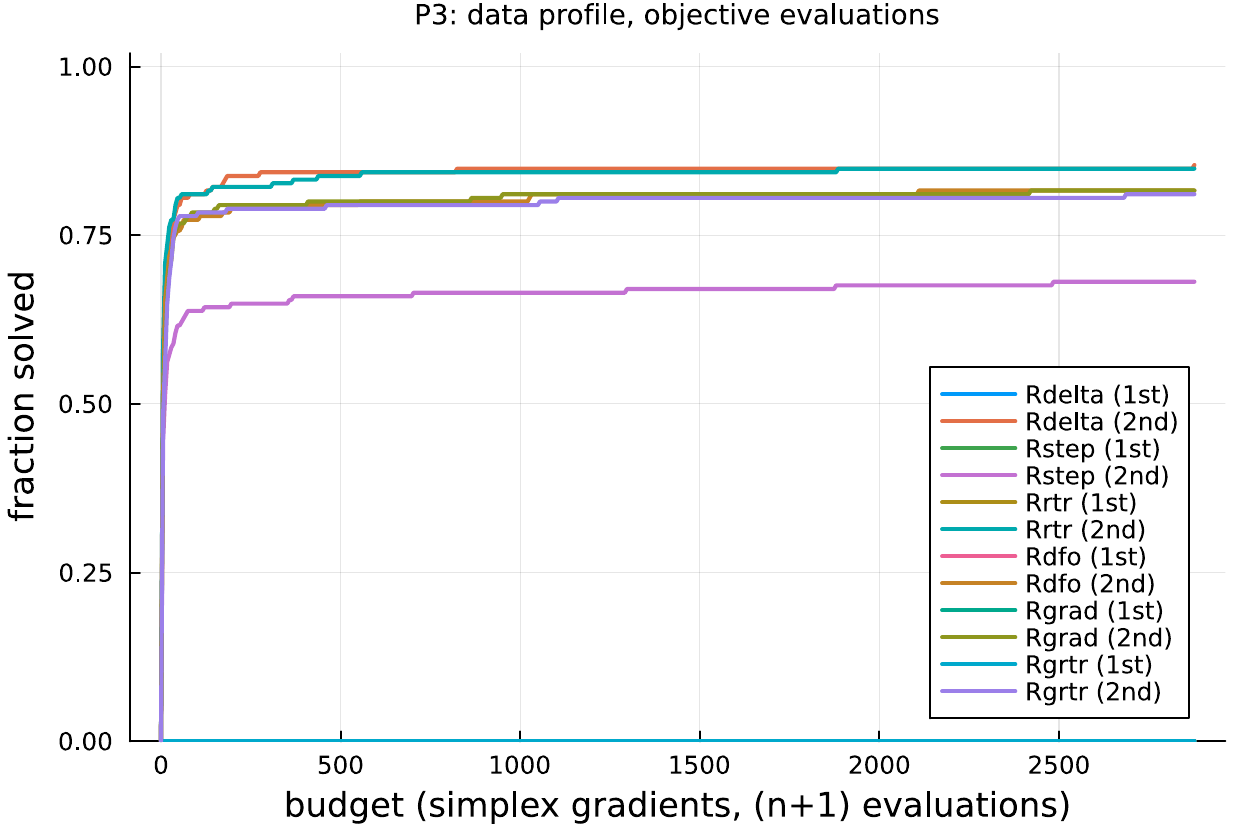}
        \caption{data profile}
    \end{subfigure}
    \caption{Panel P3. Both arms of all six mechanisms on the $185$ problems with $n \leq 200$,
    scored on the second-order criterion. The six first-order curves lie on zero at every budget.
    Section~\ref{subsec: so limits} says why that is not a measurement. Gradient and Hessian
    evaluations are omitted here and tabulated in Appendix~\ref{app: so tables}.}
    \label{fig: p3}
\end{figure}

\subsection{What these panels do not show}
\label{subsec: so limits}

\paragraph{The first-order arms are not scored on the second-order criterion.}
Every one of those six rows reads $0/185$, in the listing of
Section~\ref{subsec: so p3} and as the flat curves of Figure~\ref{fig: p3}. That number
records the instrument and not the mechanism, and no count in
Table~\ref{tab: so certificate} is taken from it. The first-order arm
performs zero Hessian evaluations on all $185$ problems, as the Hessian column of
Table~\ref{tab: so overhead} states, and at $n \leq 200$ the leftmost eigenvalue is obtained from
that evaluation. The criterion was therefore never evaluated for those runs. The zero says the
certificate was not computed, not that the returned points fail it, and the fraction of
first-order points that already satisfy the second-order conditions is measured nowhere in this
\thisdoc{}. One Hessian evaluation per run would measure it.

Section~\ref{subsec: stopping} states the convention that forbids the same reading with the arms
exchanged: scoring on \texttt{first\_order} alone would mark every successful $\tau$-anchored run
as unsolved and give the configurations that do best a reliability of zero. The objection is the
same one and it is why those rows are omitted.

\paragraph{Eleven problems are excluded from the certificate panel.}
Eleven of the $196$ have $n > 200$, where the eigenpoint wrapper estimates the leftmost
eigenvalue rather than forming it. The estimate is one-sided by
Section~\ref{app: measure}, so it understates the negative curvature available and the
second-order test is optimistic there. Scored over all $196$ the second-order arms report $167$,
$167$, $160$, $161$, $160$ and $136$ certified points for $\Rdelta$, $\Rrtr$, $\Rdfo$, $\Rgrad$,
$\Rgrtr$ and $\Rstep$, between nine and ten more than over the $185$, so nine or ten of those
eleven are certified on an estimate. Table~\ref{tab: so certificate} is read off the $185$ for
that reason, and Figure~\ref{fig: so arms}, which is drawn on the $196$, is read for the ordering
of the arms and not for the counts.

The archive exports eighteen profiles, three panels in six cost measures. Nine are printed here:
three selected from P1 and P2, and the whole of P3 in Figure~\ref{fig: p3}. Within a panel the
objective- and gradient-evaluation profiles have the same intercept and the same asymptote as the
iteration profile, a trust-region iteration evaluating the objective at exactly one trial point,
and the data profiles reorder nothing, so P3 is printed entire for the certificate it carries and
not because its six measures separate. Its six flat first-order curves are the artefact of the
first paragraph above, drawn.

\paragraph{Six runs raised.}
Six runs of the $4\,704$ ended in \texttt{exception}, all of them in a second-order arm, and the
archive does not record which problems or which exception. They count as failures throughout, so
no rate above is inflated by them, and six runs cannot move a solve rate by more than one
problem. Two further hazards were declared with the campaign and both are nil: no run began at a
point already satisfying the gradient test, and no run's own status disagreed with the post hoc
criterion.

\paragraph{What is measured is average cost on one problem set.}
These are aggregate results on $196$ CUTEst problems at one pair of tolerances. Nothing here
bears on the sharpness of a worst-case bound, and nothing here settles the claim as
Table~\ref{tab: claims} states it, which asserts a repair \emph{on the same problems and from the
same starting points} as the failures of Claims~\ref{claim: cauchy} and~\ref{claim: bad limits}.
Those two configurations are still not rerun under $\omega_k = \tau_k$. What this section adds to
Claim~\ref{claim: order} is the aggregate half: the repair, where a mechanism admits it at all,
is bought with Hessian evaluations and not with reliability.

\section{\texorpdfstring{What is settled, and what this design cannot show}{What is settled, and what this design cannot show}}

\subsection{\texorpdfstring{The verdicts, collected}{The verdicts, collected}}
\label{sec: verdict}

Each claim is answered where the evidence for it is reported, so the reader meets a
verdict at the point it is earned rather than at the end. Table~\ref{tab: verdicts} is the
index of those answers and adds nothing to them. Six of the nine are settled outright, one on both sides but in
two different currencies, one in half, and one is bounded by a case the second-order anchor does
not repair.

\begin{table}[htbp]
    \centering
    \small
    \begin{tabular}{@{}r p{0.34\textwidth} p{0.32\textwidth} l@{}}
        \toprule
        & Claim & Verdict & Stated in \\
        \midrule
        \ref{claim: asymptotics} & the asymptotics of the radius & Settled & \S\ref{subsec: P1} \\
        \ref{claim: rate of decay} & harmonic decay of the anchored radius & Settled & \S\ref{subsec: P1} \\
        \ref{claim: zeta} & the activity threshold & Settled, in both regimes & \S\ref{sec: threshold} \\
        \ref{claim: rate} & the local rate after inactivity & First half stands, second weakly & \S\ref{sec: inactivity} \\
        \ref{claim: cauchy} & the Cauchy point & Settled on the designed problems & \S\ref{sec: model vs radius rule} \\
        \ref{claim: bad limits} & a limit that is not a minimiser & Settled & \S\ref{sec: model vs radius rule} \\
        \ref{claim: escape} & escape from a flat region & Settled, by measurement and by proof & \S\ref{subsec: flat} \\
        \ref{claim: order} & the second-order anchor & Partially settled, and bounded & \S\ref{subsec: flat}, \S\ref{sec: second order cutest} \\
        \ref{claim: no uniform cap} & no cap is large enough & Settled & \S\ref{sec: threshold} \\
        \bottomrule
    \end{tabular}
    \caption{The nine claims of Table~\ref{tab: claims}, the verdict each receives, and the
section that states it.}
    \label{tab: verdicts}
\end{table}

\subsection{\texorpdfstring{What this design cannot show}{What this design cannot show}}
\label{sec: threats}

A caveat that attaches to one measurement is made where that measurement is, and is only
named here. What follows in full is what belongs to the design as a whole and has no other
home.

\paragraph{One failure mode is machine-dependent.}
A run stopped by the wall-clock budget is counted as unsolved on the same footing as one
stopped by the iteration budget, and Section~\ref{subsec: settings} names the machine that
decides the first. The two are separate outcomes in the run record, \texttt{max\_time} and
\texttt{max\_iter}, but no count below separates them, so a reader cannot see how much of any
solve count the time cap decided.

\paragraph{The collection is not a sample.}
The selection is every unconstrained CUTEst problem with $2 \leq n \leq 1000$, which is a
census of a collection and not a draw from a population, and the grid of
Section~\ref{subsec: interaction} has no replication, so none of the percentages carries a
sampling error and none supports an interval. The correlation of $-0.84$ between tail activity
and efficiency in Section~\ref{subsec: like for like} is a description of eight points, not an
estimate. That is the weaker half of the point. CUTEst is not itself a draw from any population
of problems anyone solves; it is a set accumulated over decades for reasons of its own. So these
numbers describe these $196$ problems exactly and nothing carries them to a problem outside the
set, the recommendation of Section~\ref{sec: conclusion} included, which is made on this
evidence and no more. The upper bound on $n$ is a budget decision, and a rule that separates only
on large problems would be invisible here. What differs between two runs of one configuration is
not the arithmetic, which is deterministic, but whether the $120$ seconds of
Section~\ref{subsec: settings} end a run, which the machine decides.

\paragraph{One split is by outcome and not by design.}
The rows of Table~\ref{tab: mu theta} are separated by what the runs did rather than by
what was set, since a run whose cap binds is a run that wanted a larger radius, which is a
property of the problem as much as of the cap. Nothing in this \thisdoc{} conditions on that split
without saying so.

\paragraph{Five caveats are made where their measurements are.}
They are collected here by pointer and not restated. The Cauchy diagnostic separates
nothing at any cap, in Section~\ref{subsec: mu sweep}. The order estimator admits fits that no
rate can produce, thirty-nine below $1$ and twenty-two above $2$ of $431$, in
Section~\ref{subsec: order rules}; neither count changes the reading of Claim~\ref{claim: rate},
whose evidence is the overlap of the eight distributions rather than the value of any median. The
factor axis is not isolated, the two adaptive rules running at constants the other six do not
share, in Section~\ref{subsec: like for like}. The initial multiplier moves with the cap at the
two smallest caps, which Update~\ref{update: rgrad generic} forces through
$\mu_0 \leq \overline{\mu}$, in Section~\ref{subsec: mu sweep}. And the reading of
$\lambda^*_{\min} < 10^{-2}$ off eleven problems runs backwards through the theory on two
unchecked provisos, in the same section. Section~\ref{subsec: so limits} does the same
for the second-order campaign.

\paragraph{Where Parts~I and~II might have stated the same rule differently.}
The introduction promises the places where the two theoretical parts state a mechanism
differently in a way an experiment can see. On the criticality measure there is no such place,
and it is worth recording once that there is not. \PartII{} defines
$\tau_k = \max\{\norm{g_k},\, -\lambda_k^-\}$, \PartI{} does not use the symbol for a
criticality measure at all, and Appendix~\ref{app: measure} states the reading this study uses.
Every result of Section~\ref{subsec: flat} that turns on the second-order anchor is in that
reading, and so is the software.

\section{Conclusion}
\label{sec: conclusion}

The radius update rule is the least standardised component of a trust-region method, and this
study does not standardise it. What it supplies is the comparison a reader looking for a reason
to prefer one rule could not find, on one iteration in which the rule is the only thing that
changes.

The rule is worth less than the model and more than nothing. Crossing the mechanisms with four
model Hessians puts $96.5$ per cent of the variation on the model and $1.6$ on the rule. That
margin is small and the aggregate still resolves it, the eight configurations at matched
constants spanning $138$ to $169$ problems of $196$. Where the rule matters is not where
intuition puts it. The four rules separate by $27$ problems under the exact Hessian, by $11$
under L-BFGS, by $7$ under SR1 and by $6$ under the identity. The better the model, the more
the radius decides, because no radius repairs a wrong direction.

The rule governs the length of the step, so it is visible exactly where the radius binds. That
is why the activity of the constraint is the instrument this study runs on, and why mechanisms
whose first-order behaviour is indistinguishable separate under it. Conditional on the
constraint releasing, the observed order belongs to the model and the subproblem solver, and the
rule does not enter it.

The side conditions of \PartII{} are not technical convenience to be tuned away. Two results of
one shape say so. For every cap, however large, there is a minimiser of the cardinal sine in
whose basin the constraint stays active at every large iteration. For every anchoring constant,
however large, there is a flatness on which escape takes arbitrarily long. Neither is escaped by
choosing a larger constant, and both hold with the exact Hessian.

For a reader who has to choose within the criticality-anchored family, uncapped
$\Rgrad$ is the robust default on this collection, and
Section~\ref{sec: threats} says why that is not a statement about problems outside it.
Raising the cap through
$0.1$, $1$, $10$ and $100$ and then removing it takes reliability from $102$ to $161$ problems
of $196$, monotonically and with no two profiles crossing. Across the anchors the
recommendation is narrower than that. On the eight configurations of
Table~\ref{tab: comparison} the retrospective rule is ahead of uncapped $\Rgrad$ on every
column, solving $169$ problems against $163$ at an efficiency of $0.449$ against $0.184$, and
$\Rdelta$ is ahead of it on both. What the cap sweep establishes is that no cap improves on
removing the cap, not that the gradient anchor is the best anchor. The choice carries a tension. The cap
is what supplies the bound the asymptotic theory assumes, and removing it is what the aggregate
rewards, so the configuration the collection prefers is the one the theory covers least. That
tension is the subject of this study and not a defect of it.

A reader choosing a rule is choosing how much the radius costs where it binds. A reader choosing
a model is choosing whether it binds at all.

\section*{Acknowledgements}
This research was funded by a Discovery Grant of the Natural Sciences and Engineering
Research Council of Canada (NSERC), held by the second author.

\bibliography{References}

\appendix

\section{The algorithmic frame}
\label{app: frame}

\subsection{The basic trust-region algorithm}
\label{app: btr}

\begin{algorithm}[!tbp]
    \small
    \SetAlgoLined
    \DontPrintSemicolon

    \textbf{Step 0: Initialisation.} Choose $\eta \geq 0$. Set $x_0 \in \R^p$, $\Delta_0 \in (0, \infty)$. Compute $f(x_0)$ and set $k = 0$. \; ~\\

    \textbf{Step 1: Model definition and step computation.} Define a local model $m_k$ of $f$ in the neighbourhood $\calB_k$. \; ~\\

    \textbf{Step 2: Model minimisation.} Compute a step $s_k$ that sufficiently reduces $m_k$ with $x_k + s_k \in \calB_k$, and define $x_k^{\mathrm{cand}} = x_k + s_k$. \; ~\\

    \textbf{Step 3: Acceptance of the trial point.} Compute $f(x_k^{\mathrm{cand}})$ and define
    \begin{equation}
        \rho_k = \frac{f(x_k) - f(x_k^{\mathrm{cand}})}{m_k(0) - m_k(s_k)}.
        \label{eqn:definition rho}
    \end{equation}
    If $\rho_k \geq \eta$, set $x_{k+1} = x_k + s_k$; otherwise $x_{k+1} = x_k$. \; ~\\

    \textbf{Step 4: Trust-region radius update.} Compute $\Delta_{k+1}$ according to some update rule. \; ~\\

    \textbf{Step 5: Next iteration.} Increment $k$ by 1 and repeat \textbf{Steps 1 to 5} while some stopping criterion is not met. \;

    \caption{Basic trust-region algorithm (BTR).}   
    \label{alg:generic BTR}
\end{algorithm}

The neighbourhood of Step~1 is the trust region
\begin{equation}
    \calB_k = \{ x \in \Rp : \norm{x - x_k} \leq \Delta_k \},
    \label{eqn: trust region}
\end{equation}
$\norm{\cdot}$ denoting the Euclidean norm throughout. Every update of
Appendix~\ref{app: update mechanisms} supplies Step~4 of Algorithm~\ref{alg:generic BTR} and
changes nothing else.

\subsection{Truncated conjugate gradients}
\label{app: steihaug}

\begin{algorithm}[!tbp]
\caption{Truncated conjugate gradients, \texttt{SteihaugCG}}
\label{alg: SteihaugCG}
\DontPrintSemicolon
\KwIn{$g$; operator $B$, through products $Bv$ only; $\Delta$; forcing constants
      $\chi$, $\theta$; cap $j_{\max}$}
\KwOut{$s$ and the flag \texttt{active}}
\BlankLine
\lIf{$\norm{g} = 0$}{\KwRet $(0, \textnormal{false})$}
$s \leftarrow 0$; $r \leftarrow -g$; $d \leftarrow r$;
$\text{tol} \leftarrow \min\{\chi, \norm{g}^{\theta}\}\,\norm{g}$
  \tcp*{forcing condition}
\For{$j = 1, \dots, j_{\max}$}{
    $\kappa \leftarrow d^\top B d$\;
    \If(\tcp*[f]{negative curvature}){$\kappa \leq 0$}{
        $\tau \leftarrow$ positive root of $\norm{s + \tau d} = \Delta$;
        \KwRet $(s + \tau d,\ \textnormal{true})$\;
    }
    $\alpha \leftarrow r^\top r / \kappa$\;
    \If(\tcp*[f]{the iterate leaves the region}){$\norm{s + \alpha d} \geq \Delta$}{
        $\tau \leftarrow$ positive root of $\norm{s + \tau d} = \Delta$;
        \KwRet $(s + \tau d,\ \textnormal{true})$\;
    }
    $s \leftarrow s + \alpha d$; $r_{+} \leftarrow r - \alpha B d$\;
    \lIf{$\norm{r_{+}} < \text{tol}$}{\KwRet $(s,\ \textnormal{false})$}
    $\beta \leftarrow r_{+}^\top r_{+} / r^\top r$; $d \leftarrow r_{+} + \beta d$; $r \leftarrow r_{+}$\;
}
\KwRet $(s, \textnormal{false})$\;
\end{algorithm}

\section{The quadratic decrease condition under inexact subproblem solution}
\label{app: mqd}

Condition~\eqref{eqn: mqd} is stated for an exact minimiser of the model. It survives a truncated
conjugate-gradient solver, and with room to spare. Throughout, $H_k$ is the model Hessian that
Algorithm~\ref{alg: SteihaugCG} sees as the operator $B$, and $p_j$ is the direction that algorithm
calls $d$ at its $j$-th inner iteration.

The argument rests on one property of a subspace minimiser. If $s_k$ minimises $m_k$ over a
subspace $\calK \subseteq \Rp$, then $t \mapsto m_k(s_k + tv)$ is minimised at $t = 0$ for every
$v \in \calK$, so
\begin{equation}
    \langle H_k s_k + g_k,\, v\rangle = \langle \nabla m_k(s_k),\, v\rangle = 0
    \qquad \text{for all } v \in \calK ,
    \label{eqn: galerkin subspace}
\end{equation}
the residual being orthogonal to the search space. Taking $v = s_k \in \calK$ gives
\begin{equation}
    \langle H_k s_k + g_k,\, s_k\rangle = 0 ,
    \label{eqn: galerkin}
\end{equation}
which is subspace optimality contracted with the step itself.

\subsection{Conjugate gradients on a positive definite Hessian}
\label{app: mqd cg}

Write $r_j \defeq -\nabla m_k(s_j) = -(g_k + H_k s_j)$ for the residual at the $j$-th inner
iterate and $p_j$ for the $j$-th direction, with $s_0 = 0$ and $p_0 = r_0 = -g_k$. The
method of \citet{HestenesStiefel1952} sets
\begin{equation}
    \alpha_j = \frac{\norm{r_j}^2}{p_j^\top H_k p_j} ,
    \quad
    s_{j+1} = s_j + \alpha_j p_j ,
    \quad
    r_{j+1} = r_j - \alpha_j H_k p_j ,
    \quad
    p_{j+1} = r_{j+1} + \frac{\norm{r_{j+1}}^2}{\norm{r_j}^2}\, p_j ,
    \label{eqn: cg recurrences}
\end{equation}
and produces $H_k$-conjugate directions spanning the Krylov space
\begin{equation}
    \calK_j \defeq \operatorname{span}\{g_k,\, H_kg_k,\, \dots,\, H_k^{\,j-1}g_k\}
    = \operatorname{span}\{r_0,\dots,r_{j-1}\}
    = \operatorname{span}\{p_0,\dots,p_{j-1}\} .
    \label{eqn: krylov}
\end{equation}

\begin{lemma}[Conjugate gradients return the subspace minimiser]
\label{lem: cg minimiser}
    If $H_k$ is positive definite on $\calK_j$, then
    $s_j = \arg\min\{ m_k(s) : s \in \calK_j \}$, so $r_j \perp \calK_j$
    and~\eqref{eqn: galerkin} holds at $s_j$.
\end{lemma}

\begin{proof}
    Write $s = \sum_{i<j}\lambda_i p_i$ for a general element of $\calK_j$. Conjugacy removes the
    cross terms from the quadratic part, so
    \begin{equation}
        m_k(s) - f(x_k)
          = \sum_{i<j}\Bigl[ \lambda_i\, g_k^\top p_i
                             + \tfrac12 \lambda_i^2\, p_i^\top H_k p_i \Bigr] ,
        \label{eqn: separates}
    \end{equation}
    a sum of independent scalar quadratics, each strictly convex since
    $p_i^\top H_k p_i > 0$. The $i$-th is minimised at
    $\lambda_i^* = -g_k^\top p_i / (p_i^\top H_k p_i)$, and $\lambda_i^* = \alpha_i$ because
    $-g_k^\top p_i = r_i^\top p_i = \norm{r_i}^2$. The first equality here uses
    $r_i = -g_k - H_k s_i$ together with $p_i^\top H_k s_i = 0$, which is conjugacy, and the
    second uses $p_i = r_i + \beta_i p_{i-1}$ together with $r_i \perp p_{i-1} \in \calK_i$. The
    inner iterate $s_j = \sum_{i<j}\alpha_i p_i$ therefore minimises~\eqref{eqn: separates} over
    $\calK_j$, and \eqref{eqn: galerkin subspace} gives $r_j \perp \calK_j$.
\end{proof}

\begin{remark}[Conjugacy is what does the work]
\label{rem: conjugacy}
    Conjugacy is what makes the line searches independent, so that a sequence of one-dimensional
    minimisations solves the $j$-dimensional problem. Without it, minimising along $p_j$ would
    spoil the optimality already reached along $p_0,\dots,p_{j-1}$. An equivalent route to the
    same conclusion runs through the mutual orthogonality of the residuals, $r_j^\top r_i = 0$ for
    $i < j$, which together with~\eqref{eqn: krylov} puts $r_j \perp \calK_j$ by construction.
\end{remark}

Lemma~\ref{lem: cg minimiser} applies whenever the inner iteration stops of its own accord, and in
particular on the residual test $\norm{r_j} \leq \min\{\chi, \norm{g_k}^{\theta}\}\norm{g_k}$ of
Algorithm~\ref{alg: SteihaugCG}. The iterate $s_j$ is then the exact subspace minimiser whatever $j$
one stops at.

\subsection{The boundary case, where the identity becomes an inequality}
\label{app: mqd boundary}

Truncation breaks the hypothesis of Lemma~\ref{lem: cg minimiser}, the returned step being a point
on the trust-region boundary along the current direction rather than an inner iterate. The
equality is lost. What survives is an inequality of the right sign.

\begin{lemma}[The truncated step]
\label{lem: truncated step}
    Let $s_j$ be the last inner iterate inside the trust region, $p_j$ the current direction, and
    let the returned step be $s_k = s_j + \tau p_j$ with $\tau > 0$ chosen so that
    $\norm{s_k} = \Delta_k$. Assume either
    \begin{enumerate}[label=(\alph*),leftmargin=2.2em]
        \item $p_j^\top H_k p_j > 0$ and $\tau \leq \alpha_j$, the truncation being triggered
              because the full step leaves the region, or
        \item $p_j^\top H_k p_j \leq 0$, the truncation being triggered by negative curvature.
    \end{enumerate}
    Then
    \begin{equation}
        \langle H_k s_k + g_k,\, s_k \rangle \;\leq\; 0 .
        \label{eqn: galerkin ineq}
    \end{equation}
\end{lemma}

\begin{proof}
    The proof uses three properties of the iteration,
    \begin{equation}
        r_j \perp \calK_j \ni s_j ,
        \qquad
        p_j^\top H_k s_j = 0 ,
        \qquad
        r_j^\top p_j = \norm{r_j}^2 .
        \label{eqn: three facts}
    \end{equation}
    The first is Lemma~\ref{lem: cg minimiser}. The second holds because
    $s_j = \sum_{i<j}\alpha_i p_i$ and $p_j^\top H_k p_i = 0$ for $i < j$. The third follows from
    $p_j = r_j + \beta_j p_{j-1}$ and $r_j \perp p_{j-1}$, an element of $\calK_j$.

    Now $\nabla m_k(s_k) = g_k + H_k(s_j + \tau p_j) = -r_j + \tau H_k p_j$, so
    \begin{align}
        \langle H_k s_k + g_k,\, s_k\rangle
          &= \langle -r_j + \tau H_k p_j,\ s_j + \tau p_j\rangle \notag\\
          &= \underbrace{-r_j^\top s_j}_{=\,0}
             \;-\; \tau\, \underbrace{r_j^\top p_j}_{=\,\norm{r_j}^2}
             \;+\; \tau\, \underbrace{p_j^\top H_k s_j}_{=\,0}
             \;+\; \tau^2\, p_j^\top H_k p_j \notag\\
          &= \tau\Bigl( \tau\, p_j^\top H_k p_j - \norm{r_j}^2 \Bigr) .
        \label{eqn: boundary identity}
    \end{align}
    In case~(a), $\tau \leq \alpha_j = \norm{r_j}^2/(p_j^\top H_kp_j)$ gives
    $\tau\, p_j^\top H_kp_j \leq \norm{r_j}^2$, so the bracket is non-positive. In case~(b) the
    bracket is at most $-\norm{r_j}^2 < 0$ outright, the step being taken along a direction of
    non-positive curvature. Since $\tau > 0$, \eqref{eqn: galerkin ineq} follows.
\end{proof}

\begin{remark}[The hypothesis $\tau \leq \alpha_j$ needs no assuming]
\label{rem: tau below alpha}
    That hypothesis is the geometric content of truncation. The scalar $\alpha_j$ minimises $m_k$
    along $p_j$, and the iteration truncates precisely because $s_j + \alpha_jp_j$ lies outside
    the ball. Steihaug's monotonicity gives $s_j^\top p_j \geq 0$, so
    $t \mapsto \norm{s_j + t p_j}$ increases on $t \geq 0$ \citep[Thm.~2.1]{Stei83} and the
    boundary is met before the line minimiser, whence $\tau \leq \alpha_j$. The truncated step
    stops short of the line minimiser, and the directional derivative of $m_k$ along $p_j$ is
    still negative at $s_k$.
\end{remark}

\subsection{The condition under truncation}
\label{app: mqd proposition}

\begin{proposition}[Quadratic model decrease for truncated conjugate gradients]
\label{prop: mqd}
    Let $s_k$ be produced by Algorithm~\ref{alg: SteihaugCG} applied to $m_k$, whether the iteration
    stops in the interior or on the boundary. Then
    \begin{equation}
        m_k(0) - m_k(s_k) \;\geq\; \tfrac12\, s_k^\top H_k s_k .
        \label{eqn: half energy}
    \end{equation}
    If moreover $\lambda_{\min}(H_k) \geq \bar\delta > 0$, then~\eqref{eqn: mqd} holds with
    $\mykappa{mqd} = \tfrac12\bar\delta$.
\end{proposition}

\begin{proof}
    Lemma~\ref{lem: cg minimiser} in the interior case and Lemma~\ref{lem: truncated step} on the
    boundary both give $\langle H_ks_k + g_k,\, s_k\rangle \leq 0$, that is
    $-g_k^\top s_k \geq s_k^\top H_k s_k$. Hence
    \begin{equation*}
        m_k(0) - m_k(s_k)
          = -g_k^\top s_k - \tfrac12 s_k^\top H_k s_k
          \;\geq\; s_k^\top H_k s_k - \tfrac12 s_k^\top H_k s_k
          = \tfrac12 s_k^\top H_k s_k ,
    \end{equation*}
    which is~\eqref{eqn: half energy}, and $s_k^\top H_ks_k \geq \bar\delta\norm{s_k}^2$ gives the
    second assertion.
\end{proof}

\begin{remark}[What this replaces]
\label{rem: mqd replaces}
    The usual argument treats the boundary separately. It invokes Steihaug's monotonicity to place
    the unconstrained minimiser outside the region,
    $\norm{s_k} = \Delta_k \leq \norm{H_k^{-1}g_k} \leq \norm{g_k}/\bar\delta$, and then applies
    Cauchy decrease to obtain
    \begin{equation*}
        m_k(0) - m_k(s_k) \;\geq\; \frac{\mykappa{mdc}\,\bar\delta}{2}
            \min\Bigl\{\frac{\bar\delta}{\mykappa{umh}},\, 1\Bigr\}\norm{s_k}^2
        = \frac{\mykappa{mdc}\,\bar\delta^{\,2}}{2\,\mykappa{umh}}\norm{s_k}^2 ,
    \end{equation*}
    the minimum collapsing because $\bar\delta \leq \mykappa{umh}$ always.
    Proposition~\ref{prop: mqd} supersedes this in three ways. It is one case rather than two. It
    uses neither Steihaug's monotonicity, nor Cauchy decrease, nor the existence of $H_k^{-1}$.
    And since $\mykappa{mdc} \leq 1$ and $\bar\delta \leq \mykappa{umh}$, the constant
    $\bar\delta/2$ is at least as large as
    $\mykappa{mdc}\bar\delta^{\,2}/(2\mykappa{umh})$, usually by a wide margin.
\end{remark}

\begin{remark}[Two results of the same shape]
\label{rem: mqd prior}
    Two results in the literature bound the same quantity, and neither delivers what
    Proposition~\ref{prop: mqd} delivers here. \citet{Yuan2000} proves that the truncated step
    attains at least half the reduction of the global minimiser of the subproblem, for a convex
    quadratic. That bound is sharper as a statement about the solver and weaker as a route
    to~\eqref{eqn: mqd}, since it is stated against a quantity that only the exact solution
    supplies. \citet[Lem.~8]{CurtisRobinsonRoyerWright2021} reach an inequality of the shape
    of~\eqref{eqn: mqd} for an indefinite $H_k$, by running the inner iteration on $H_k$ shifted
    by a positive multiple of the identity. That buys the constant with a modification of the
    algorithm rather than an assumption on the model. \citet{GouldLucidiRomaToint1999} avoid the question, minimising over the
    Krylov space subject to the constraint rather than truncating within it.
    Proposition~\ref{prop: mqd} assumes neither convexity nor a shift,
    and~\eqref{eqn: half energy} holds for the algorithm as
    Section~\ref{subsec: subsolvers} runs it.
\end{remark}

\begin{remark}[What is still needed, and what is not]
\label{rem: mqd indefinite}
    Inequality~\eqref{eqn: galerkin ineq} is a property of the algorithm alone. It holds for any
    symmetric $H_k$, definite or not, and it refers to no radius rule. Everything
    model-dependent sits in the last step of Proposition~\ref{prop: mqd}, the passage from
    $\tfrac12 s_k^\top H_ks_k$ to $\tfrac12\bar\delta\norm{s_k}^2$.

    That step needs a lower bound on $\lambda_{\min}(H_k)$, which does not follow from directional
    consistency conditions of \citet{DennMore74} type, those controlling $H_k$ only along $s_k$.
    Where the bound fails, with $H_k$ indefinite and the inner iteration truncated on negative
    curvature, \eqref{eqn: half energy} is still true and vacuous, $s_k^\top H_ks_k$ being
    negative. A lower bound on the decrease then has to come from elsewhere.

    In that regime the radius rule can supply what the model cannot. Tie the radius to the
    criticality measure, $\Delta_k \leq c\norm{g_k}$ with $c = \overline{\mu}$ for capped $\Rgrad$
    at every iteration, or $c = \zeta$ for $\Rdfo$ at every iteration passing the criticality
    test. Then $\norm{s_k} = \Delta_k$ on the boundary gives $\norm{g_k} \geq \norm{s_k}/c$, and
    Cauchy decrease returns
    \begin{equation}
        m_k(0) - m_k(s_k) \;\geq\; \frac{\mykappa{mdc}}{2c}
            \min\Bigl\{\frac{1}{c\,\mykappa{umh}},\, 1\Bigr\}\norm{s_k}^2 ,
        \label{eqn: mqd from radius}
    \end{equation}
    in which $c$ is a user parameter rather than a property of the model, and no positive
    definiteness enters.

    Two limits of~\eqref{eqn: mqd from radius} are worth stating. It needs
    $\norm{s_k} = \Delta_k$, so it reaches only the iterations on which the constraint binds,
    and \PartII{} invokes~\eqref{eqn: mqd} under hypotheses that do not include activity. Where
    the constraint has released, the route closes. And it buys the constant with a parameter the
    user has to set small, which is the same choice Remark~\ref{rem: cauchy degeneration} shows
    turns the method into gradient descent. Running the inner iteration on a shifted operator, as
    in \citet{CurtisRobinsonRoyerWright2021}, pays neither price. Nor does the
    \texttt{EigenPoint} wrapper of Section~\ref{subsec: subsolvers}, which guarantees
    $\tfrac12\abs{\lambda_{\min}}\Delta_k^2$ outright, so that $\norm{s_k} \leq \Delta_k$
    gives~\eqref{eqn: mqd} with $\mykappa{mqd} = \tfrac12\abs{\lambda_{\min}(H_k)}$ wherever
    $H_k$ is indefinite. \citet{RebjockBoumal2025} take a third route, showing that truncation filters the small
    negative eigenvalues out before they can do damage.
\end{remark}

\section{The update mechanisms}
\label{app: update mechanisms}

The updates are grouped by what the radius is tied to. Which ratio selects the branch and
whether the multiplicative factor is fixed or adaptive vary within the groups, and
Table~\ref{tab: roster axes} records both.

\subsection{The criticality measure}
\label{app: measure}

The updates $\Rdfo$, $\Rgrad$, $\Rgrtr$ and $\RAdaptiveGrad$ tie the radius to a criticality
measure $\omega_k$, either the first-order measure $\norm{g_k}$ or a second-order measure
$\tau_k$. Writing an update as $\Rgrad(\omega, \overline{\mu})$ names the measure and the cap
as arguments, so that one rule covers both orders of criticality.
The second-order measure is
\begin{equation}
    \tau_k = \max\bigl\{ \norm{g_k},\, -\lambda_k^- \bigr\} ,
    \label{eqn: app criticality measure}
\end{equation}
with $\lambda_k^-$ the leftmost eigenvalue of the model Hessian $H_k$, taken as $0$ when that
eigenvalue is positive. This is the measure \PartII{} defines, and this study reads $\tau_k$
in that sense throughout; the first-order analysis does not use the symbol for a criticality
measure at all.
Two consequences are worth stating once. The eigenvalue is the \emph{model}'s and not the
objective's, so $\tau_k$ is a property of the pair $(x_k, H_k)$: a model that is positive
definite by construction reports $\tau_k = \norm{g_k}$ whatever $\nabla^2 f(x_k)$ does. And
where that eigenvalue is estimated rather than formed, the estimate is one-sided, so an
under-resolved estimate understates the negative curvature available and with it $\tau_k$.

\subsection{Standing conventions on the constants}
\label{app: constants}

Unless an update states otherwise, $0 < \gamma_1 \leq \gamma_2 < 1 < \gamma_3$ and
$0 \leq \eta \leq \eta_1 \leq \eta_2 < 1$ with $\eta_1 \neq 0$, where $\eta$ is the
acceptance threshold of Step~3 and $\eta_1$ and $\eta_2$ are the update thresholds. Caps
satisfy $\overline{\mu} \in (0, +\infty]$ with $\mu_0 \leq \overline{\mu}$. Two updates
depart from this. \Rdeltastep\ admits $\gamma_2 \leq 1$, and \RAdaptstep\ requires
$\eta < \eta_1$ strictly.

\subsection{Anchored to the previous radius}
\label{app: anchored delta}

The radius carries no information about the problem beyond its own history and the ratio.

      \begin{update}{\Rdelta}[R-delta: multiplicative factor on $\Delta$] \label{update:multiplicative factor}
            \begin{equation}
                \Delta_{k+1} \in
                \begin{cases}
                    [\gamma_1 \Delta_k,\, \gamma_2 \Delta_k), & \text{if } \rho_k < \eta_1 \quad \text{(unsuccessful)}, \\
                    [\gamma_2 \Delta_k,\, \Delta_k), & \text{if } \eta_1 \leq \rho_k < \eta_2 \quad \text{(successful)}, \\
                    [\Delta_k,\, \gamma_3 \Delta_k], & \text{if } \rho_k \geq \eta_2 \quad \text{(very successful)}.
                \end{cases}
                \label{eqn:appendix R-delta}
            \end{equation}
        \end{update}

        \begin{update}{\Rdeltastep}           
            \[
            \Delta_{k+1} \in 
            \begin{cases}
                [\gamma_1 \|s_k\|, \gamma_2 \|s_k\|) & \text{if } \rho_k < \eta_1, \\
                [\gamma_2 \Delta_k, \Delta_k) & \text{if } \rho_k \in [\eta_1, \eta_2), \\
                [\Delta_k, \gamma_3 \Delta_k] & \text{if } \rho_k \geq \eta_2,
            \end{cases}
            \]
            where \( 0 < \gamma_1 \leq \gamma_2 \leq 1 < \gamma_3 \). 
            \label{eqn:classical rule with aggressive decrease}
        \end{update}  

\subsection{Anchored to the step}
\label{app: anchored step}

The radius is rebuilt from the step actually taken, so a step interior to the region
contracts it even after a successful iteration.

        \begin{update}{\Rstep}[R-step: step-dependent update]
            Let \( 0 < \gamma_1 \leq \gamma_2 < 1 < \gamma_3 \) and
            \( \eta \leq \eta_1 \leq \eta_2 < 1 \) with $\eta_1 \neq 0$. Then
            \begin{equation}
                \Delta_{k+1} \in
                \begin{cases}
                    [\gamma_1 \norm{s_k},\, \gamma_2 \norm{s_k}), & \text{if } \rho_k < \eta_1, \\
                    [\gamma_2 \norm{s_k},\, \norm{s_k}), & \text{if } \eta_1 \leq \rho_k < \eta_2, \\
                    [\norm{s_k},\, \gamma_3 \norm{s_k}], & \text{if } \rho_k \geq \eta_2 .
                \end{cases}
                \label{update: step-dependent}
            \end{equation}
        \label{update:step-dependent update}
        \end{update}
        The expansion branch is capped at $\gamma_3\norm{s_k}$, the step-based counterpart of the
        bounded expansion allowed by Remark~7 of \PartI{}. Every radius this rule returns is one the
        uncapped rule may return, so the results stated there apply to it unchanged.

        \begin{update}{\RAdaptstep}
                \[
                     \Delta_{k+1} = R_{\eta_1}(\rho_k) \|s_k\|,
                \]
                where \( R_{\eta_1} : \mathbb{R} \to \mathbb{R}_+ \) is a \textbf{non-decreasing} function satisfying, with \(0 \leq  \eta  < \eta_1 < 1 \) and \( 0 < \gamma_1 \leq \gamma_2 < 1 < \gamma_3 < M\),
                \begin{align*}
                    & \lim_{t \to -\infty} R_{\eta_1}(t) = \gamma_1 ; \\
                    & \forall \, t < \eta,\,  R_{\eta_1}(t) \leq \gamma_2 ; \\
                    & R_{\eta_1}(\eta_1) = \gamma_3 ; \\
                    & \lim_{t \to +\infty} R_{\eta_1}(t) = M.
                \end{align*}
                
                \label{update: adaptive step-size dependent update}
            \end{update}

\subsection{Anchored to the criticality measure}
\label{app: anchored measure}

The radius is a multiple of $\omega_k$, or is gated by it, so it vanishes with the measure
and the multiplier $\mu_k$ carries the iteration history.

        \begin{update}{$\Rdfo(\omega)$}
        \label{update: rdfo generic}
            Let $\eta \leq \eta_1 < 1$, $0 < \gamma_1 \leq \gamma_2 < 1 < \gamma_3$ and $\zeta > 0$. Then
            \begin{equation}
                \Delta_{k+1} \in
                \begin{cases}
                    [\gamma_1 \Delta_k,\ \gamma_2 \Delta_k], & \rho_k < \eta_1, \\[2pt]
                    \bigl\{\gamma_2 \Delta_k\bigr\}, & \rho_k \geq \eta_1
                        \text{ and } \Delta_k > \zeta\,\omega_k, \\[2pt]
                    [\Delta_k,\ \gamma_3 \Delta_k], & \rho_k \geq \eta_1
                        \text{ and } \Delta_k \leq \zeta\,\omega_k.
                \end{cases}
                \label{eqn: rdfo rule}
            \end{equation}
        \end{update}

        \begin{update}{$\Rgrad(\omega, \overline{\mu})$}[Expansion capped at $\overline{\mu}$]
        \label{update: rgrad generic}
            Let $0 < \gamma_1 \leq \gamma_2 < 1 < \gamma_3$ and
            $\eta \leq \eta_1 \leq \eta_2 < 1$. The radius is $\Delta_k = \mu_k\,\omega_k$ with
            $\mu_0 \leq \overline{\mu}$ and
            \begin{equation}
                \mu_{k+1} \in
                \begin{cases}
                    [\gamma_1 \mu_k,\ \gamma_2 \mu_k], & \rho_k < \eta_1, \\[2pt]
                    [\gamma_2 \mu_k,\ \mu_k), & \eta_1 \leq \rho_k < \eta_2, \\[2pt]
                    \bigl\{\min\{\gamma_3 \mu_k,\ \overline{\mu}\}\bigr\},
                        & \rho_k \geq \eta_2 \text{ and } \norm{s_k} > \tfrac12 \Delta_k, \\[2pt]
                    \{\mu_k\}, & \text{otherwise}.
                \end{cases}
                \label{eqn: rgrad rule}
            \end{equation}
        \end{update}

        \begin{update}{$\Rgrad(\omega)$}[Uncapped variant]
        \label{update: rgrad uncapped rule}
            Update~\ref{update: rgrad generic} with $\overline{\mu} = +\infty$, so that the expansion
            branch reads $\mu_{k+1} = \gamma_3\mu_k$ and the last branch
            $\mu_{k+1} = \mu_k$.
        \end{update}

            \begin{update}{$\RAdaptiveGrad(\omega)$}[radius tied to the measure, multiplier driven by the adaptive factor]
            \label{update: radaptgrad}
                Let $0 \leq \eta < \eta_1 \leq \eta_2 < 1$, $0 < \gamma_1 \leq \gamma_2 < 1 < \gamma_3 < M$, and let $R_{\eta_1}$ be the non-decreasing function of
                Update~\ref{update: adaptive step-size dependent update}. The radius is
                $\Delta_k = \mu_k\,\omega_k$ and
                \begin{equation}
                    \mu_{k+1} =
                    \begin{cases}
                        \mu_k, & \rho_k \geq \eta_2 \text{ and } \norm{s_k} \leq \tfrac12 \Delta_k, \\[2pt]
                        R_{\eta_1}(\rho_k)\,\mu_k, & \text{otherwise},
                    \end{cases}
                    \label{eqn: radaptgrad rule}
                \end{equation}
                the measure $\omega_k$ being as in Update~\ref{update: rgrad generic}.
            \end{update}

            \begin{update}{$\RAdaptiveGrad(\omega, \overline{\mu})$}[Capped variant]
            \label{update: radaptgrad capped}
                Update~\ref{update: radaptgrad} with $\mu_0 \leq \overline{\mu}$ and the second branch
                replaced by $\mu_{k+1} = \min\{R_{\eta_1}(\rho_k)\,\mu_k,\ \overline{\mu}\}$.
            \end{update}

\subsection{Driven by the retrospective ratio}
\label{app: retrospective}

The branch is selected by $\widetilde{\rho}_{k+1} = \frac{f(x_k) - f(x_{k+1})}{m_{k+1}(-s_k) - m_{k+1}(0)}$, which uses the model built at the new iterate.

        \begin{update}{\Rrtr}[R-RTR: retrospective update] \label{update:retrospective}
            At iteration $k \geq 0$, with $ \widetilde{\rho}_{k+1}
              = \frac{f(x_k) - f(x_{k+1})}{m_{k+1}(-s_k) - m_{k+1}(0)}$,  $k \geq 0$, 
            and retrospective thresholds
            $0 < \widetilde\eta_1 \leq \widetilde\eta_2 < 1$:
            if $\rho_k < \eta$ (step rejected), set
            \[
                \Delta_{k+1} =
                \begin{cases}
                    \min\!\big\{\gamma_2 \norm{s_k},\, \gamma_1\,\Delta_k\big\},
                        & \rho_k < 0,\\[4pt]
                    \gamma_2 \norm{s_k}, & \text{otherwise},
                \end{cases}
            \]
            if $\rho_k \geq \eta$ (step accepted), set
            \[
                \Delta_{k+1} =
                \begin{cases}
                    \max\{\gamma_3 \norm{s_k},\; \Delta_k\},
                        & \widetilde{\rho}_{k+1} \geq \widetilde\eta_2,\\[4pt]
                    \Delta_k,
                        & \widetilde{\rho}_{k+1} \in [\widetilde\eta_1, \widetilde\eta_2),\\[4pt]
                    \gamma_2 \norm{s_k},
                        & \widetilde{\rho}_{k+1} \in [0, \widetilde\eta_1),\\[4pt]
                    \min\!\big\{\gamma_2\norm{s_k},\, \gamma_1\,\Delta_k\big\},
                        & \widetilde{\rho}_{k+1} < 0.
                \end{cases}
            \]
        \end{update}

        \begin{update}{$\Rgrtr(\omega, \overline{\mu})$}[radius tied to the measure, branch selected by the
        retrospective ratio]
        \label{update: grtr}
            Let $0 < \gamma_1 \leq \gamma_2 < 1 < \gamma_3$,
            $\eta \leq \eta_1 \leq \eta_2 < 1$ and
            $0 < \widetilde\eta_1 \leq \widetilde\eta_2 < 1$. The radius is
            $\Delta_k = \mu_k\,\omega_k$ with $\mu_0 \leq \overline{\mu}$ and
            \begin{equation}
                \mu_{k+1} \in
                \begin{cases}
                    [\gamma_1\mu_k,\ \gamma_2\mu_k], & \rho_k < \eta_1, \\[2pt]
                    [\gamma_1\mu_k,\ \gamma_2\mu_k], & \rho_k \geq \eta_1
                        \text{ and } \widetilde\rho_{k+1} < \widetilde\eta_1, \\[2pt]
                    [\gamma_2\mu_k,\ \mu_k), & \rho_k \geq \eta_1
                        \text{ and } \widetilde\rho_{k+1} \in [\widetilde\eta_1, \widetilde\eta_2), \\[2pt]
                    \bigl\{\min\{\gamma_3\mu_k,\ \overline{\mu}\}\bigr\},
                        & \rho_k \geq \eta_1,\ \widetilde\rho_{k+1} \geq \widetilde\eta_2
                        \text{ and } \norm{s_k} > \tfrac12\Delta_k, \\[2pt]
                    \{ \mu_k \}, & \text{otherwise},
                \end{cases}
                \label{eqn: grtr rule}
            \end{equation}
            the measure $\omega_k$ and the cap $\overline{\mu} \in (0, +\infty]$ being as in
            Update~\ref{update: rgrad generic}.
        \end{update}

\section{What the anchors do to the radius}
\label{app: p1 asymptotics}

The sections above separate the mechanisms by what the radius is anchored to. On one
problem that axis is visible directly. The problem is $f(x) = \rme^{-x}$ from $x_0 = 0$, the P1
of Section~\ref{subsec: P1}, where the Newton step is $1$ everywhere, every iteration is very
successful, and nothing but the rule moves the radius. The figures below are drawn on it over a
few hundred iterations; the values Section~\ref{subsec: P1} reports come from a longer run of the
same configuration, so these are read for the shape and that section for the numbers.
Figure~\ref{fig: p1 dichotomy} shows which anchors drive the radius to zero, and
Figure~\ref{fig: p1 rate} the rate at which they do.

\begin{figure}[htbp]
    \centering
    \begin{subfigure}[b]{\linewidth}
        \centering
        \includegraphics[width=0.86\linewidth]{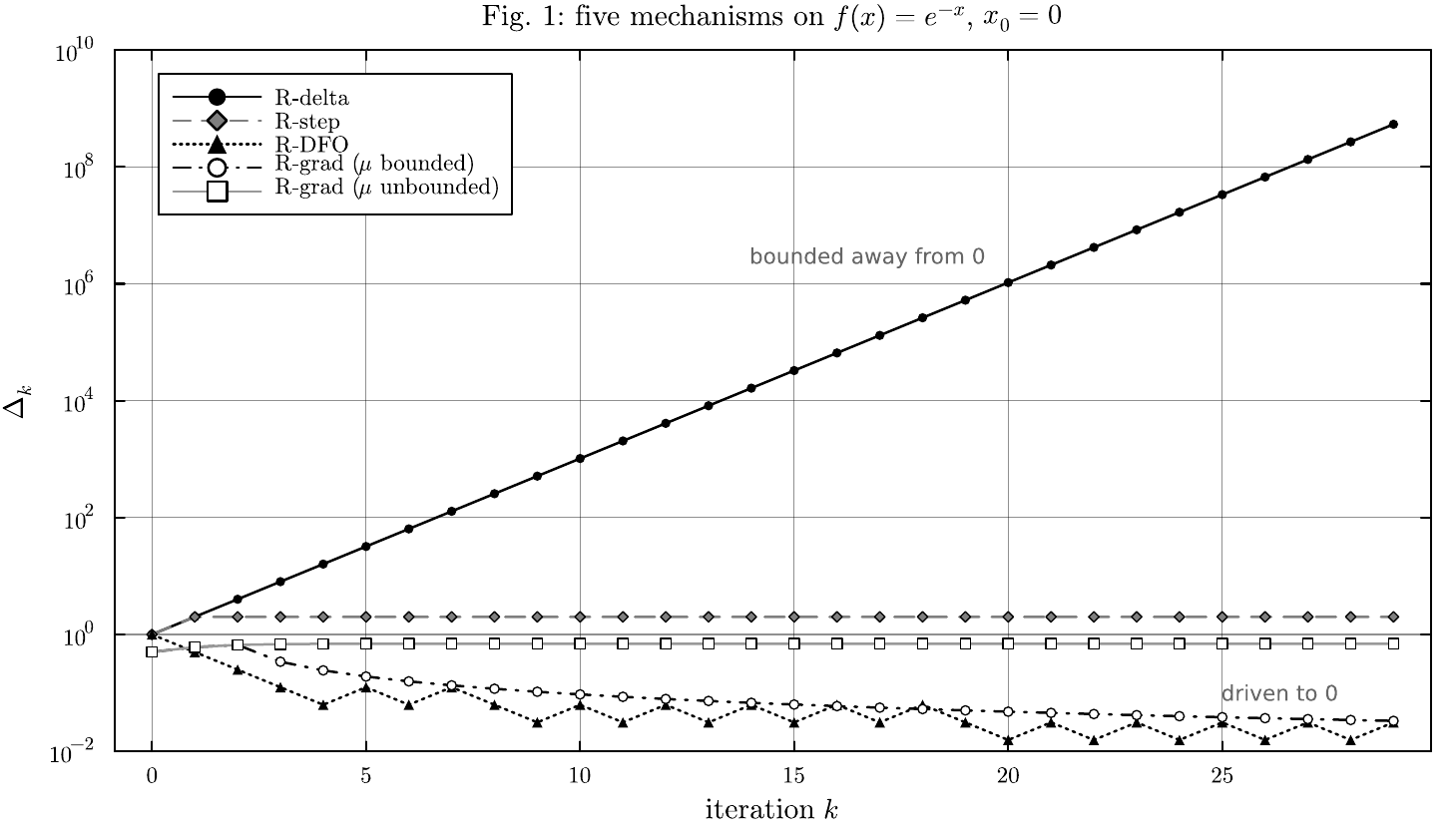}
        \caption{all five mechanisms}
        \label{fig: p1 all mechanisms}
    \end{subfigure}

    \medskip
    \begin{subfigure}[b]{\linewidth}
        \centering
        \includegraphics[width=0.86\linewidth]{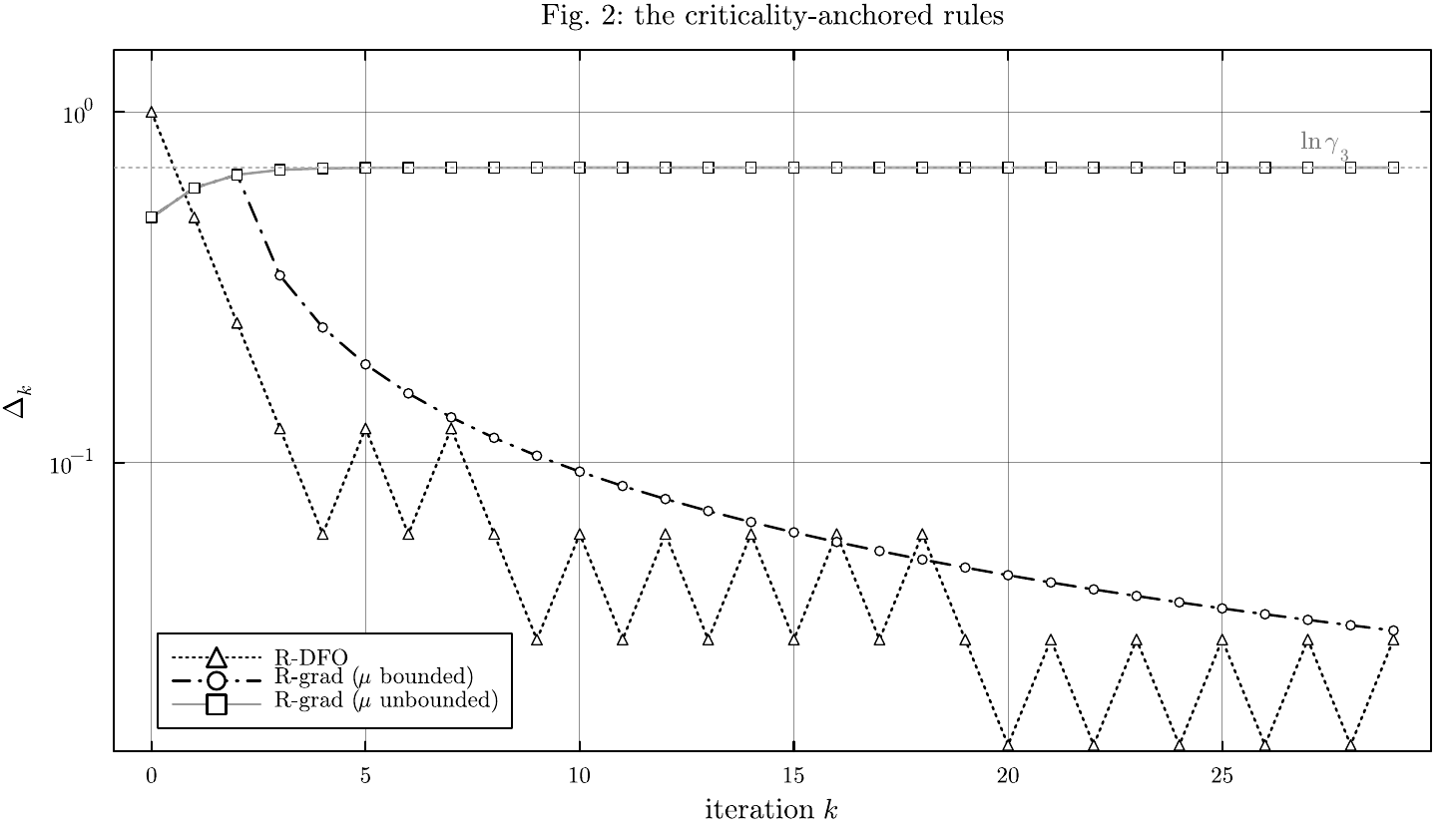}
        \caption{the criticality-anchored three}
        \label{fig: p1 anchored}
    \end{subfigure}
    \caption{The anchoring dichotomy on $f(x) = \rme^{-x}$. Above, $\Rdelta$ grows
    geometrically and $\Rstep$ and uncapped $\Rgrad$ settle at a constant, while $\Rdfo$ and
    capped $\Rgrad$ are driven to zero: the radius vanishes under a criticality anchor and not
    otherwise.
    Below, the three criticality-anchored rules on their own. Uncapped $\Rgrad$ flattens at
    $\ln\gthree$, which is the whole of the difference the cap makes, and the two capped rules
    fall. $\Rdfo$ falls in a sawtooth because its criticality test fires only when
    $\Delta_k > \zeta\omega_k$.}
    \label{fig: p1 dichotomy}
\end{figure}

\begin{figure}[htbp]
    \centering
    \begin{subfigure}[b]{\linewidth}
        \centering
        \includegraphics[width=0.86\linewidth]{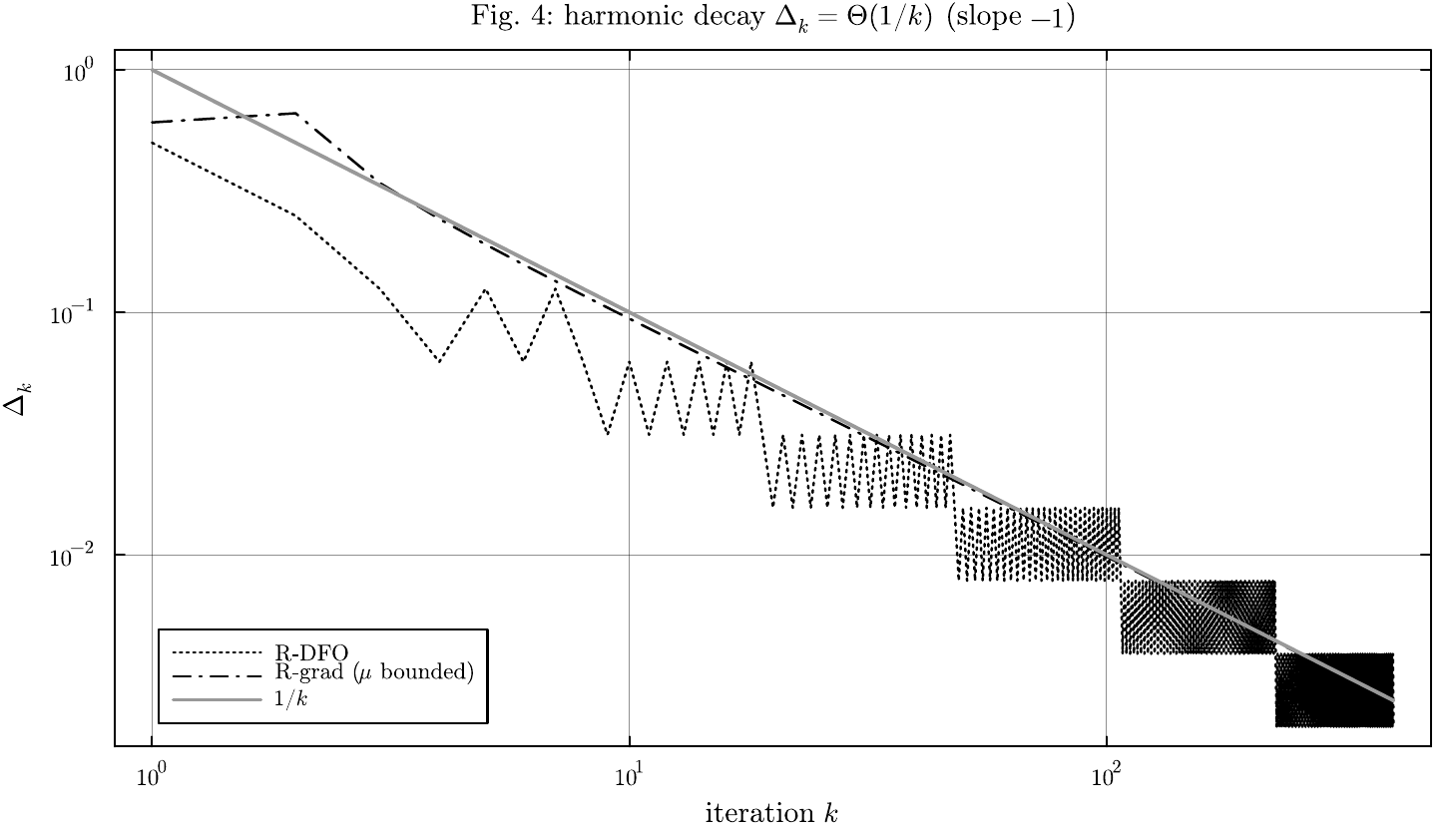}
        \caption{$\Delta_k$ against $1/k$}
        \label{fig: p1 harmonic}
    \end{subfigure}

    \medskip
    \begin{subfigure}[b]{\linewidth}
        \centering
        \includegraphics[width=0.86\linewidth]{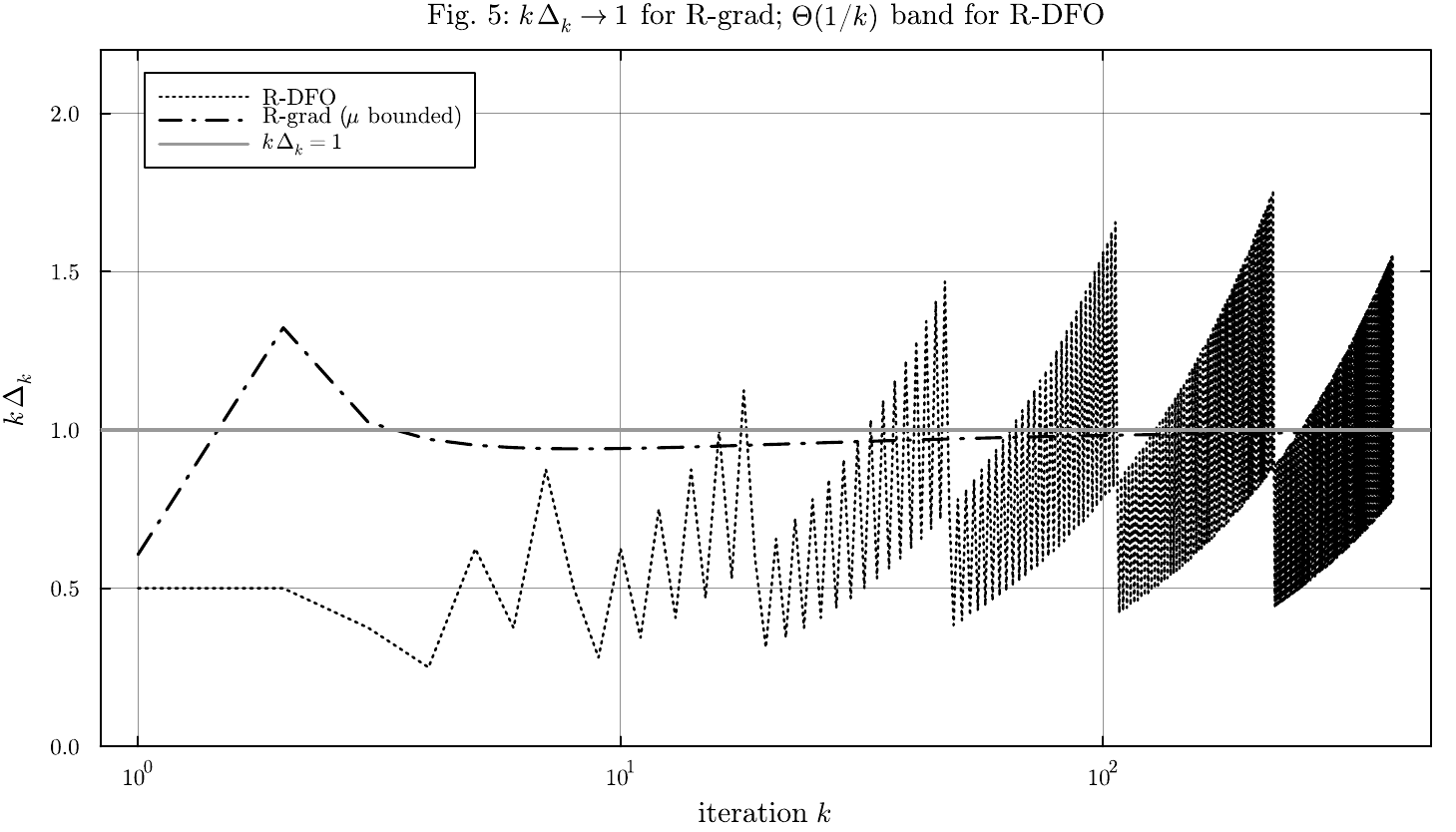}
        \caption{$k\Delta_k$}
        \label{fig: p1 k delta}
    \end{subfigure}
    \caption{The rate at which the anchored radius falls. Above, on log axes, against a
    reference line of slope $-1$. Below, the same runs as $k\Delta_k$, where a harmonic decay is
    a horizontal line. Capped $\Rgrad$ approaches $k\Delta_k = 1$, which is the constant
    \PartII{} derives and Section~\ref{subsec: P1} measures to five figures. $\Rdfo$ stays inside
    a band rather than converging, so it is $\Theta(1/k)$ without a limiting constant, and the
    distinction is what separates a rule that contracts every iteration from one that contracts
    only when its test fires.}
    \label{fig: p1 rate}
\end{figure}

\section{Cost and outcome tables for the second-order campaign}
\label{app: so tables}

Source archive \texttt{exp\_2026-08-31\_10-38-22\_second\_order\_profiles}. Every total is
taken over every problem of the panel, met or not, with no filtering on success and no mean
over survivors, so the columns are comparable across rows without a survivorship correction.
Wall-clock time is censored by the $120$ second limit and is reported for completeness only.
Tables~\ref{tab: so cost p1}, \ref{tab: so cost p2} and~\ref{tab: so cost p3} give the cost
on Panels P1, P2 and P3, and Table~\ref{tab: so status p3} how each run of Panel P3 ended.

\begin{table}[htbp]
    \centering
    \small
    \begin{tabular}{@{}llrrrrrr@{}}
        \toprule
        Rule & arm & certified & $f$ evals & $g$ evals & H-prods & H evals & wall (s) \\
        \midrule
        \Rdelta   & 2nd & $167/196$ & $66\,616$ & $56\,130$ & $1\,406\,250$ & $53\,089$ & $1173.4$ \\
        \Rstep    & 2nd & $136/196$ & $416\,228$ & $409\,455$ & $1\,839\,674$ & $397\,158$ & $1348.0$ \\
        \Rrtr     & 2nd & $167/196$ & $62\,811$ & $54\,071$ & $1\,224\,300$ & $44\,264$ & $1171.3$ \\
        \Rdfo     & 2nd & $160/196$ & $87\,969$ & $73\,252$ & $1\,933\,452$ & $69\,421$ & $1260.6$ \\
        \Rgrad    & 2nd & $161/196$ & $80\,423$ & $68\,113$ & $698\,500$ & $64\,448$ & $1016.5$ \\
        \Rgrtr    & 2nd & $160/196$ & $75\,434$ & $63\,999$ & $764\,438$ & $59\,946$ & $1131.6$ \\
        \bottomrule
    \end{tabular}
    \caption{Panel P1, the second-order arm alone on all $196$ problems, scored on the
    second-order criterion. No first-order row appears, so nothing in this panel compares the
    two arms.}
    \label{tab: so cost p1}
\end{table}

\begin{table}[htbp]
    \centering
    \small
    \begin{tabular}{@{}llrrrrrr@{}}
        \toprule
        Rule & arm & met & $f$ evals & $g$ evals & H-prods & H evals & wall (s) \\
        \midrule
        \Rdelta   & 1st & $165/196$ & $81\,451$ & $66\,824$ & $823\,786$ & $0$ & $1075.4$ \\
                  & 2nd & $167/196$ & $66\,613$ & $56\,126$ & $1\,405\,921$ & $53\,085$ & $1173.9$ \\
        \Rstep    & 1st & $138/196$ & $407\,429$ & $401\,848$ & $1\,692\,614$ & $0$ & $1248.4$ \\
                  & 2nd & $136/196$ & $416\,230$ & $409\,456$ & $1\,839\,677$ & $397\,159$ & $1347.9$ \\
        \Rrtr     & 1st & $169/196$ & $44\,398$ & $38\,979$ & $641\,934$ & $0$ & $1084.9$ \\
                  & 2nd & $168/196$ & $62\,821$ & $54\,086$ & $1\,224\,200$ & $44\,279$ & $1170.5$ \\
        \Rdfo     & 1st & $162/196$ & $106\,097$ & $90\,697$ & $853\,974$ & $0$ & $1205.3$ \\
                  & 2nd & $160/196$ & $87\,940$ & $73\,222$ & $1\,933\,404$ & $69\,392$ & $1257.6$ \\
        \Rgrad    & 1st & $163/196$ & $60\,884$ & $52\,660$ & $664\,489$ & $0$ & $1065.8$ \\
                  & 2nd & $161/196$ & $80\,232$ & $67\,965$ & $697\,671$ & $64\,300$ & $1016.7$ \\
        \Rgrtr    & 1st & $160/196$ & $74\,907$ & $67\,405$ & $705\,225$ & $0$ & $1076.5$ \\
                  & 2nd & $160/196$ & $75\,262$ & $63\,849$ & $763\,614$ & $59\,796$ & $1132.3$ \\
        \bottomrule
    \end{tabular}
    \caption{Panel P2, both arms on all $196$ problems, scored on the first-order criterion.
    Every run halts at its first first-order iterate, so both arms perform the same task and
    the difference between two rows of a rule is what the second-order machinery costs. The
    first-order arm records no Hessian evaluation at all.}
    \label{tab: so cost p2}
\end{table}

\begin{table}[htbp]
    \centering
    \small
    \begin{tabular}{@{}llrrrrrr@{}}
        \toprule
        Rule & arm & certified & $f$ evals & $g$ evals & H-prods & H evals & wall (s) \\
        \midrule
        \Rdelta   & 1st & $0/185$ & $66\,328$ & $55\,736$ & $774\,382$ & $0$ & $1049.5$ \\
                  & 2nd & $158/185$ & $63\,465$ & $53\,089$ & $1\,266\,347$ & $53\,089$ & $1150.9$ \\
        \Rstep    & 1st & $0/185$ & $389\,344$ & $386\,313$ & $1\,623\,356$ & $0$ & $1222.7$ \\
                  & 2nd & $126/185$ & $401\,336$ & $397\,158$ & $1\,622\,548$ & $397\,158$ & $1297.4$ \\
        \Rrtr     & 1st & $0/185$ & $30\,738$ & $28\,601$ & $588\,252$ & $0$ & $1050.8$ \\
                  & 2nd & $157/185$ & $49\,809$ & $44\,264$ & $1\,050\,838$ & $44\,264$ & $1113.1$ \\
        \Rdfo     & 1st & $0/185$ & $88\,071$ & $77\,498$ & $766\,452$ & $0$ & $1069.4$ \\
                  & 2nd & $151/185$ & $82\,684$ & $69\,421$ & $1\,757\,668$ & $69\,421$ & $1131.5$ \\
        \Rgrad    & 1st & $0/185$ & $46\,598$ & $42\,191$ & $620\,303$ & $0$ & $1041.2$ \\
                  & 2nd & $151/185$ & $76\,181$ & $64\,448$ & $558\,919$ & $64\,448$ & $989.5$ \\
        \Rgrtr    & 1st & $0/185$ & $59\,942$ & $55\,543$ & $648\,661$ & $0$ & $1046.8$ \\
                  & 2nd & $150/185$ & $71\,168$ & $59\,946$ & $600\,147$ & $59\,946$ & $1104.5$ \\
        \bottomrule
    \end{tabular}
    \caption{Panel P3, both arms on the $185$ problems with $n \leq 200$, scored on the
    second-order criterion. The zeros in the first-order column are structural and
    Section~\ref{subsec: so limits} says why. Two columns coincide exactly in every
    second-order row, the gradient and Hessian evaluations, because the dense curvature
    estimate is formed once per iterate below $n = 200$. The certificate therefore costs one
    Hessian per gradient, whatever the rule, and the six rows differ only in how many
    gradients the rule needs.}
    \label{tab: so cost p3}
\end{table}

\begin{table}[htbp]
    \centering
    \small
    \begin{tabular}{@{}llrrrrrr@{}}
        \toprule
        & & \multicolumn{2}{c}{criterion met} & \multicolumn{4}{c}{run ended on} \\
        \cmidrule(lr){3-4}\cmidrule(l){5-8}
        Rule & arm & 1st-order & 2nd-order & iteration cap & time cap & stall & exception \\
        \midrule
        \Rdelta   & 1st & $155$ & $0$ & $1$ & $8$ & $21$ & $0$ \\
                  & 2nd & $0$ & $158$ & $2$ & $9$ & $16$ & $0$ \\
        \Rstep    & 1st & $128$ & $0$ & $31$ & $9$ & $17$ & $0$ \\
                  & 2nd & $0$ & $126$ & $33$ & $9$ & $17$ & $0$ \\
        \Rrtr     & 1st & $159$ & $0$ & $0$ & $8$ & $18$ & $0$ \\
                  & 2nd & $0$ & $157$ & $1$ & $8$ & $19$ & $0$ \\
        \Rdfo     & 1st & $153$ & $0$ & $4$ & $8$ & $20$ & $0$ \\
                  & 2nd & $0$ & $151$ & $4$ & $8$ & $22$ & $0$ \\
        \Rgrad    & 1st & $153$ & $0$ & $0$ & $8$ & $24$ & $0$ \\
                  & 2nd & $0$ & $151$ & $3$ & $7$ & $23$ & $1$ \\
        \Rgrtr    & 1st & $150$ & $0$ & $1$ & $7$ & $27$ & $0$ \\
                  & 2nd & $0$ & $150$ & $2$ & $8$ & $24$ & $1$ \\
        \bottomrule
    \end{tabular}
    \caption{How the $185$ runs of each configuration of Panel P3 ended, with
    $\varepsilon_g = 10^{-5}$ and $\varepsilon_H = 10^{-6}$. Two further columns of the
    archive table are omitted because they are zero in every row, a user-requested stop and a
    missing record. The first-order arm reaches a first-order point on $128$ to $159$ problems
    and certifies none of them, which is the distinction
    Section~\ref{subsec: so limits} turns on. $\Rstep$ is the only rule for which the
    iteration cap is a common outcome rather than a rare one.}
    \label{tab: so status p3}
\end{table}

\end{document}